\documentclass{amsart}
\usepackage[utf8]{inputenc}

\usepackage{lineno}
\usepackage{paralist}
\usepackage{extpfeil}

\usepackage{amssymb}
\usepackage{amsmath}
\usepackage{amsfonts}
\usepackage{tikz-cd}
\usetikzlibrary{arrows}
\usepackage{relsize}

\usepackage{centernot}
\usepackage{mathtools}

\makeatletter
\newcommand{\xMapsto}[2][]{\ext@arrow 0599{\Mapstofill@}{#1}{#2}}
\def\Mapstofill@{\arrowfill@{\Mapstochar\Relbar}\Relbar\Rightarrow}
\makeatother

\tikzset{
curvarr/.style={
  to path={ -- ([xshift=2ex]\tikztostart.east)
    |- (#1) [near end]\tikztonodes
    -| ([xshift=-2ex]\tikztotarget.west)
    -- (\tikztotarget)}
  }
}

\tikzset{
  curvedlink/.style={
    to path={
      let \p1=(\tikztostart.east), \p2=(\tikztotarget.west),
      \n1= {abs(\y2-\y1)/4} in
      (\p1) arc(90:-90:\n1) -- ([yshift=2*\n1]\p2) arc (90:270:\n1)
    },
  }
}

\usepackage{amsthm}
\usepackage{enumitem}
\usepackage{mathrsfs}
\usetikzlibrary{decorations.pathmorphing}
\usepackage[utf8]{inputenc}
\usepackage[T1]{fontenc}
\usepackage{indentfirst}
\usepackage{xcolor}
\usepackage{mathtools}
\usepackage{lineno}
\usepackage{hyperref}
\hypersetup{
    colorlinks,
    citecolor=black,
    filecolor=black,
    linkcolor=black,
    urlcolor=black
}
\usepackage{thmtools}

\numberwithin{equation}{section}
\newtheorem{theorem}[equation]{Theorem}
\newtheorem{thm}[theorem]{Theorem}
\newtheorem{corollary}[theorem]{Corollary}
\newtheorem{proposition}[theorem]{Proposition}

\newtheorem{lemma}[theorem]{Lemma}
\newtheorem{cor}[theorem]{Corollary}
\newtheorem{prop}[theorem]{Proposition}
\newtheorem{lem}[theorem]{Lemma}
\newcommand{\myendsymbol}{\ensuremath{\diamondsuit}}

\declaretheorem[
  style=definition,
  title=Definition,
  qed={$\myendsymbol$},
  sharenumber=thm,
]{define}

\declaretheorem[
  style=definition,
  title=Definition,
  qed={$\myendsymbol$},
  sharenumber=thm,
]{dfn}

\declaretheorem[
  style=definition,
  title=Example,
  qed={$\myendsymbol$},
  refname={example,examples},
  Refname={Example,Examples},
  sharenumber=thm,
]{example}

\declaretheorem[
  style=definition,
  title=Example,
  qed={$\myendsymbol$},
  refname={example,examples},
  Refname={Example,Examples},
  sharenumber=thm,
]{exa}

\declaretheorem[
  style=definition,
  title=Convention,
  qed={$\myendsymbol$},
  sharenumber=thm,
]{conv}

\declaretheorem[
  style=definition,
  title=Notation,
  qed={$\myendsymbol$},
  sharenumber=thm,
]{ntn}

\declaretheorem[
  style=definition,
  title=Remark,
  qed={$\myendsymbol$},
  sharenumber=thm,
]{remark}

\declaretheorem[
  style=definition,
  title=Remark,
  qed={$\myendsymbol$},
  sharenumber=thm,
]{rmk}

\declaretheorem[
  style=definition,
  title=Question,
  qed={$\myendsymbol$},
  sharenumber=thm,
]{question}

\declaretheorem[
  style=definition,
  title=Problem,
  qed={$\myendsymbol$},
  sharenumber=thm,
]{problem}

\newcommand{\calB}{\mathcal{B}}
\newcommand{\calC}{\mathcal{C}}

\newcommand{\calF}{\mathcal{F}}

\newcommand{\calH}{\mathcal{H}}
\newcommand{\calI}{\mathcal{I}}

\newcommand{\calT}{\mathcal{T}}

\newcommand{\scrM}{\mathscr{M}}
\newcommand{\scrL}{\mathscr{L}}
\newcommand{\scrF}{\mathscr{F}}
\newcommand{\scrU}{\mathscr{U}}
\newcommand{\scrG}{\mathscr{G}}
\newcommand{\scrO}{\mathscr{O}}

\newcommand{\scrP}{\mathscr{P}}

\newcommand{\frakm}{\mathfrak{m}}

\newcommand{\frakp}{\mathfrak{p}}

\newcommand{\frakx}{\mathfrak{x}}

\DeclareMathOperator{\MaxR}{MaxR}
\DeclareMathOperator{\lct}{lct}

\DeclareMathOperator{\Hom}{Hom}

\DeclareMathOperator{\Var}{Var}

\DeclareMathOperator{\Spec}{Spec}

\DeclareMathOperator{\reg}{reg}
\DeclareMathOperator{\rank}{rank}

\DeclareMathOperator{\codim}{codim}

\DeclareMathOperator{\Sing}{Sing}

\DeclareMathOperator{\Ind}{Ind}
\DeclareMathOperator{\Feyn}{Feyn}
\DeclareMathOperator{\MSP}{MSP}
\DeclareMathOperator{\FMP}{FMP}
\DeclareMathOperator{\Matroidal}{Mtrdl}

\newcommand{\boldb}{{\mathbf{b}}}

\newcommand{\boldu}{{\mathbf{u}}}
\newcommand{\boldv}{{\mathbf{v}}}

\newcommand{\bsx}{{\boldsymbol x}}

\newcommand{\bsalpha}{{\boldsymbol \alpha}}

\newlength{\myl}
\newcommand{\de}{{\mathrm d}}
\newcommand{\eps}{\varepsilon}

\newcommand{\into}{\hookrightarrow}

\renewcommand{\to}{\longrightarrow}
\newcommand{\ideal}[1]{{\langle#1\rangle}}
\newcommand{\bracket}[2]{#1^{[p^{#2}]}}

\newcommand{\minus}{\smallsetminus}

\newcommand{\matB}{{\mathsf B}}
\newcommand{\matD}{{\mathsf D}}
\newcommand{\matN}{{\mathsf N}}
\newcommand{\matF}{{\mathsf F}}
\newcommand{\matM}{{\mathsf M}}
\newcommand{\matQ}{{\mathsf Q}}
\newcommand{\matU}{{\mathsf U}}
\newcommand{\matL}{{\mathsf L}}

\newcommand{\CC}{\mathbb{C}}

\newcommand{\FF}{\mathbb{F}}

\newcommand{\KK}{\mathbb{K}}

\newcommand{\NN}{\mathbb{N}}
\newcommand{\PP}{\mathbb{P}}
\newcommand{\QQ}{\mathbb{Q}}
\newcommand{\RR}{\mathbb{R}}

\newcommand{\ZZ}{\mathbb{Z}}

\newcommand{\VExt}{{V_{\textrm{Ext}}}}
\newcommand{\massm}{{\boldsymbol m}}

\DeclareMathOperator{\iso}{\simeq}

\newcommand\uli[1]{\ensuremath{\clubsuit}{UW:\tiny{#1}}\ensuremath{\clubsuit}}

\newcommand\dan[1]{\ensuremath{\spadesuit}{DB:\tiny{#1}}}

\usepackage [english]{babel}
\usepackage [autostyle, english = american]{csquotes}
\MakeOuterQuote{"}

\title[]{Matroidal polynomials, their singularities, and applications to Feynman diagrams}

\author{Daniel Bath}
\address{Daniel Bath\\
KU Leuven\\ 
Departement Wiskunde,\\ 
Celestijnenlaan 200B,
Leuven 3001,
Belgium.}
\email{dan.bath@kuleuven.be}

\author{Uli Walther}
\address{ Uli~Walther\\
  Purdue University\\
  Dept.\ of Mathematics\\
  150 N.\ University St.\\
  West Lafayette, IN 47907\\ USA}
\email{walther@purdue.edu}

\thanks{During the preparation of this manuscript, 
DB was supported by FWO Grant \#12E9623N.
UW was supported in part by NSF Grant DMS-2100288 and by Simons Foundation Collaboration Grant for Mathematicians
\#580839. 
UW was also supported by the National Science Foundation under Grant DMS-1928930 and by the Alfred P.\ Sloan Foundation under grant G-2021-16778, while in residence at the Simons Laufer Mathematical Sciences Institute (formerly MSRI) in Berkeley, California, during the Spring 2024 semester.}

\dedicatory{Dedicated to Brandy Ambrose--Bath, beloved in perpetuity.}

\begin{document}
\sloppy

\begin{abstract}
Given a matroid or flag of matroids we introduce several broad classes of polynomials satisfying Deletion-Contraction identities, and study their singularities. 

There are three main families of polynomials captured by our approach: matroidal polynomials on a matroid (including matroid basis polynomials, configuration polynomials, Tutte polynomials); flag matroidal polynomials on a flag matroid; and Feynman integrands. The last class includes under general kinematics the inhomogeneous Feynman diagram polynomials which naturally arise in the Lee--Pomeransky form of the Feynman integral attached to a Feynman diagram.

Assuming that the primary underlying matroid is connected and of positive rank (and in the flag case, has rank at least two),  we show: a) in positive characteristic, homogeneous  matroidal  polynomials are strongly $F$-regular;  b) over an algebraically  closed field  of characteristic zero, the associated jet schemes of (flag) matroidal polynomials as well as those of Feynman integrands are irreducible. Consequently, all these polynomials have rational singularities (or are smooth).
\end{abstract}
\maketitle
\tableofcontents

%%%%%%%%%%%%%%%%%%%%%%%%%%%%%%%%%%%%%%%%%%%%%%%%%%%%%%%%%%
\section{Introduction}
%%%%%%%%%%%%%%%%%%%%%%%%%%%%%%%%%%%%%%%%%%%%%%%%%%%%%%%%%%

A \emph{configuration} is the choice of a subspace $W$ inside an $n$-dimensional vector space $V$ with basis $E=\{e_1,\ldots,e_n\}$ over the field $\KK$. If $W$ arises as the row span of a $\rank(W)\times n$ matrix $A$, this realization of the associated matroid of columns of $A$ induces a \emph{configuration polynomial} $\psi_A$ in the polynomial ring $\KK[E]$; up to a nonzero factor, it only depends on $W$ (and not on $A$). In the classical case, when $A$ is the (truncated) incidence matrix of a graph $G$, this is the Kirchhoff polynomial $\phi_G$ of $G$. Configuration polynomials in general were introduced in \cite{BEK} and placed in a matroidal context in \cite{Patterson}. 

An alternative generalization of Kirchhoff polynomials to all (and not just realizable) matroids are the \emph{matroid basis polynomials}: the sum of those monomials (with coefficient 1) that encode the bases of the matroid. These classes of polynomials have received much recent attention inspired by the work of Huh and his collaborators, see \cite{LogConcavePolynomials,EurHuhLogConcavity}. The fundamental underlying commonality of Kirchhoff, matroid basis, and configuration polynomials is that they satisfy certain recursive \emph{Deletion-Restriction} identities. Based on these,  a study of all these classes of polynomials was undertaken in \cite{DSW}, focussing on size and structure of the singular locus. One particular consequence of this study was the discovery that matroid basis polynomials of representable matroids can fail to be configuration polynomials.

\medskip

In this article we formulate a triad of classes of (often inhomogeneous) polynomials of rising complexity that contain all of the above types of polynomials and uniformly enjoy Deletion-Restriction identities. These are the \emph{matroid support polynomials}, the \emph{matroidal polynomials}, and the \emph{flag matroidal polynomials}.

Matroid support polynomials are simply those whose monomial support agrees with the bases of some matroid. Matroid support polynomials abound in combinatorial contexts and elsewhere. For example, the uniform matroid $\matU_{d,n}$ with $n>d\in\NN$ supports the $d$-th elementary symmetric polynomial in $n$ variables, as well as in fact every polynomial with monomial support exactly all squarefree monomials of degree $d$,
%fix $d < n$ and consider a $\mathbb{K}^{\times}$-weighted sum of all squarefree monomials of degree $d$ on $n$ variables:
%\begin{equation*}
%    \sum_{\substack{I \subseteq \{1, \dots, n\} \\ |I| = d}} c_{I} \bsx^{I} \in \KK[x_1, \dots, x_n] \qquad\text{ with } c_I \in \KK^\times  \, \forall I.
%\end{equation*}
%This is a matroid support polynomial on the uniform matroid $\matU_{d,n}$ 
compare Example \ref{exa-UniformMat}. Another natural class of examples arises in graph theory: if $\calT_G^k$ enumerates the  circuit-free sets of edges of cardinality $k$ in a graph $G$, and if $r$ is the common size of the spanning forests of $G$, %vertices. We consider the $\KK^\times$-weighted sum of monomials
%\begin{equation*}
%    \sum_{T \subseteq \text{Tree}_{n}} c_{T} \prod_{ (i,j) \in T} y_{i,j} \in \KK[\{y_{i,j} \}_{1 \leq i < j \leq n}] \quad c_T \in \KK^\times  \, \forall T.
%\end{equation*}
then any $\KK^\times$-weighted sum over $\calT^r_G$ (resp.\ $\calT^k_G$ with $k\le r)$ gives rise to a matroid support polynomial for 
%Again this is a matroid support polynomial, now on 
the graphic matroid attached $\matM_G$ to $G$ from Example \ref{ex-GraphicMat} (resp.\ for the truncations of $\matM_G$, compare Example \ref{ex - truncations are quotients}). 

The second class on our list, comprised of matroidal polynomials, is new. Its constituents require, aside from the matroid in question, another input type that we name \emph{singleton data}. This ingredient determines to what extent the resulting expression strays from being a matroid support polynomial.
The additional freedom permits the construction of inhomogeneous objects such as Tutte polynomials and the polynomial enumerating all sets of maximal rank in a matroid (Example \ref{ex-matroidal-polys}.\ref{item-maximalrank}), while "trivial" singleton data simply produce matroid support polynomials. 

Thirdly, we consider flag matroids, \emph{i.e.}\ finite sequences of matroid quotients. A basic example arises from fixing a matroid $\matM$ and then taking a sequence of truncations (Example \ref{ex - truncations are quotients}); each flag matroid polynomial here is then a
 $\KK^\times$-weighted enumeration of the independent sets of $\matM$ of rank at least the number of truncations in the flag (Example \ref{ex - repeated trunctations, monomial support on independent sets}).
 Our flag matroidal polynomials are an adaptation of our matroidal polynomials to the flag case.

\medskip

In the greater part of this article we investigate the type of singularities the members of the above classes offer. Typically,
an element of chaos surrounds the singularities arising out of geometric constructions involving all matroids.  We cite three examples: a) in \cite{BelkaleBrosnanMatroidsKontsevich}, Belkale and Brosnan studied
the hypersurfaces attached to Kirchhoff polynomials and proved that the collection of all Kirchhoff hypersurface complements generates the ring of all geometric motives; b) fixing a representable matroid $\matM$ and the moduli space of all its representations in the appropriate Grassmannian, Mn\"{e}v \cite{Mnev} and Sturmfels \cite{SturmfelsMatroidStrata} showed that these moduli spaces can have arbitrarily complicated singularities over $\QQ$ if one varies $\matM$; c) according to \cite{DSW}, size and Cohen--Macaulayness of the singular locus can vary wildly on the classes of configuration or matroid basis polynomials on $\matM$  (see Subsection \ref{subsect-ModuliSpace}). 

In stark contrast, our studies here on (flag) matroidal polynomials show that they universally enjoy very mild singularities: the singularities of every irreducible polynomial on our list are rational. Remarkably, our discussion involves all matroids, whereas the cases listed above all reside within the subclass of representable matroids (asymptotically with $n$ of density zero \cite{NelsonMatroidNonRep}).

Rational singularities are somewhat rare, and usually have appeared in situations that are amenable to birational or characteristic $p$ methods. For example, the singular points of the following are rational: toric varieties and other quotients by linear transformation groups; generic determinantal varieties \cite{BrunsVetter}; Hankel determinantal varieties \cite{HankelDeterminatalRings}; Schubert varieties \cite{PositivityComplexFlag}; positroid varieties \cite{PositriodVarieties}; theta divisors \cite{ThetaDivisors}; moments maps of quivers with at least one vertex and two loops \cite{BudurRationalMoments}. On the other hand, many ``simple'' varieties do not have rational singularities: hyperplane arrangements; degree $d \geq n$ affine cones of smooth hypersurfaces in $\PP_{\CC}^{n-1}$; most free divisors with strong homogeneity conditions \cite{NarvaezMacarroDuality} (c.f. Remark \ref{rmk-rat-sing-properties}). The birational geometry of Kirchhoff, configuration and matroidal polynomials as at this point relatively unexplored. While the results of Belkale and Brosnan should perhaps be viewed as a caution, our results might  be interpreted as a beacon of hope.

\medskip

Our methods of investigation here rely on the concept of jet spaces, a theory that classically developed in the study of differential equations but was transplanted into algebraic geometry. We focus particularly on Musta\c t\u a's results from \cite{MustataJetsLCI,SingularitiesPairsMustata} characterizing rational singularities in terms of irreducibility of jet spaces. These tools in turn rely on deriving certain dimensional inequalities. In order to establish the necessary estimates, we use the Deletion-Contraction identities in an inductive scheme.  This process aims to preserve connectedness under deletion of edges, and in the face of the non-existence of a suitable such edge relies on the construction of a certain collection of edges termed a \emph{handle}. 

Packaging Theorems \ref{thm - matroid poly, connected implies m-jets irreducible}, \ref{thm - flag matroidal poly, m-jets irreducible} and Corollaries \ref{cor-mtrdl-rat-sing}, \ref{cor-flag-mtrdl-ratsing} together yields:
\begin{thm}
    Let $\matM$ be a connected matroid of positive rank; let $\scrM$ be a terminally strict and terminally connected flag matroid of terminal rank at least two. Over an algebraically closed field $\KK$ of characteristic zero and for all $m \in \mathbb{Z}_{\geq 0}$, we have:
    \begin{itemize}
        \item the $m$-jets of any matroidal polynomial $\zeta_{\matM}$ on $\matM$ are irreducible;
        \item the $m$-jets of any flag matroidal polynomial $\zeta_{\scrM}$ on $\scrM$ are irreducible. 
    \end{itemize}
    In particular, over $\KK$, the following have rational singularities (or are smooth):
    \begin{itemize}
        \item any matroidal polynomial $\zeta_{\matM}$ on $\matM$;
        \item any flag matroidal polynomial $\zeta_{\scrM}$ on $\scrM$.
    \end{itemize}
\end{thm}

\medskip

For the first class we consider, the matroid support polynomials, we also undertake a study in finite characteristic. Here, the Frobenius homomorphism allows to define a number of singularity types that parallel notions in characteristic zero; see Subsection \ref{subsec-primer-F} for details. For example, 
%The theorem below, Theorem \ref{thm-%strongFreg}, shows  that matroid support %polynomials are in fact strongly 
the notion of \emph{strong $F$-regularity} is, according to results of K.~Smith, a positive characteristic strengthening of the characteristic zero notion of a rational singularity, see  \cite{Smith-AJM97,TakagiWatanabe}. We prove here:

\begin{thm}
    Let $\matM$ be a connected matroid on the ground set $E$. Let $\chi_{\matM} \in \FF[\{x_{e}\}_{e \in E}]$ be a matroid support polynomial where $\FF$ is a field of finite characteristic. Then $\chi_\matM$ is strongly $F$-regular.
\end{thm}

\medskip

The main motivation for our study arises in Quantum Field Theory, concerned with the qualitative and quantitative properties of certain integrals the importance of which we sketch in Subsection \ref{subsec-FInt}; we refer the reader to   \cite{AluiffiMarcolliParametricFeynman,BrownFeynmanAmplitudes} for a sample of the literature.  These integrals arise from a \emph{Feynman diagram}, a graph $G$ decorated with mass data (on the edges) and a momentum function (on the vertices). These data combine for the formulation of the \emph{Feynman integral} \eqref{eqn-FeynmanIntegral} over an expression involving a \emph{Feynman diagram polynomial}. Understanding the singularity behaviour of the denominator throws interesting light on the convergence of the integral, \emph{cf.}\ Remark \ref{rmk-rat-sing-properties}.

Feynman diagram polynomials are special instances of the \emph{Feynman integrands} of Section \ref{sec-Feynman}; their structure is somewhat different from all previously discussed classes since they involve products, and not just sums, of matroid support polynomials. Nonetheless, we show in our last result, Theorem \ref{thm-feyn-poly-ratsing} (based on Theorem \ref{thm - feynman integrand, irreducible jets}, Corollary \ref{cor-Feyn-ratsing}), that the vast majority of Feynman integrals exhibit a denominator that has rational singularities:

\begin{thm}
    Let $G = G(V, E, \massm, p)$ be a Feynman diagram such that (the graphic matroid) $\matM_G$ is connected. Then the Feynman diagram polynomial $\mathscr{G}$ has rational singularities (or is smooth) and its $m$-jets are irreducible, provided that the Feynman diagram satisfies one of the following:
    \begin{itemize}
    \item $\massm,p$ are identically zero,  or
    \item $p$ is scalar and not identically zero, or
    \item $p$ is vector-valued and has general kinematics.
    \end{itemize}
    The conclusions also hold for Feynman integrands $\Feyn(\zeta_\matN, \Delta_{\massm}^{E}, \xi_\matM)$ provided that the underlying matroid is connected of positive rank.
\end{thm}

The basic steps of the proof are in parallel to the (flag) matroidal case, with a more complicated process for finding an  appropriate  handle when necessary, since additional properties are required of the handle here. 

\medskip

We start below with an introductory section on the topics of matroids, matroidal polynomials, jets, and $F$-singularities. Then in Section \ref{sec-Fsing} we discuss matroid support polynomials and show that they are strongly $F$-regular in positive characteristic. Section \ref{sec-jets} sees the return of matroidal polynomials, and the jet-theoretic proof that they enjoy rational singularities over algebraically closed fields of characteristic zero. Sections \ref{sec-flag-matroids} introduces flag matroids, and establishes the rationality of the singularities of flag matroidal polynomials. Section \ref{sec-Feynman} begins with Feynman integrands and, via the by now standard chain of arguments, shows that they, too, have rational singularities. The paper concludes with the definition of Feynman integral polynomials and the proof that they are indeed Feynman integrands. Future research directions appear as a postscript. 

%%%%%%%%%%%%%%%%%%%%%%%%%%%%%%%%%%%%%%%%%%%%%%%%%%%%%%%%%%
\section{Basic Notions}
%%%%%%%%%%%%%%%%%%%%%%%%%%%%%%%%%%%%%%%%%%%%%%%%%%%%%%%%%%
We will use the following notation throughout.
\begin{asparaitem}
\item
We denote by $\NN,\ZZ,\QQ,\RR,\CC$ the natural, integer, rational, real and complex numbers; $\KK$ will denote an arbitrary field; $\FF$ a field of finite characteristic $p>0$. 
\item For two sets $A$ and $B$ we write $A\minus B$ for $\{a\in A\mid a\notin B\}$. 
\item For a subscheme $Z$ of affine $n$-space over $\KK$, $I(Z)$ denotes the corresponding defining ideal in the coordinate ring of $\KK^n$; $\Var(I)$ is the variety attached to an ideal.
\item If $f\neq 0$ is an element of a polynomial ring, $\min(f)$ will denote the sum of nonzero parts of $f$ of lowest degree.
\end{asparaitem}

\subsection{A Primer on Matroids}
%%%%%%%%%%%%%%%%%%%%%%%%%%%%%%%%%%%%%%%%%%%%%%%

Throughout, $\matM$ will denote a matroid on the finite ground set $E$ of size
\[
|E|=n.
\]
We review briefly some properties of matroids.
\begin{dfn}
A \emph{matroid} $\matM$ on the ground set $E$ is a collection $\calB_\matM\subseteq 2^E$ of the power set of $E$ called \emph{bases} that satisfies the exchange axiom 
\[
[B,B'\in \calB_\matM\text{ and }e\in B]\Longrightarrow [\exists e'\in B'| (\{e\}\cup(B'\minus\{e'\})), (\{e'\}\cup(B\minus \{e\})) \in \calB_\matM].
\]
\end{dfn}
Matroids reflect the combinatorial information about linear dependence of given vectors, about algebraic dependence of given polynomials, extend the notion of spanning forests in graphs, and also arise in other contexts. The language of matroids borrows much from graph theory; for example, the elements of $E$ are known as \emph{edges} of $\matM$.

We chose to define a matroid via its set of bases $\calB_\matM$;  alternatives include listing its its \emph{independent sets} $\calI_\matM$ (the collection of all subsets of its bases), its \emph{circuits} $\calC_\matM$ (the minimal non-independent sets in $2^E$), its \emph{rank function} $\rank_\matM(-)\colon 2^E\to\NN$ (the cardinality of a maximal independent subset to a given subset of $E$), and its collection of flats $\calF_\matM$ (the subsets of $2^E$ that maximize a given rank value). We will use these alternative notions and descriptions  freely here and refer to \cite{OxleyMatroidTheory} for details, proofs and theorems on matroids.
Nonetheless, the following constructions on matroids are so essential to us that we include them here explicitly for convenience.
\begin{dfn}\label{dfn-del-contr}
    Let $\matM$ be a matroid on $E\ni e$. The \emph{deletion
    matroid} $M\setminus e$ of $\matM$ by $e$ is the matroid on $E\minus \{e\}$ whose bases are the bases of
    $\matM$ that do not use $e$.

    The \emph{contraction matroid} $\matM/\{e\}$ of $\matM$ by
      $e$ is the matroid on $E\minus\{e\}$ with
      bases
      \[
      \calB_{\matM/\{e\}}=\{B\minus \{e\}\mid B\in\calB_\matM
      \text{ and }e\in B\}.
      \]
  \end{dfn}

\begin{dfn}\label{dfn-connected}
A matroid $\matM$ on $E$ is \emph{disconnected} if $E=E'\sqcup E''$ is the disjoint union, and there are matroids $\matM'$ on $E'$ and $\matM''$ on $E''$ such that $\calB_\matM=\calB_{\matM'}\times \calB_{\matM''}$.
\end{dfn}

Generalizing the idea of a dual to a planar graph, every matroid permits a dual matroid.
\begin{dfn}\label{dfn-dual-matroid}
The dual matroid $\matM^\perp$ of $\matM$ is the matroid on $E$ with bases 
\[
\calB_{\matM^\perp}:=\{E\minus B\mid B\in\calB_\matM\}.
\]
Obviously, $(\matM^\perp)^\perp=\matM$,  $\rank(\matM)+\rank(\matM^\perp)=|E|$, and duals of connected matroids are connected.
\end{dfn}

\begin{exa} \label{ex-GraphicMat}
If $G$ is a graph on the edge set $E$, then the collection of spanning forests of $G$ (maximal sets of edges that do not contain circuits) form the \emph{graphic matroid} $\matM_G$ to $G$. Non-isomorphic graphs can have isomorphic matroids. The \emph{cographic matroid} to $G$ is the dual of $\matM_G$, and is in the planar case the graphic matroid to the dual graph. The graphic matroid to a disconnected graph is disconnected, but so are graphic matroids to trees and various other graphs.
\end{exa}

\begin{exa}[Matroids on a Singleton]\label{exa-free-dep-mat}
Suppose $E=\{e\}$ is a singleton. Then $E$ admits two matroids: the \emph{free matroid} 
\[
\matF_{e} \textrm{ with bases } \calB_{\matF_e}=\{\{e\}\},
\]
and the \emph{loop (or dependent) matroid}
\[
\matD_{e} \textrm{ with bases } \calB_{\matD_e}=\emptyset.
\]

More generally, $e\in E$ in an arbitrary matroid $\matM$ is a \emph{loop} if $e$ is not contained in any basis of $\matM$; $e$ is a \emph{coloop} if it is part of every basis of $\matM$.
\end{exa}

\begin{exa}[Uniform matroids] \label{exa-UniformMat}
Let $0\le k\le n$ be integers. The \emph{uniform matroid} $\matU_{k,n}$ is the matroid on $n$ edges in which the bases are precisely the subsets of $E$ with cardinality $k$. So, for example, $\matU_{0,n}$ is a collection of $n$ loops; $\matU_{1,n}$ is the graphic matroid to a graph with two vertices, $n$ edges and no loops; $\matU_{n-1,n}$ is the graphic matroid to an $n$-gon.
\end{exa}

Several of our arguments will involve an induction on $n=|E|$, along deletion and contraction processes. An integral tool in such process will be the following concept.

      \begin{dfn}\label{dfn-handle}
    A \emph{handle} of $\matM$ is a non-empty set of edges
    $H\subseteq E$ such that for any circuit $C$ of $\matM$, $H$ is
    either disjoint to, or completely inside $C$. We call a handle
    \emph{maximal} if it is maximal within the set $\calH_\matM$ of
    all handles of $\matM$. A handle $H$ is \emph{proper} if $\neq E$.
  \end{dfn}
Details regarding handles can be viewed, for example, in \cite{DSW}. We quickly survey the landscape. First, a relationship between handles and connectedness:
\begin{prop}[{\cite[Prop.~2.8]{DSW}}]
Suppose $\matM$ is connected and $C$ a circuit of $\matM$. If $\matM \neq C$, then there exists a handle $H$ such that $\matM \setminus H$ is connected and $\matM \setminus H$ contains $C$ (as circuit, of course). 
\end{prop}
%%%

Second, we record how handles $H$ on $\matM$ behave with respect to $\matM \setminus h$ and $\matM / h$. We use these facts several times.

\begin{rmk}[{\cite[Lemma~2.4]{DSW}}]\label{rmk-handle basics} Throughout, $\matM$ is a matroid on $E$ admitting a proper handle such that $H \in \Ind_\matM$ and $H$ contains no coloops.
    \begin{asparaenum}
        \item $\rank(\matM \setminus H) = \rank(\matM) - |H|$;
        \item For any $h \in H$, every element of $H \minus \{h\}$ is a coloop on $\matM \setminus h$;
        \item $\rank(\matM \setminus H) = \rank(\matM) - |H| + 1$;
        \item For any $h \in H$, the set $H \minus\{h\}$ is a handle on $\matM / h$. Moreover, $H \minus\{h\} \in \Ind_{\matM \setminus h}$ and $H \minus\{h\}$ contains no coloops of $\matM / h$.
    \end{asparaenum}
\end{rmk}
\subsection{Polynomials attached to matroids}
%%%%%%%%%%%%%%%%%%%%%%%%%%%%%%%%%%%%%%%%%%%%%%%

Let $\matM$ be an arbitrary matroid on the ground set $E$,  and denote 
\[
\KK[E] := \KK[\{x_{e}\}_{e \in E}],
\]
the ring of polynomials in (variables matching) the edges of $\matM$ over $\KK$. For any subset $S\subseteq E$ we write
\[
\bsx^S:=\prod_{e\in S}x_e.
\]
We define here two classes of polynomials in $\KK[E]$,
the \emph{matroid support polynomials} and \emph{matroidal polynomials}. Members of both classes are of degree at most one in each variable, and satisfy certain Deletion-Contraction relations. The former class is comprised of homogeneous polynomials, while the second is also entitled to inhomogeneous members, even on matroids on a singleton: \emph{cf.}\ Example \ref{ex-matroidal-polys}.

\subsubsection{Matroid support polynomials}
%%%%%%%%%%%%%%%%%%%%%%%%%%%%%%%%%%%%%%%%%%%%
\begin{dfn}\label{dfn-msp}
  If $\matM$ is a matroid on the ground set $E$, then a 
  polynomial $\chi_\matM\in S_\matM=\KK[E]$ is a \emph{matroid support
    polynomial} (on $\matM$) if
  \[
  \qquad \chi_\matM=\sum_{B\in\calB_\matM}c_B\bsx^B,\qquad c_B\neq 0\,\,\forall B\in\calB_\matM
  \]
  has monomial support exactly the bases of $\matM$. We also introduce
  \[
  \MSP(\matM):=\{\chi_\matM\in \KK[E]\mid \chi_\matM \text{ is a matroid support polynomial on }\matM\}.
  \]
\end{dfn}
Matroid support polynomials on $\matM$ are squarefree and homogeneous of degree
$\rank(\matM)$. One can show that they factor only if the come from a decomposable matroid, along the lines of \cite[Lemma~2.4(b)]{DSW}. It is interesting to note that the implication cannot be reversed. For example, $f=axz+bxw+cyz+dyw$ is an MSP on the sum of the Boolean matroid of rank two with itself, but $f$ factors precisely when $ad=bc$.

\subsubsection{Matroidal Polynomials}
%%%%%%%%%%%%%%%%%%%%%%%%%%%%%%%%%%%%%%%%%%%%%%%%%%%%%%%%%%

\begin{conv}
In the context of matroidal polynomials, we will always assume that the characteristic of the underlying field $\KK$ is zero and that $\KK$ is algebraically closed. 
\end{conv}
A matroidal polynomial on $E$ involves two ingredients: a matroid $\matM$ on $E$, and a choice of singleton data $\sigma$. 
\begin{dfn}[Singleton Data]
    Let $E$ be a ground set with corresponding coordinate ring $\KK[E]$.

\smallskip
    
    For each $g \in E$ fix two monic polynomials $f_{g}, d_{g} \in \KK[x_{g}]$ such that
    \[
    \deg f_{g} = 1,  \qquad \deg d_{g} \leq 1.
    \]
    We call $f_{g}$ the \emph{free monic singleton on $g$}; $d_{g}$ the \emph{dependent monic singleton on $g$}. We call the collection $\{f_{g}, d_{g}\}_{g \in E}$ the \emph{singleton data}, and frequently use $\sigma$ as abbreviation for chosen singleton data.
\end{dfn}
\begin{conv}
Singleton data $\sigma$ on $E$ induce singleton data 
\[
\sigma_{|_{\scriptstyle E'}}:=\{f_e,d_e\}_{e\in E'}
\]
on any subset $E'$ of $E$. In the following definition (and throughout unless expressly indicated otherwise), given singleton data on $E$ it will be understood that the induced singleton data are used on subsets of $E$. 
\end{conv}

\begin{dfn}[Matroidal Polynomials]
\label{dfn - new - matroidal polys}
Given singleton data $\sigma=\{f_{g}, d_{g}\}_{g \in E}$, let $\matM$ be a matroid on $E$. We define the class $\Matroidal(\matM,\sigma)$ of \emph{matroidal polynomials on $\matM$ with singleton data $\sigma$} inductively as follows. 

 The members
    \[
    \zeta_\matM\in \Matroidal(E, \sigma) 
    \]
    are elements of $\KK[E]$ and attached to a matroid $\matM$ on $E$ (explicit in the notation) and the singleton data $\sigma$ (not explicit in the notation). 

\smallskip

    \noindent\textit{Augmentation}:
    When $E = \emptyset$, the only matroid is $\emptyset$, which allows for no singleton data.  The only matroidal polynomials are here
    \[
    \Matroidal(\emptyset, \emptyset) := \KK^{\times}.
    \]
    If $|E|=1$, there are two possible  matroids $\matD_e,\matF_e$ on $E$, compare Example \ref{exa-free-dep-mat}. The class of matroidal polynomials on $E$ is then given by:
    \[
    [\zeta_\matM \in \Matroidal(E, \{f_e, d_e\})] \iff 
%\begin{cases}
%[\matM = \matF_e] \implies [\zeta_{\matM_e} \in \{\alpha_e f_e\}], \\
%[\matM = \matD_e] \implies [\zeta_{\matM_e} \in \{\alpha_e d_e\}].
%\end{cases}   \]\[
 \begin{cases}
\zeta_{\matM} \in \KK^\times\cdot f_e\text{ if }\matM=\matF_e,\\
\zeta_{\matM} \in \KK^\times\cdot d_e\text{ if }\matM=\matD_e.
\end{cases}%\Matroidal(,\{f_e\}):=\{\alpha_e f_{e} \mid \alpha \in \KK^{\times}\},
%    \qquad \Matroidal(\matD_{e},\{d_e\}):=\{\alpha_e d_{e} \mid \alpha \in \KK^{\times}\}.
    \]

\smallskip

    \noindent\textit{Recursion}:

    For $|E| \geq 2$, $\zeta_\matM\in\KK[E]$ is an element of $\Matroidal(E, \sigma)$ exactly when all of the following three conditions hold:
    \begin{enumerate}[label=(\roman*)]
        \item (Deletion-Contraction of non-(co)loops)\\
        Suppose that $e \in E$ is neither a loop nor a coloop on $\matM$. Then there exist matroidal polynomials $\zeta_{\matM\setminus e}, \zeta_{\matM/ e} \in \Matroidal(E \minus \{e\}, \sigma_{|_{\scriptstyle E \minus \{e\}}})$, attached to the matroids $\matM\setminus e$ and $\matM/e$ respectively, such that 
        \[
        \qquad\zeta_\matM = \zeta_{\matM\setminus e} + f_{e} \zeta_{\matM/ e}. 
        %\quad \text{ while } \quad \deg(\zeta_{\matM\setminus e}) \geq \deg (\zeta_{\matM/ e}).
        \]
        \item (Deletion-Contraction of coloops)\\
        Suppose that $e \in E$ is a coloop on $\matM$. Write $\matN = \matM/e = \matM\setminus e$. Then there exists a matroidal polynomial $\zeta_{\matN} \in \Matroidal(E \minus \{e\}, \sigma_{|_{\scriptstyle E\minus \{e\}}})$ attached to the matroid $\matN$ such that
        \[
        \zeta_\matM = f_{e} \zeta_{\matN}.
        \]
        \item (Deletion-Contraction  of loops)\\
        Suppose that $e \in E$ is a loop on $\matM$. Write $\matL = \matM/e = \matM\setminus e$. Then there exists a matroidal polynomial $\zeta_{L} \in \Matroidal(E \minus \{e\}, \sigma_{|_{\scriptstyle E\minus \{e\}}})$ attached to the matroid $L$ such that
        \[
        \zeta_\matM = d_{e} \zeta_{L}.
        \]
 %       \item (Irreducibility Axiom)\\
 %       If $\matM$ is connected, then $\zeta_\matM$ is irreducible.
    \end{enumerate}

    Finally,     
    \[
    X_{\zeta_\matM}:=\Var(\zeta_\matM) \subseteq \KK^{E}.
    \]
    is \emph{the matroidal polynomial hypersurface} defined by the matroidal polynomial $\zeta_\matM \in \Matroidal(E, \{f_{g}, d_{g}\}_{g \in E})$.
\end{dfn}
\begin{rmk}
 Let $\zeta_\matM\in\Matroidal(E,\sigma)$ be given and let $e \in E$ be an edge of $\matM$. The polynomials $\zeta_{\matM/ e}$ and $\zeta_{\matM\setminus e}$ referenced in the Deletion-Contraction axioms (with singleton data inherited from $\sigma$) are uniquely determined by $\zeta_\matM$; this follows from $f_{e},d_e$ being monic in $\KK[x_e]$   while $\zeta_{\matM/ e}, \zeta_{\matM\setminus e} \in \KK[E \minus \{e\}]$. The point is, that they are in fact matroidal again. In the context of a chosen such $e$ and $\zeta_\matM$ we will always mean the thus defined $\zeta_{\matM/ e}$ and $\zeta_{\matM\setminus e}$, without each time pointing towards Deletion-Contraction.
 \end{rmk}
 
The following remark explores some basic features of the preceding definition.
\begin{rmk}\label{rmk - basic matroid polynomial observations}\begin{asparaenum}
\item
\label{item-Mtrdl=scalable}
    If  $\zeta_\matM \in \Matroidal(E, \sigma)$ is a matroidal polynomial attached to the matroid $\matM$ on $E$, and if $\alpha \in \KK^{\times}$, then $\alpha \zeta_\matM \in \Matroidal(E,\sigma)$ is a matroidal polynomial attached to $\matM$,
    as follows immediately from the linearity of the axioms for matroidal polynomials.
\item If one is only interested in the geometric properties of matroidal polynomials on a given set $E$ and with given singleton data $\sigma$, one may safely assume that $f_e=x_e$ for every $e\in E$, via a linear coordinate change in $\KK[E]$.
\item Matroidal polynomials contain only square-free monomials, since products in the recursive recipe never multiply polynomials with shared variables.  
\item\label{item-U12-matrdl} The calculations in this item and the next three illustrate that the coherence condition on the Deletion-Contraction recursion process (on every possible path between $E$ and the empty set) places strong restrictions on the coefficients that appear in the recursion. 

We consider the matroidal polynomials on the 2-circuit $\matM=\matU_{1,2}$ on edges $E=\{f,g\}$. Up to a coordinate change, $f_e=x_e$ for all $e\in E$. 

The case $d_g=d_h=1$ leads to a linear form $\zeta_\matM=\alpha_gx_g+\alpha_hx_h$, $\alpha_g\alpha_h\neq 0$. The case $d_g=x_g+\delta_g$, $d_h=1$ can be ruled out (with $\delta_g\in\KK$): Deletion-Contraction on the two respective edges leads to
$c_1x_h+x_gc_2\cdot 1=\zeta_\matM=c_3x_g+c_4x_h(x_g+\delta_g)$
which, with nonzero $c_4$, is impossible.
The case $d_g=x_g+\delta_g,d_h=x_h+\delta_h$ leads to
\[
c_1x_h+x_gc_2(x_h+\delta_h)=\zeta_\matM=
c_3x_g+x_hc_4(x_g+\delta_g)
\]
and shows that we must have 
$
c_4\delta_g=c_1$, $c_2\delta_h=c_3$, $c_2=c_4
$.
Since $c_1,c_2,c_3,c_4$ are supposed to be nonzero, we have $\delta_g\neq 0\neq \delta_h$, and
\[
\zeta_\matM=c_4(\delta_gx_h+\delta_hx_g+x_gx_h)=c_4\left((1+x_g\delta_h)(1+x_h\delta_g)-1\right).
\]
Up to rescaling $x_g,x_h$, the only matroidal polynomials on $\matU_{1,2}$ are thus the multiples of $x_g+x_h$, and those of $(1+x_h)(1+x_g)-1$. 
%\[
%\alpha_{h} x_{h} + \alpha_{gh} x_{g} x_{h} = \zeta_{\matM \setminus g} + f_{g} \zeta_{\matM / g} = \zeta_{\matM} = \zeta_{\matM \setminus h} + f_{h} \zeta_{\matM / h} = \beta_{g} x_{g} + \beta_{hg} x_{h} x_{g}
%\]
%where $\alpha_{h},\alpha_{gh},\beta_g,\beta_{hg} \in \KK^\times$. As such an equality is impossible, $\matM$ admits no matroidal polynomials with singleton data $\sigma$. This 

\item\label{item-parallel-sigma} Let $\matM$ be arbitrary and suppose $g,h$ are parallel edges. Deletion-Restriction requires the existence of a matroidal polynomial on the minor $\{f,g\}\simeq \matU_{1,2}$. The previous item shows that this necessitates either $d_f=d_g=1$ or $d_f=x_f+\delta_f$ and $d_h=x_h+\delta_h$ for nonzero $\delta_g,\delta_h$. 

\item\label{item-U1n-mtrdl}Suppose $\matM=\matU_{1,n}$. The previous two items imply that (up to rescaling all $x_e$) the matroidal polynomials on $\matU_{1,n}$ are either multiples of $\sum_{e\in E}x_e$, or come from $\sigma= \{x_e,x_e+\delta_e\}_{e\in E}$ where all $\delta_e$ are nonzero. In the latter case, rescale all $x_e$ so that $\sigma=\{x_e,x_e+1\}_{e\in E}$. Matroidal polynomials on $\matU_{1,2}$ for such $\sigma$ are documented above, and then by induction one derives for $|E|>2$ that
$\zeta_{\matU_{1,n}}=\alpha_e\left(\prod_{e'\neq e}(1+x_{e'})-1\right)+\beta_ex_e\prod_{e'\neq e}(1+x_{e'})=\alpha_e\left( (1+\frac{\beta_e x_1}{\alpha_e})\prod_{e'\neq e}(1+x_{e'})\right)-\alpha_e$. Now vary $e\in E$ in this process and compare coefficients to see that, for such $\sigma$ and in the right coordinates, $\zeta_{\matU_{1,n}}$ is a multiple of $\big[ \prod_{e \in E}(1+x_e) \big] -1$. 

\item \label{item-mat-poly-dependentFact} Suppose that $\sigma$ is singleton data on $E$ such that $f_{e} = x_{e}$ for all $e$ in $E$ and let $\matM$ be a loopless matroid on $E$. If $e \in E$ is not a coloop, then $d_{e} \neq x_e$.

Suppose this is false; let $g \in E$ witness the untruth. Pick a circuit $C$ containing $g$; since $g$ is no loop, we may find $h \in C \minus \{g\}$. 
    %Moreover, for any $h\in E$, there are matroidal polynomials in $\Matroidal({E\minus\{h\}}, \sigma_{|{E\minus\{h\}}})$ attached to $\matM \setminus h$ and $\matM / h$. So for any minor $\matF$ of $\matM$, we may find a matroidal polynomial $\zeta_{\matF}$ on appropriately restricted singleton data. 
Deleting all edges outside $C$, and then contracting $C\minus\{g,h\}$ 
makes the $2$-circuit $\matM_{1,2}=\matF$ on $\{g, h\}$ a minor of $\matM$. So there is a matroidal polynomial $\zeta_\matF \in \Matroidal(\{g, h\}, \sigma_{| \{g, h\}})$. By assumption, this singleton data satisfies $f_g = x_g, f_h=x_h, d_g = x_g$. But in Remark \ref{rmk - basic matroid polynomial observations}.\eqref{item-U12-matrdl} we showed that no matroidal polynomial on $\matU_{1,2}$ exists with such singleton data.

\item Any $\zeta_\matM$ is the product of $\prod_{e\textrm{ loop of }\matM}d_e$ with a matroidal polynomial on the underlying loopless matroid.

\item\label{item-mat-poly-nonzero} Matroidal polynomials are not zero. In fact, the degree of a matroidal  polynomial for the matroid $\matM$ on the edge set $E$ is at least the rank of $\matM$. Indeed, it is safe to assume that $\matM$ has no loops or coloops. Then let $B$ be a basis for $\matM$, choose singleton data $\sigma$, and let $\zeta_\matM\in\Matroidal(E,\sigma)$ be attached to $\matM$. Now consider the minor $\matB$ of $\matM$ in which all edges of $\matM$ that are not in $B$ have been deleted. 
Deletion-Contraction on the edges of $E\minus B$ implies the existence of a matroidal polynomial $\zeta_\matB$ on $\matB$ that fits into the recursion process. 
As $B$ is a basis, there is a unique (up to $\KK^\times$-scaling) matroidal polynomial $\zeta_\matB$ to $\sigma$ on $\matB$, namely $\prod_{b\in B}f_b$. The recursion process from $\zeta_\matM$ to $\zeta_\matB$  shows that $\zeta_\matM-\zeta_{\matB}$ is in the ideal $(\{f_e\mid e\in E\minus B\})$ and therefore $\zeta_\matM$ maps via the evaluation morphism with kernel $(\{f_e\mid e\in E\minus B\})$ to a polynomial of degree $|B|$.

%\item If $B$ is a basis for $\matM$, the monomial $\bsx^B$ may not appear in $\zeta_\matM$. For example, let $\matM$ be the graphic matroid attached to the triangle, with $\calB_\matM=\{\{a,b\},\{a,c\},\{b,c\}\}$. Choose $f_a=x_a,f_b=x_b,f_c=x_c=d_c$. Then deletion of $e=a$ shows that $f_af_b+f_c(f_a+f_bd_a)=x_ax_b+x_c(x_a+x_bx_c)$ is missing the term $x_bx_c$ to the basis $\{b,c\}$. Of course, $\zeta_\matM$ \emph{does} necessarily contain some multiple of $x_bx_c$; in this case it is $x_ax_bx_c$. We do not know whether the terms of minimal degree in $\zeta_\matM$ must all correspond to bases of $\matM$. 

%\item The previous example shows that a connected matroid may support polynomials that satisfy axioms (i,ii,iii) of the recursion conditions for matroidal polynomials, but yet be reducible. In particular, axiom (iv) is not redundant.

\item\label{item-mat-poly-orderDelCont}
Suppose $S = \{s_{1}, \dots, s_{a}\}, T = \{t_{1}, \dots, t_{b}\} \subseteq E$ are disjoint. Assume given singleton data $\sigma$ for $E$. Let $\matN$ be a minor of $\matM$ that arises from $\matM$ by deleting all elements of $S$ and contracting all elements of $T$, in some fixed order, allowing for the possibility of interlacing deletions and contractions.  (The resulting matroid $\matN$ does not depend on the chosen order of deletions and contractions).  
Then there is a matroidal polynomial $\zeta_{\matN}$ attached to singleton data $\sigma_{|_{E\minus(S\cup T)}}$  that fits into the Deletion-Contraction axiom along this order of edges in $S\cup T$.
In general, this matroidal polynomial $\zeta_\matN$  depends not just on $\zeta_\matM$, $S$ and $T$, but also on the order how the Deletion-Contraction axiom was iterated in the process that lead from $\matM$ to $\matN$. 
Example \ref{ex - config poly, minor issues} exhibits a matroidal polynomial $\zeta_{\matM}$, a minor $\matN$ of $\matM$ such that $\matN = \matM / S \setminus T = \matM \setminus T / S$, and induced matroidal polynomials $\zeta_{\matM \setminus S /T}$ and $\zeta_{\matM / T\setminus S}$ whose hypersurfaces have geometrically different properties. 
\item\label{item-mat-poly-degree} Let $F_\matM\subsetneq \KK[E]$ denote the ideal generated by $\{f_e\mid e\in E(\matM)\}$. Then $\zeta_\matM$ is contained in $(F_\matM)^{\rank \matM}\subsetneq \KK[E]$. Indeed, Deletion-Contraction terminating with a basis of $\matM$ shows the claim by induction on rank. In particular, if the  rank of $\matM$ exceeds 1 then all $\zeta_\matM$ are singular at the point $\Var(F_\matM)$.
\item\label{item-partition} Consider a factorization $\zeta_\matM = ab$. Let $e \in E$. Since the $x_{e}$-degree of $\zeta_\matM$ is at most one it is impossible for the $x_{e}$-degree of both $a$ and $b$ to be positive. Thus $a$ and $b$ use disjoint variables: the factorization induces a partition $A \sqcup B \subseteq E$ where $a \in \KK[A], b \in \KK[B]$.
\item\label{item-loopy} For a matroid with only loops and coloops, any $\zeta_\matM$ cuts out a Boolean arrangement.

\end{asparaenum}
\end{rmk}

\begin{proposition} \label{prop - initial term matroidal poly}
    Let $\sigma$ be singleton data on $E$ such that $f_{e} = x_{e}$ for all $e \in E$. If $\matM$ is loopless and $\zeta_\matM \in \Matroidal(E, \sigma)$, then $\min (\zeta_\matM) \in \MSP(\matM)$.
\end{proposition}

\begin{proof}
    By Remark \ref{rmk - basic matroid polynomial observations}.\eqref{item-mat-poly-dependentFact}, for all non-coloops $e \in E$ we have $d_e \neq x_e$. We proceed via cases. 

    \emph{Case 1: $\matM$ is the rank-1 loopless matroid $\matU_{1, n}$}. 
    
    \noindent We proceed by induction on $n$, with the case $n=1$ clear. Let $e\in E$ and set $F = E \minus\{e\}$, $\matF = \matM\setminus e$.   Contracting $e$ leads to a matroid of $n-1$ loops; deleting it leads to $\matU_{1,n-1}$. By Deletion-Contraction,
    \[
    \zeta_{\matM} = \zeta_{\matF} + b x_{e} \prod_{g\in E\minus\{e\}} d_{g}
    \]
    where $\zeta_\matF\in\Matroidal(\matF,\sigma_{|_F})$ and $b \in \KK^{\times}$. Each $d_{g}$ has a constant term and so $\min (\zeta_\matM - \zeta_\matF) \in \KK^{\times} x_{e}$. By the inductive framework, $\min(\zeta_{\matF}) = \sum_{f \in F} c_{f} x_{f}$ for units $c_{f} \in \KK^{\times}$. \emph{Case 1} follows.

    \emph{Case 2: $\matM \neq \matU_{1,n}$}.

    \noindent Again, proceed by induction on $n$, with the case $n=1$ clear. We may assume that $\matM$ has no coloops. Since $\matM \neq \matU_{1,n}$, we find $e \in \matM$ such that the smallest circuit containing $e$ has cardinality at least three. Then $\matM\setminus e$ and $\matM/e$ are both loopless, hence subject to the inductive hypothesis. Deletion-Contraction on $e$ yields
    \[
    \min(\zeta_{\matM}) = \min( \zeta_{\matM\setminus e} + x_{e} \zeta_{\matM/e}) = \min( \zeta_{\matM\setminus e} ) + x_{e} \min (\zeta_{\matM/e})
    \]
    where the last equality uses: the induction hypothesis; that there is no cancellation of terms since $\zeta_{\matM\setminus e}, \zeta_{\matM/e} \in \KK[E \minus \{e\}]$; that, for all matroids $\matF$, elements of $\MSP(\matF)$ are homogeneous of degree $\rank(\matF)$. Thus, $\min(\zeta_\matM) \in \MSP(\matM)$ follows.
\end{proof}

\begin{cor}\label{cor-irred}
    Let $\zeta_\matM \in \Matroidal(E, \sigma)$ be a matroidal polynomial attached to $\matM$. If $\matM$ is connected of positive rank then $\zeta_\matM$ is irreducible.
\end{cor}

\begin{proof}
    We may prove this after a linear change of coordinates on $\KK[E]$, arranging that $\sigma$ satisfies $f_{e} = x_{e}$ for all $e \in E$. By Proposition \ref{prop - initial term matroidal poly}, $\min(\zeta_\matM) \in \MSP(\matM)$. Any nontrivial factorization of $\zeta_\matM$ induces a nontrivial factorization on $\min(\zeta_\matM)$, which is itself irreducible since $\matM$ is connected of with positive rank.
\end{proof}

\begin{example} \label{ex-matroidal-polys}
    Given a matroid $\matM$ on $E$, the following are classes of matroidal polynomials. %If $\rank M \geq 2$, they are all singular matroidal polynomials.
    \begin{enumerate}[label=(\alph*)]
        \item\label{item-matroidbasis} The \emph{matroid basis polynomial}
        \[
        \Psi_{\matM} = \sum_{B\in\calB_\matM} \bsx^{B}
        \]
        is a matroidal polynomial on $E$ with singleton data $\{f_{g} = x_{g}; d_{g} = 1\}_{g \in E}$, homogeneous of degree $\rank(\matM)$. Here the matroidal polynomials arising in the Deletion-Contraction process are again matroid basis polynomials, $\zeta_{\matM\setminus e}=\Psi_{\matM\setminus e}$ and $\zeta_{\matM/e}=\Psi_{\matM/e}$. 
        %Irreducibility on connected $\matM$ follows as in \cite[Prop.~3.8]{DSW}. 
        \item\label{item-maximalrank} The \emph{matroid maximal rank polynomial}
        \[
        \MaxR_{\matM} = \sum_{\substack{A \subseteq E \\ \rank A = \rank E}} \bsx^{A}
        \]
        is a matroidal polynomial on $E$ with singleton data $\{f_{g} = x_{g}, d_{g} = 1 + x_{g}\}_{g \in E}$, inhomogeneous of degree $|E|$. 
        %The Deletion-Contraction polynomials are the matroid maximal rank polynomials on $\matM\setminus e$ and $\matM/e$. Note that $\deg(\text{MaxR}_{\matM}) = |E|$.  As for irreducibility on connected matroids, assume $\matM$ is connected and $ab = \text{MaxR}_{\matM}$. By Remark \ref{rmk - basic matroid polynomial observations}.\eqref{item-partition}, there is a partition $A \sqcup B \subseteq E$ where $a \in \KK[A]$, $b \in \KK[B]$. The factorization induces a factorization on minimal degree terms: $\min(a) \min(b) = \min (\text{MaxR}_{\matM}) = \Psi_{\matM}$, the matroid basis polynomial. As $\Psi_{\matM}$ is irreducible if and only if $\matM$ is connected, we may assume that $\min(a) = \Psi_{\matM}$. As $\matM$ is connected, every $e \in E$ appears in a basis of $\matM$ and hence a monomial of $\Psi_{\matM}$.  So $A = E$, $B = \emptyset$ and $b \in \KK^{\times}$.
        \item\label{item-configpoly} Let $A\in\CC^{r,n}$ be a full rank matrix. It induces a matroid $\matM$ of rank $r$, with bases the maximal independent column sets of $A$. If $E$ is the ground set of $\matM$ and $X$ is the diagonal $n\times n$ matrix with $X_{e,e}=x_e$  then the determinant $\Psi_A$ of $AXA^T$ is the \emph{matroid configuration polynomial} to $A$; it is a matroid support polynomial of $\matM$ and thus a matroidal polynomial if $\matM$ is connected. See \cite{DSW,Patterson,BEK} for more details. 
     %   \item\label{item-MSP-are-MtrdlM} It follows from Deletion-Contraction that matroid support polynomials $\chi_\matM=\sum_{B\in\calB_\matM}c_B\bsx^B$ on $\matM$ are exactly the matroidal polynomials for the singleton data $\sigma=\{f_e=x_e,d_e=1\}_{e\in E}$.  
        %Namely, $\chi_\matM=x_e\sum_{e\in B\in\calB_\matM}c_B\bsx^{B\minus\{e\}}+\sum_{e\notin B\in \calB_\matM}c_B\bsx^B$ splits $\chi_\matM$ into two matroid support polynomials, which (via induction on $|E|$) are matroidal, satisfying Deletion-Contraction. In reverse, it is clear that  $\sigma$ produces homogeneous matroidal polynomials with full support.   
        \end{enumerate}
\end{example}

Suppose $\zeta_\matM$ is a matroidal polynomial for the matroid $\matM$ on $E$. Regarding just the matroid structure, 
the operations of matroid deletion and matroid contraction of (subsets of) elements of $E$ commute. If the resulting minor is a single edge, the singleton data determine the corresponding matroidal polynomial on the minor induced from $\zeta_\matM$. The following example shows that for larger minors of $\matM$, $\zeta_\matM$ can discriminate between the different ways an abstract matroid can arise from $\matM$ as minor.

\begin{example} \label{ex - config poly, minor issues}

Let $\lambda_{1} < \lambda_{2} < \cdots < \lambda_{6}$ and $x, y, z$ be sufficiently generic elements of $\mathbb{Z}_{\geq 1}$. Consider the matrices
    \[
    A = \begin{pmatrix} 
    1 & 0 & 0 & 1 & 2 & 3 \\
    0 & 1 & 0 & 2 & 3 & 4 \\
    0 & 0 & 1 & 2 & 6 & 12
    \end{pmatrix}
    \quad B = \begin{pmatrix}
    x^{\lambda_{1}} & x^{\lambda_{2}} & x^{\lambda_{3}} & x^{\lambda_{4}} & x^{\lambda_{5}} & x^{\lambda_{6}} \\
    y^{\lambda_{1}} & y^{\lambda_{2}} & y^{\lambda_{3}} & y^{\lambda_{4}} & y^{\lambda_{5}} & y^{\lambda_{6}} \\
    z^{\lambda_{1}} & z^{\lambda_{2}} & z^{\lambda_{3}} & z^{\lambda_{4}} & z^{\lambda_{5}} & z^{\lambda_{6}}
    \end{pmatrix}.
    \]
The column matroids $\matM(A)$ and $\matM(B)$ are both $\matU_{3,6}$ (\cite[Ex.~ 5.3, Ex.~2.20]{DSW}). Now consider the $6 \times 12$ block matrix
\[
G = \begin{pmatrix}
    I_{3} & 0 & A \\
    0 & I_{3} & B
\end{pmatrix}
\]
and let $S$ denote the column set $\{1,2,3\}$, $T$ the column set $\{4, 5, 6\}$.

Let $\matM(G)$ be the column matroid of $G$. The row space of $G$ is a realization $W \subseteq \KK^{12}$ with associated configuration polynomial $\zeta_{\matM(G)}$. Contracting elements of $S=\{1,2,3\}$ from smallest to largest and \emph{then} deleting elements of $T=\{4,5,6\}$ from smallest to largest gives, according to Deletion-Contraction, a configuration polynomial $\zeta_{\matM(B)}$ for $B$. On the other hand, deleting elements of $T$ from smallest to largest and \emph{then} contracting the elements of $S$ from smallest to largest induces a configuration polynomial $\zeta_{\matM(A)}$ of $A$. Thus, both are configuration polynomials to the "same" minor $\matU_{3,6}$ of $\matM(G)$, but define qualitatively different hypersurfaces by \cite[Ex.~5.3]{DSW}: the Jacobian scheme of the former is reduced and Cohen--Macaulay; the Jacobian scheme of the latter is not reduced and not Cohen--Macaulay.
%Up to scalar multiple, $\rho_{\matM(B)}$ is the induced matroidal polynomial of $\zeta_{\matM(G)}$ obtained by contracting $\{1, 2, 3\}$ in order and then deleting $\{4,5,6\}$ in order. Similarly, $\varphi_{\matM(A)}$ is (up to scalar multiple) induced from $\zeta_{\matM(G)}$ by deleting $\{4,5,6\}$ in order and then contracting $\{1,2,3\}$ in order. By construction $\matM(G) / S \setminus T = \matM(B) = \matU_{3,6} = \matM(A) = \matM(G) / T \setminus S$. 
\end{example}

Matroidal polynomials generalize matroid support polynomials, under the most simple choice of singleton data:

\begin{proposition}\label{prop-MSP-are-MtrdlM}
Let $\matM$ be a matroid on $E$. Then the class of matroid support polynomials on $\matM$ agrees with the class of matroidal polynomials on $E$ attached to $\matM$ with singleton data $\sigma = \{f_e = x_e, d_e = 1\}_{e \in E}$:
    \begin{equation*}
        \MSP(\matM) = \{\zeta_{\matN} \in \Matroidal(E, \{f_e = x_e, d_e = 1\}_{e \in E} \mid \matM = \matN) \}.
    \end{equation*}
\end{proposition}

\begin{proof}
    Induce on $|E|$ and use Deletion-Contraction identities.
\end{proof}

    %We induce on $|E|$ with the base cases of $E = \emptyset$ and $|E| = 1$ immediate by construction. So assume an inductive hypothesis and let $\zeta_\matM$ be a matroidal polynomial on ground set $E$. Let $\alpha \in \KK^{\times}$. We verify the matroidal polynomial axioms. Certainly $\zeta_\matM$ is irreducible if and only if $\alpha \zeta_\matM$ is irreducible. Let $e \in E$ be a non-(co)loop. By assumption there are matroidal polynomials $\zeta_{\matM\setminus e}$ and $\zeta_{\matM/ e}$ such that 
    %\begin{equation} \label{eqn - new matroidal poly closed under scalars}
    %    \zeta_\matM = \zeta_{\matM\setminus e} + f_{e} \zeta_{\matM/ e}.
    %\end{equation}
    %By the inductive set-up, $\alpha \zeta_{\matM\setminus e}$ and $\alpha \zeta_{\matM/ e}$ are matroidal polynomials with respect to $\matM\setminus e$ and $\matM/e$ respectively. So \eqref{eqn - new matroidal poly closed under scalars} yields
    %\[
    %\alpha \zeta_\matM = (\alpha \zeta_{\matM\setminus e}) + f_{e} (\alpha \zeta_{\matM/ e}) \quad \text{ AND } \quad \deg( (\alpha \zeta_{\matM\setminus e}) \geq \deg(\alpha \zeta_{\matM/ e})
    %\]
    %and $\alpha \zeta_\matM$ satisfies Deletion-Contraction for non-(co)loops. The verification of Deletion-Contraction for loops and coloops is entirely similar. 
%\end{proof}

\subsection{A Primer on $F$-singularities}\label{subsec-primer-F}
%%%%%%%%%%%%%%%%%%%%%%%%%%%%%%%%%%%%%%%%%%%%%%%%
\begin{ntn}
Whenever we deal with fields of positive characteristic, we use the letter $\FF$ as default.
\end{ntn}

Our main references for
background on this subsection are
\cite{HH94TAMS,MaPolstra}.
Throughout this section, let $R$ be a ring of characteristic $p>0$
containing 
the field $\FF$.
On $R$ there is the \emph{Frobenius morphism} $F$, the ring
endomorphism $F$ that raises elements to the $p$-th power.  We
denote $F^e$ the $e$-fold iterate of $F$.  In order to distinguish
various copies of $R$ that arise as source and target of the
Frobenius, we shall write
\[
F^e\colon R\to F^e_*R,\qquad r\mapsto F^e_*r^{p^e}.
\]
So, $F^e_*R$ denotes the $R$-module whose additive group structure is
that of $R$, and which is an $R$-module via $F^e$. The element of
$F^e_*R$ that corresponds to $r\in R$ in the group isomorphism $R\iso
F^e_*R$ will be denoted $F^e_*r$.  Note that $F^e_*R$ is not just an
$R$-module: it has a natural multiplication that makes it an
$R$-algebra via $F$.

Given an ideal $I\subseteq R$, the Frobenius allows us to form the
\emph{$e$-th bracket} or \emph{Frobenius power} $\bracket{I}{e}$ of
$I$ as the ideal
\[
\bracket{I}{e}:=R\{r^{p^e}\mid r\in I\}
\]
generated by the $p$-th powers of the elements of $I$.

The ring $R$ is \emph{$F$-finite} if $F^e_*R$ is a finitely generated
$R$-module for some $e$ (or, equivalently, for all $e$).  Moreover,
$F$-finiteness is preserved under localization and under the formation
of quotients, polynomial rings and power series rings. In particular,
if the extension degree $[\FF:\FF^p]$ is finite, then all quotients of
polynomial rings $\FF$ are $F$-finite. Specifically, this applies when
$\FF$ is finite, or when it is separably closed.

Recall that an injective  morphism $\phi\colon M\to M'$ of $R$-modules is
\emph{pure} if $\phi\otimes_RN$ is injective for every $R$-modules $\matN$. 

\begin{dfn}\label{dfn-F-pure}
  The ring $R$ is \emph{$F$-pure} if $F^e\colon R\to F^e_*R$ is pure
  for some positive $e$,
  The ring $R$ is called
  \emph{$F$-split} if the map splits as map of $R$-modules.
\end{dfn}
A split map stays split under any linear functor and hence being
$F$-split implies being $F$-pure. On the other hand, $F$-purity
implies that the Frobenius is injective and so $R$ has no nilpotent
elements. For $F$-finite rings, and for complete local rings, the two
concepts are in fact equivalent by \cite[Cor.~2.4]{MaPolstra}.

If $F^e\colon R\to F^e_*R$ is split or pure, so is the corresponding
map for all $e'>0$, see \cite[Dfn.~2.1]{MaPolstra}. If $R$ is a field,
then $F_*R$ is a vector space over $R$ and in particular free, so $R$
is $F$-split. 
A polynomial ring $\FF[x_1,\ldots,x_n]$ is $F$-split since the map
\[
\bsx^\boldu\mapsto\left\{\begin{array}{cl}
\bsx^\boldv&\text{if }\boldu=p\cdot \boldv\text{ for some $\boldv\in\NN^n$};\\
0&\text{else;}\end{array}\right.
\]
(when combined with a splitting for $\FF\to F_*\FF$) splits the
Frobenius.

Fedder\cite{Fedder} gave a useful explicit criterion to determine
whether a quotient of a polynomial ring $R$ over an $F$-finite field
by an ideal $I\subseteq R$ generated by homogeneous polynomials is
$F$-pure: it is so if and only if the ideal
quotient
\[
\bracket{I}{1}\colon I\not\subseteq 
\bracket{\frakm}{1}
\]
is not in the Frobenius power of the maximal ideal $\frakm$ generated by the
indeterminates. For example, when $I=(f)$ is principal and $f$
homogeneous, then $R/I$ is $F$-pure if and only if $f^{p-1}$ is not in
$\bracket{\frakm}{1}$, \cite[Remark 2.8]{MaPolstra}.

A strengthening of the splitting condition is $F$-regularity:

\begin{dfn}\label{dfn-Freg}
The $F$-finite ring $R$ is \emph{strongly $F$-regular} if for all
non-zerodivisors $c\in R$ there exists $0<e\in\NN$ such that the
$R$-linear map
$R\to F^e_*R$ that sends $1$ to $F^e_*c$ splits as $R$-morphism.
\end{dfn}
Note that if $c=c'c''$ in $R$, and if $1\mapsto F^e_*c$ splits, then
so does $1\mapsto F^e_*c'$. Indeed, multiplication by $F^e_*c''$ on
$F^e_*R$ provides an $R$-linear endomorphism and it can be combined
with a splitting of $1\mapsto F^e_*c$ for a splitting of the map that
sends $1\mapsto F_*^ec'$. In particular, splitting any
$1\mapsto F^e_*c$ implies that the ring is $F$-split, hence $F$-pure.

For $F$-finite rings, strong $F$-regularity localizes, and can be
checked locally. Moreover, $F$-finite regular rings are strongly
$F$-regular, while in turn $F$-finite strongly $F$-regular rings are
normal domains.

The following appears already in \cite{Fedder}; we reproduce here the
proof from \cite[Claim 2.7]{MaPolstra}.

\begin{lem}\label{lem-dsum}
  Let $(S,\frakm,\FF)$ be a polynomial ring over an perfect 
  field $\FF$ with homogeneous maximal ideal $\frakm=\ideal{x_1,\ldots,x_n}$. Then
  $F_*S$ is the free $S$-module
  \begin{eqnarray}\label{eqn-dsum}
    F_*S=\bigoplus_{0\le u_1,\ldots,u_n< p}S\cdot F_*(\bsx^\boldu)
  \end{eqnarray} 
  in multi-index notation. For any such multi-index, let $\Phi_\boldu$
  be the 
   $S$-morphism $\Phi\colon F_*S\to S$ that identifies
  the summand $S\cdot F_*(\bsx^\boldu)$ with $S$ and which is
  zero on all other summands. Set
  \[
  (p-1,\ldots,p-1)=:\bsalpha.
  \]

  Let $I\subseteq \frakm$ be an ideal of $S$ and choose $s\in S$. Then
  the composition of $\Phi_\bsalpha$ following multiplication by
  $F_*s$ on $F_*S$ induces a morphism from $F_*(S/I)$ to $S/I$
  precisely when $sI\subseteq I^{[p]}$.
\end{lem}
\begin{proof}
  If $sI\subseteq I^{[p]}$ then
  $\Phi_\boldu(F_*s\cdot F_*I)=\Phi_\boldu(F_*(sI))\subseteq
  \Phi_\boldu(F_*I^{[p]})=\Phi_\boldu(IF_*(RS)=I$ for all $\boldu$. In
  the other direction, suppose $sI\not\subseteq I^{[p]}$ and pick
  $r'\in I$ with $r:=sr'\notin I^{[p]}$. Decompose $F_*r$ according to
  the direct sum decomposition \eqref{eqn-dsum}. One of the components
  $r_\boldu F_*(\bsx^\boldu)$ of $F_*r$ is outside $F_*(I^{[p]})=IF_*S$,
  which implies that $r_\boldu$ is not in $I$. Then the corresponding
  projection $\Phi_\boldu$ sends this component $r_\boldu\cdot
  F_*\bsx^\boldu$ to the element $r_\boldu$ outside $I$ in $S$, and
  kills all other components. It follows that
  $r_\boldu=\Phi_\bsalpha(F_*r\bsx^{\bsalpha-\boldu})$ is outside
  $I$. As $r=sr'\in I$, the map $\Phi_\bsalpha(F_*s\cdot(-))\colon
  F_*S\to S$ does not send $F_*I$ to $I$, and hence does not induce a
  morphism $F_*(S/I)\to S/I$.
\end{proof}

\subsection{A Primer on Jets}
%%%%%%%%%%%%%%%%%%%%%%%%%%%%%%%%%%%%%%%%%%%%

\begin{ntn}
Whenever we deal with jets of varieties, the underlying field will be called $\KK$ and it will be assumed to be algebraically closed of characteristic zero.
\end{ntn}

Consider an affine scheme 
\[
X \subseteq \KK^n:=\mathbb{A}_{\KK}^{n},
\]
embedded in affine $n$-space over $\KK$.  We give a working definition of the $m$-jets of $X$ that is sufficient for our needs. A more intrinsic definition as well as deeper topics in jet schemes and arc spaces can be found in  \cite{SingularitiesPairsMustata, MustataJetsLCI}.

In light of the embedding $X\into\KK^n$, the $m$-jets of $X$ will appear as subschemes of the $m$-jets of $\KK^n$, and so we
begin with the $m$-jets of $\KK^{n}$:
\begin{define} \label{def - m-jets of affine}
Write the coordinate ring of $\KK^{n}$ as $\KK[x_{1}, \dots, x_{n}]$. For each $q \in \mathbb{Z}_{\geq 0}$ introduce new variables $x_{i}^{(q)}$. With the standard identification  $x_{i} = x_{i}^{(0)}$, the \emph{$m^{\text{th}}$-jet scheme of $\KK^{n}$} is
\[
\scrL_{m}(\KK^{n}) := \Spec \KK[\{x_{1}^{(q)}, \dots, x_{n}^{(q)}\}_{0 \leq q \leq m}].
\]
\end{define}

To construct the jet schemes of $X \subseteq \KK^{n}$, first define a universal $\KK$-linear derivation $D$ on $\KK[\{x_{1}^{(q)}, \dots, x_{n}^{(q)}\}_{0 \leq q \leq \infty}]$ by the laws: 
\begin{itemize}
\item Leibniz rule: $D(gh)=gD(h)+(Dg)h$; \item Prolongation: $Dx_{i}^{(q)} = Dx_{i}^{(q+1)}$.
\end{itemize}
With the understanding that the 0-fold iteration $D^0$ is to be the identity map, we now introduce the following concept. 
\begin{dfn} \label{dfn-jets-der}
    Let $I \subseteq \KK[x_{1}, \dots, x_{n}]$ be the defining ideal of an affine scheme $X \subseteq \KK^{n}$, generated by $f_1,\ldots,f_r$. 
    With the standard identification $x_i^{(q)}=x_i$, the  \emph{$m$-th jet scheme of $X$} is 
    %denoted by
    %\[
    %\scrL_{m}(\KK^{n}, X),
    %\]
    %and it is 
    the affine scheme 
    %with defining ideal
    \[
    \scrL_{m}(\KK^{n}, X) := \Var(\{D^{q}f_{1}, \dots, D^{q}f_{r}\}_{0 \leq q \leq m}) \subseteq %\KK[x_{1}^{(q)}, \dots, x_{n}^{(q)}]_{0 \leq q \leq m}.
    \scrL_m(\KK^n),
    \]
    and in particular
    \[
    \scrL_{0}(\KK^{n}, X) = X. 
    \]
    These are independent of the generator choice for $I$ since $D$ is a derivation. We agree that $\scrL_{-1}(\KK^n,X)$ is simply $\KK^n$.
\end{dfn}
    For $-1 \leq t < m$ we have a natural projection map
    \begin{eqnarray} \label{def - affine jets, projection map}
        \pi_{m, t}: \scrL_{m}(\KK^{n}) &\xtwoheadrightarrow{\phantom{xxx}}& \scrL_{t}(\KK^{n})\\
    (x_{1}^{(q)}, \dots, x_{n}^{(q)})_{0 \leq q \leq m} &\xmapsto[]{\pi_{m,t}}& (x_{1}^{(q)}, \dots, x_{n}^{(q)})_{0 \leq q \leq t}.    \notag
    \end{eqnarray}
    given by forgetting the $x_{i}^{(q)}$-coordinates for $1 \leq i \leq n$ and $t+1 \leq q \leq m$.
We can sort the $m$-jets of $X \subseteq \KK^{n}$ based on their image under $\pi_{m,0}$:

\begin{define} \label{def - m-jets lying over}
    Let $Z \subseteq \KK^{n}$ be Zariski locally closed. We say that an \emph{$m$-jet $\gamma \in \scrL_{m}(\KK, X)$ lies over $Z$} if it is an element of
    \[
    \scrL_{m}(\KK^{n}, X, Z) := \{\gamma \in \scrL_{m}(\KK^{n}, X) \mid \pi_{m,0}(\gamma) \in Z \}.
    \]
    We convene that $\scrL_{-1}(\KK^n,X,Z)$ is $Z$; this is in harmony with the earlier convention $\scrL_{-1}(\KK^n,X)=\KK^n$.
%    This has defining ideal 
%    \[
%    I(\scrL_{m}(X,Z) = I(\scrL_{m}(X)) + I(Z) \subseteq \KK[x_{1}^{(q)}, \dots, x_{n}^{(q)}]_{0 \leq q \leq n}
%    \]
%    where $I(Z) \subseteq \KK[x_{1}, \dots, x_{n}] = \KK[x_{1}^{(0)}, \dots, x_{n}^{(0)}]$ is the defining ideal of $Z$. 
\end{define}

\begin{remark}
    Whenever there is ambiguity, we will use a superscript  to clarify what projection we are considering. For example, $\pi_{m,t}^{\KK^{n}}: \scrL_{m}(\KK^{n}) \to \scrL_{t}(\KK^{n})$ is the projection $\pi_{m,t}: \scrL_{m}(\KK^{n}) \to \scrL_{t}(\KK^{n})$. 
\end{remark}

\begin{remark} \label{rmk - always a m-jet lying over a point in X}
Consider a monomial $x_{i_{1}} \cdots x_{i_{k}} \in \KK[x_{1}, \dots, x_{n}]$. Then \begin{align}\label{eqn-D-m+1}
    D^{m} (x_{i_{1}} \cdots x_{i_{k}} ) 
    &= \sum_{j_{1} + \cdots + j_{k} = m} \binom{m}{j_{1}, \dots, j_{m}} (D^{j_{1}}x_{i_{1}}) \cdots (D^{j_{k}}x_{i_{k}}) 
\end{align}    
is in the ideal $(\{x_{1}^{(q)}, \dots, x_{n}^{(q)}\}_{1 \leq q \leq m})$ of $\KK[\{x_{1}^{(q)}, \dots, x_{n}^{(q)}\}{_{0 \leq q \leq m}}]$.
(The point to note is the lower bound of the $q$-indices for the ideal).
Since $D$ is $\KK$-linear and additive, we quickly deduce:
\begin{equation} \label{eqn - m-jets fiber over alpha in X is nonempty}
        [\alpha \in X = \scrL_{0}(\KK^{n}, X)] \implies [(\alpha, 0, \dots, 0) \in \scrL_{m}(\KK^{n}, X)].
\end{equation}
In particular,  if $\alpha \in X$, then $\pi_{m,0}^{-1}(\alpha) \cap \scrL_{m}(\KK^{n}, X) \neq \emptyset.$
\end{remark}

A point of $\scrL_{m}(\KK^{n}, X)$ either lies over (with respect to $\pi_{m,0})$ the smooth or singular part of $X$. This gives a decomposition into Zariski closed sets:
\[
\scrL_{m}(\KK^{n}, X) = \scrL_{m}(\KK^{n}, X, X_{\Sing}) \cup \overline{\scrL_{m}(\KK^{n}, X, X_{\reg})}.
\]
One checks that the projection $\pi_{m,0}: \scrL_{m}(\KK, X, X_{\reg}) \to X_{\reg}$ is a locally trivial fibration with fibers $\KK^{m \cdot\dim (X)}$ and so 
\[
\dim(\overline{\scrL_{m}(\KK^{n}, X, X_{\reg})})=(m+1)\dim(X).
\]
Moreover, it turns out that $\overline{\scrL_{m}(\KK^{n}, X, X_{\reg})}$ is an irreducible component of $\scrL_{m}(\KK^{n}, X)$. In the case $X$ is a locally complete intersection (hereafter denoted l.c.i) this leads to the straightforward but powerful result:

\begin{proposition} \label{prop - Mustata's characterization of irreducibility of m-jets in terms of dim lying over sing} 
\cite[Prop.'s~1.4,~1.5,~1.7]{MustataJetsLCI}
    Let $X$ be an l.c.i.\ and choose $m \in \mathbb{Z}_{\geq 1}$. Then
    \begin{enumerate}[label=(\alph*)]
        \item $[\scrL_{m}(\KK^{n}, X)\text{ is equidimensional}\,]$ $\iff [\dim \scrL_{m}(\KK^{n}, X) \leq (m+1)\dim X]$, and this happens exactly when equality holds;
        \item $[\scrL_{m}(\KK^{n}, X)\text{ is equidimensional}\,]$ $\implies$ $[\scrL_{m}(\KK^{n}, X)\text{ is an l.c.i.}\,]$;
        \item $[\scrL_{m}(\KK^{n}, X)\text{ is irreducible}\,]$ $\iff [\dim \scrL_{m}(\KK^{n}, X, X_{\Sing}) < (m+1)\dim X]$;
        \item $[\scrL_{m}(\KK^{n}, X)\text{ is irreducible}\,]$ $\implies$ $[\scrL_{m}(\KK^{n}, X)\text{ is reduced and $X$ is normal}\,]$.
    \end{enumerate}
\end{proposition}

We now arrive at this section's main tool, the principal result of \cite{MustataJetsLCI}. This lets us, in certain situations, equate rationality in terms of irreduciblity of $m$-jets. We provide no explanation for his proof other than that motivic integration interprets irreducible component data of $\scrL_{m}(\KK^{n}, X)$ into data of a log resolution.

\begin{thm}[{\cite[{Thm.~3.3}]{MustataJetsLCI}}]\label{thm-mustata-main}
    Let $X$ be an l.c.i. Then $\scrL_{m}(\KK^{n}, X)$ is irreducible for all $m \geq 1$ if and only if $X$ has rational singularities.
\end{thm}

\begin{rmk}
\label{rmk-rat-sing-properties}
While one could use the previous theorem as a definition for rational singularities, the original definition is cast in terms of a resolution of singularities: the variety $X$ has \emph{rational singularities} if and only if for some  resolution of singularities
$Y\stackrel{\pi}{\to} X$, one has
\[
R\pi_*(\scrO_Y)=\scrO_X.
\]
If one resolution has this property, then every does. In practice, this is a complicated definition to work with, but the property has interesting consequences. For example, the \emph{log-canonical threshold} $\lct(\frakx,\Var(f))\subseteq \CC^n$ of a hypersurface singularity $f(\frakx)=0$ is the number 
\[
\lct(\frakx,X):=\sup\{c\in\RR\mid 1/||f||\text{ is locally $L^2$-integrable at }\frakx\},
\]
and $\frakx\in \Var(f)$ is a point near which $X$ has rational singularities only if $\lct(\frakx,\Var(f))$ equals the maximum possible value 1. Moreover, $-1$ exceeds the largest root of the local reduced Bernstein--Sato polynomial $b_{f,\frakx}(s)/(s+1)$ attached to $f$ near $p$ if and only if $\Var(f)$ has rational singularities near $\frakx$ \cite[Thm.~0.4]{SaitoOnBFunction}. See \cite{Kollar} for more details. 
\end{rmk}

We conclude with three elementary lemmas that we make frequent use of. 

\begin{lemma} \label{lemma - intersecting irreducible jet scheme with inverse image of avoidant variety downstairs}
    Consider $f, g \in \KK[X] = K[x_{1}, \dots, n]$, with $n \geq 2$, and let $X_{f}$ be the hypersurface corresponding to $f$. Suppose that for $m \in \mathbb{Z}_{\geq 1}$ we have that:
    \begin{enumerate}[label=(\alph*)]
        \item $\scrL_{m}(\KK^{n}, X_{f})$ is irreducible, necessarily of dimension $(m+1)(n-1)$;
        \item $\dim (\Var(f) \cap \Var(g)) \leq n - 2$.
    \end{enumerate}
    Then 
    \[
    \dim \scrL_{m}(\KK^{n}, X_{f}, \Var(g)) < (m+1)(n-1).
    \]
\end{lemma}

\begin{proof}
    Let $C$ be an irreducible component of $\Var(g)$. Since $\scrL_{m}(\KK^{n}, X_{f})$ is irreducible of dimension $(m+1)(n-1)$, the claim amounts to checking that intersecting $\scrL_{m}(\KK^{n}, X_{f})$ with $\pi_{m,0}^{-1}(C)$ causes a dimension drop. Since $\scrL_{m}(\KK^{n}, X_{f})$ is irreducible, it suffices to show that $\scrL_{m}(\KK^{n}, X_{f}) \not\subseteq \pi_{m,0}^{-1}(C)$. By assumption (b) we may find $\alpha \in \Var(f) \setminus C$. By Remark \ref{rmk - always a m-jet lying over a point in X}, $(\alpha, 0, \dots, 0) \in \scrL_{m}(\KK^{n}, X_{f}) \setminus \pi_{m}^{-1}(C)\neq \emptyset$. 
\end{proof}

\begin{lemma} \label{lemma - jet schemes of products, not nec disjoint variables} Suppose $f, g \in \KK[x_{1}, \dots, x_{n}]$.  If $\mathfrak{p}$ is a prime ideal containing $I(\scrL_{m}(\KK^{n}, \Var(fg)))$ then there exists $-1 \leq t \leq m$ such that
\[
\mathfrak{p} \ni \{D^{a} f \}_{0 \leq a \leq t} \quad \text{ and } \quad \mathfrak{p} \ni \{D^{b} g \}_{0 \leq b \leq m-t - 1}.
\]
(When $t = -1$, $\{D^{a} f\}_{0 \leq a \leq t} = \emptyset$.) 
\end{lemma}

\begin{proof}
    We induce on $m$ with the base case $m=0$ obvious. Suppose $\mathfrak{p}$ is a minimal prime of $I(\scrL_{m+1}(\Var(fg)))$. Then it certainly contains $I(\scrL_{m}(\Var(fg)))$, and by the inductive hypothesis there is $t\in\{-1,...,m\}$ such that
    \[
    \mathfrak{p} \ni \{D^{a} f \}_{0 \leq a \leq t} \quad \text{ and } \quad \mathfrak{p} \ni \{D^{b} g \}_{0 \leq b \leq m - t-1}.
    \]
%    Moreover,
%    \begin{align*}
%        D^{m+1} (fg) 
%        &= \sum_{\ell \in [m+1]} \binom{m+1}{\ell} (D^{\ell}f) D^{m+1 - \ell} g) \\
%        &\equiv \binom{m+1}{t+1} (D^{t+1} f) (D^{m-t} g) \mod (\{D^{a} f \}_{0 \leq a \leq t}, \{D^{b} g \}_{0 \leq b \leq m - t-1}).
%    \end{align*}
    Since $\mathfrak{p} \ni D^{m+1}(fg)$, the claim follows with Equation \eqref{eqn-D-m+1}.
\end{proof}

\begin{remark}\label{rmk - jets of arrangements}
    Hypersurfaces $X \subseteq \KK^n$ with explicit descriptions of $\scrL_m(\KK^n, X)$ are scarce. When $X$ is the Boolean arrangement cut out by $x_1 \cdots x_k \in \KK[x_1, \dots, x_n]$, its $m$-jets are equidimensional of dimension $(m+1)(n-1)$. Moreover, the irreducible components of $\scrL_m(\KK^n, X)$ are
    \[
    \{x_1^{(0)} = \cdots = x_1^{(v_1)} = 0\} \cup \cdots \cup \{x_n^{(0)} = \cdots = x_k^{(v_{k})} = 0\} \subseteq \scrL_m(\KK^n)
    \]
    where: $\mathbf{v} = (v_{1}, \dots, v_{k})$ ranges over all integral vectors in $\mathbb{Z}_{\geq -1}^{k}$ such $|\mathbf{v}| = v_{1} + \cdots + v_{k} = m+1$; we convene $x_{i}^{(-1)} = 0.$ See this by inducing on $m$ or using the formula for $m$-jets of a hyperplane arrangement \cite{BudurTueArrangementsJets}.
\end{remark}

%%%%%%%%%%%%%%%%%%%%%%%%%%%%%%%%%%%%%%%%%%%%%%%%%%%%%%%%%%
\section{The $F$-singularities of Matroid Support Polynomials}
%%%%%%%%%%%%%%%%%%%%%%%%%%%%%%%%%%%%%%%%%%%%%%%%%%%%%%%%%%
\label{sec-Fsing}
\begin{ntn}
In this section the underlying field is denoted $\FF$ and stands for a field of finite characteristic $p>0$. 
\end{ntn}
We focus here exclusively on the class of homogeneous matroidal polynomials, the matroid support polynomials. We will prove that the are strongly $F$-regular. Naturally, the main tool is the Frobenius morphism; compare Subsection \ref{subsec-primer-F} for background.

\subsection{The Glassbrenner Criterion revisited}
%%%%%%%%%%%%%%%%%%%%%%%%%%%%%%%%%%%%%%%%%%%%%%%%%%

The following is a modification of Glassbrenner's criterion
\cite{Glassbrenner}, which was introduced for the case $k=1$. 

\begin{lem}\label{lem-better-glass}
  Suppose $R$ is an $F$-finite ring of characteristic $p>0$. Assume
  that $c\in R$ is a non-zerodivisor for which $R_c$ is strongly
  $F$-regular. Suppose $c$ factors as $c=a_1\cdots a_k$. Then $R$ is
  strongly $F$-regular if and only if 
  there exists $e_a\in\NN$ such that for
  each $i\in\{1,2,\ldots,k\}$ the $R$-morphism $R\to F^{e_a}_*R$ sending
  $1\mapsto F^e_*a_i$ splits.
\end{lem}
\begin{proof}
  The ``only if'' part follows from the original Glassbrenner
  criterion, in conjunction with the paragraph following Definition
  \ref{dfn-Freg}. 
  
  The proof of the ``if'' part is modeled after
  \cite[Thm.~3.11]{MaPolstra}. It follows immediately from the
  discussion following Definition \ref{dfn-F-pure} that $R$ is $F$-pure (and since it is $F$-finite also that it is $F$-split).
  
  Choose $d\in R$. As $R_c$ is strongly $F$-regular and by hypothesis, there exists $e_0$
  such that a) $R\to F_*^{e_0}R$ with $1\mapsto F_*^{e_0}a_i$ splits for all $i$, and b) there is an $R_c$-module map $\pi_c\colon F^{e_0}_*R_c\to
  R_c$ that sends $F^{e_0}_*(d/1)$ to $1/1$. Since $R$ is $F$-finite,
  $\Hom_{R_c}(F^{e_0}_*R_c,R_c)=R_c\otimes \Hom_R(F^{e_0}_*R,R)$ and
  so there is an $R$-module map $\pi\colon F^{e_0}_*R\to R$ that sends
  $F^{e_0}_*d$ to $c^n$ for some $n\in\NN$. Since 
  $R$ is $F$-pure we also have, for all $e\in \NN$,
  $R$-morphisms $F^{e+e_0}_*R\to R$ sending $F^{e+e_0}_*d\mapsto c^n$.

  One now shows that for large $e$ the exponent $n$ can be removed, and one factor may be peeled away
  from the product $c=a_1\cdots a_k$.  
  For this, choose $e_1\in\NN$ such that $n<p^{e_1-e_0}$.
%  ; then $F^{e_1-e}$ applied to $F^e_*R$ sends $F^e_*c$ to
%  $F^{e_0}_*c^{p^{e_1-e_0}}=F^{e_0}_*c^n\cdot
%  F^{e_0}_*c^{p^{e_1-e_0}-n}$. Moreover, a
  Applying $F^{e_1}$ to $\pi$ gives a morphism $F^{e_0+e_1}_*R\to
  F^{e_1}_*R$ that sends $F^{e_1+e_0}_*d$ to
  $F^{e_1}_*c^n=F^{e_1}_*(a_1)^n\cdot F^{e_1}_*(a_2\cdots
  a_k)^n$. Composing it with the endomorphism $\mu_1$ on $F^{e_1}_*R$ that
  multiplies by $F^{e_1}_*(a_1)^{p^{e_1-e_a}-n}\cdot
  F^{e_1}_*(a_2\cdots a_k)^{p^{e_1}-n}$ on $F^{e_1}_*R$ yields, for
  all large $e_1$, an $R$-morphism $F^{e_0+e_1}_*R\to F^{e_1}_*R$ that
  sends $F^{e_1+e_0}_*d$ to $F^{e_1}_*(a_1)^{p^{e_1-e_0}}\cdot
  F^{e_1}_*(a_2\cdots a_k)^{p^{e_1}}$.

  Recall that we chose $e_0$ to make $R\to F^{e_0}_* R$ with $1\mapsto
  F^{e_0}_*a_1$ split. Composition with $F^{e_1-e_0}$ for $e_1>e_0$
  gives a morphism $R\to F^{e_0}_* R\to F^{e_1}_*R$ sending $1\mapsto
  F^{e_0}_*a_1\mapsto F^{e_1}_*(a_1)^{p^{e_1-e_0}}$. This map splits
  since it is a composition of split maps. Let $\theta_1\colon
  F^{e_1}_*R\to R$ be a splitting; then $\theta$ sends
  $F^{e_1}_*(a_1)^{p^{e_1-e_0}}\cdot F^{e_1}_*(a_2\cdots a_k)^{p^{e_1}}$
  to $(a_2\cdots a_k)\in R$.

  Thus, we have for sufficiently large $e_1$ an $R$-morphism
  $\pi_2\colon F^{e_1+e_0}_*R\stackrel{F^{e_1}\pi}{\to} F^{e_1}_*R\to
  F^{e_1}_*R\stackrel{\theta}{\to} R$ that sends
  $F^{e_1+e_0}_*d\mapsto F^{e_1}_*c^n\stackrel{\mu_1}{\mapsto}
  F^{e_1}_*a_1^{p^{e_1-e_0}}\cdot F^{e_1}_*(a_2\cdots
  a_k)^{p^{e_1}}\mapsto (a_2\cdots a_k)$.

  Repetition of the argument produces a morphism 
%  The proof follows now by induction on $k$ along the lines above:
%  apply $F^{e_1}_*$ to $\pi_2$ and compose it with multiplication by
%  $F^{e_1}_*(a_3\cdots a_k)^{p^{e_1}-1}$ on $F^{e_1}_*R$ to get a map
%  $F^{2e_1+e_a}_*R\to F^{e_1}_*R$ that sends $F^{2e_1+e_0}_*d$ to
%  $F^{e_1}_*a_2\cdot F^{e_1}_*(a_3\cdots a_k)^{p^{e_1}}$; use a
%  splitting $\theta_2\colon F^{e}_*R\to R$ that sends $F^{e}_*a_2$ to
%  $1$ and consider $F^{e_1-e_a}_*\theta_2$ which sends
%  $F^{e_1}_*a_2\cdot F^{e_1}(a_3\cdots a_k)^{p^{e_1}}\mapsto
%  (a_3\cdots a_k)$. Compose with $\pi_2$ and to get $\pi_3\colon
  % F^{2e_1+e_0}_*R\to R$ \emph{etc}. The map $\pi_{k+1}\colon
  $F^{ke_1+e_0}_*R\to R$ that sends $F^{ke_1+e_0}_*d$ to $1$.
\end{proof}

Combining Lemmas \ref{lem-better-glass} and \ref{lem-dsum}, we have
\begin{prop}
    Suppose $S$ is a polynomial ring over an $F$-finite field $\FF$ of
    characteristic $p>0$. Let $I\subseteq S$ be a homogeneous ideal of
    $S$ and set $R:=S/I$. Assume that $c\in R$ is a homogeneous
    non-zerodivisor for which $R_c$ is strongly $F$-regular. Suppose
    $c$ factors as $c=a_1\cdots a_k$, all $a_i\in\frakm$. Then $R$ is
    strongly $F$-regular if and only if there exists $e_a\in\NN$ such
    that for each $i\in\{1,2,\ldots,k\}$ we have
    $a_i(\bracket{I}{e_a}:I)\not\subseteq \bracket{\frakm}{e_a}$.
\end{prop}
\begin{proof} Abbreviate $e:=e_a$.
  With assumptions as made, and in view of Lemma
  \ref{lem-better-glass}, we need to show that the $R$-linear maps
  $1\mapsto F^{e}_*a_i$ split if and only if
  $a_i(I^{[p^{e_a}]}:I)\not\subseteq \frakm^{[p^{e_a}]}$. Note that
  (by homogeneity) it suffices to show that the local ring $R_\frakm$
  is strongly $F$-regular.

  Assume first that $\FF$ is perfect. 
  If $a_i(I^{[p^{e_a}]}:I)\not\subseteq \frakm^{[p^{e_a}]}$, choose
  $s\in (I^{[p^{e_a}]}:I)$ for which $a_i s\notin\frakm^{[p^{e_a}]}$.
  Then Lemma \ref{lem-dsum} implies that multiplication by
  $F^{e_a}_*s$ on $F^{e_a}_*S$ when followed by $\Phi_\bsalpha$ gives a
  map from $F^{e_a}_*(S/I)$ to $S/I$, and this map sends $F^{e_a}_*(a_i)$ to
  $\Phi_\bsalpha(F^{e_a}_*(a_is))\notin \frakm$. Localizing at the two maximal
  ideals involved, a suitable adjustment of $s$ will therefore send
  $F^{e_a}_*(a_i)$ to $1\in R_\frakm$.

  Conversely, note that $\Hom_S(F^e_*S,S)$ is identified with the
  morphisms that arise by composing multiplication on $F^e_*S$ by some
  $F^e_*s$ with $\Phi_\bsalpha$. Hence, any $R$-morphism from
  $F^e_*(R)=F^e_*S/F^e_*I$ to $R=S/I$ is induced by some $F^e_*s$ with
  $s\in (I^{[p^e]}:I)$. If all such $s$ multiply $a_i$ into
  $\frakm^{[p^e]}$ then any such morphism sends $F^e_*a_i$ into
  $\frakm\subseteq R$, and then no splitting for $1\mapsto F^e_*a_i$
  can exist.

  Now suppose $\FF$ is only $F$-finite, and let $\bar\FF$ be a perfect
  field containing $\FF$. Extending scalars to $\bar\FF$ is a
  faithfully flat process. Thus $a_i(\bracket{I}{e_a}:I)\subseteq
  \bracket{\frakm}{e_a}$ is true over $\FF$ if and only it is so over
  $\bar\KK$. On the other hand, $\bar \FF\otimes_\FF R$ is strongly $F$-regular
  if and only $\bar \FF\otimes R$ is. Since the statement holds over
  $\bar \FF$, it also holds over $\FF$, by
  \cite[Cor.~2.13]{HochsterYao23} or \cite[Cor~4.4]{LySmith99}.
\end{proof}

\subsection{Strong $F$-regularity in the Homogeneous Case}
%%%%%%%%%%%%%%%%%%%%%%%%%%%%%%%%%%%%%%%%5

\begin{thm}\label{thm-strongFreg}
  Let $\matM$ be a connected matroid on the ground set $E$.
  
  Then any matroid support polynomial $\zeta_\matM$ on $\matM$ is, in
  every characteristic $p>0$, strongly $F$-regular. 
\end{thm}
\begin{proof}
  Let $S_\matM=\FF[\bsx_E]$ be the matroidal polynomial ring and set
  $R_\matM=S_\matM/\ideal{\zeta_\matM}$.  Since $\zeta_\matM$ is
  homogeneous, $R_\matM$ is strongly $F$-regular if and only if the
  local ring at the origin is strongly $F$-regular. Since $\zeta_\matM$ is a matroidal  polynomial on a connected matroid, it is irreducible over any field and hence geometrically irreducible.

  \medskip
  
  The strategy of the proof of strong $F$-regularity  of $\zeta_\matM$ will be
  to show that the truth of the statement for a complicated connected
  matroid can be reduced to the truth for a simpler connected matroid,
  and then check base cases. Simplicity is measured by the size of the
  ground set. The reduction comes in two flavors, an
  easy one relying on deformation, and a more elaborate one using
  Cremona transforms.  In order to determine the necessary base cases
  we will use handles.

  We can immediately dispose of the case of a matroid of rank one: connectedness forces it to be a $\matU_{1,n}$, and we have discussed the structure of the corresponding matroidal polynomials in Remark \ref{rmk - basic matroid polynomial observations}.\eqref{item-U1n-mtrdl}. We thus assume from now on that $\rank(\matM)\geq 2$.

  We start with a preparatory lemma that paves the way for using
  the modified Glassbrenner Criterion, Lemma \ref{lem-better-glass}. 
  \begin{lem}\label{lem-Fpure}
    For every edge $e\in\matM$ that is not a coloop, 
    \[
    x_e(\zeta_\matM)^{p-1}\notin\bracket{\frakm}{1}.
    \]
    In particular, if $\matM$ permits a non-coloop, all matroid
    support polynomials are $F$-pure.
  \end{lem}
  \begin{proof}
    Fix $e\in E$ that is not a coloop, and choose a basis $B$ of
    $\matM$ that does not involve $e$; such basis exists since
    otherwise $e$ would be a coloop. If $P_\matM$ is the matroid
    polytope (=support polytope) of $\matM$, then the indicator vector
    $\boldv_B$ is the unique optimal solution to the problem
    \begin{eqnarray*}
      \boldv^T\cdot \boldv_B&\to&\max,\\
      \boldv&\in&P_\matM.
    \end{eqnarray*}
    Indeed, since the vertices of $P_\matM$ have vertices inside
    the plane $\sum_Ex_e=\rank(\matM)$, the optimal solution has
    $\boldv^T\cdot \boldv_B\le\rank(\matM)$, and inside
    $P_\matM$ only $\boldv_B$ realizes this upper bound. Thus, for
    $k\in\NN$, $k\cdot \boldv_B$ is the unique optimal solution to
    \begin{eqnarray*}
      \boldv^T\cdot \boldv_B&\to&\max,\\
      \boldv&\in&kP_\matM.
    \end{eqnarray*}
    Since $kP_\matM$ is the support polytope of
    $(\zeta_\matM)^k$, the term $\bsx^{(p-1)B}$ appears in
    $(\zeta_\matM)^{p-1}$ with coefficient 1, and thus
    $x_e\bsx^{(p-1)B}$ appears in $x_e(\zeta_\matM)^{p-1}$ with
    coefficient 1 and is
    (clearly) not in $\bracket{\frakm}{1}$.
  \end{proof}

  For our first reduction, recall that strong $F$-regularity deforms
  in good situations. More precisely, suppose $R$ is a local
  $F$-finite Gorenstein local ring with a non-zerodivisor $x\in R$. If
  $R/xR$ is strongly $F$-regular, then so is $R$, \emph{cf.}\
  \cite[Thm.~5.11]{MaPolstra}.

  Suppose $\matM$ has an edge $e$ such that the deletion matroid
  $\matM\setminus e$ is connected. Then Deletion-Contraction yields
  \[
  \zeta_\matM=x_e\cdot\zeta_{\matM/ e}+\zeta_{\matM\setminus e},
  \]
  and we have 
  $R_{\matM\setminus e}=R_\matM/\ideal{x_e}$. Thus, we can
  infer strong $F$-regularity of $R_{\matM}$ from
  $R_{\matM\setminus e}$ based on the deformation property.

  Unfortunately, not every matroid has an edge $e$ that keeps the
  deletion connected. Our second reduction uses duals, see Definition \ref{dfn-dual-matroid}.
  If $e$ is a loop of $\matM$ then it appears in no basis of
  $\matM$, hence in every basis of $\matM^\perp$ and thus is a
  coloop there. Our connectedness assumption implies that neither
  $\matM$ nor its dual have loops or coloops.

  Suppose $B'$ is a basis for the deletion matroid
  $\matM^\perp\setminus e$. Let $B=E\minus B'$ be its
  complement. By definition, if $e$ is not a coloop in
  $\matM^\perp$, we have
  \[
  (\matM/ e)^\perp= (\matM^\perp)\setminus e
  \]
  since the bases for both displayed matroids are labeled by the 
  bases for $\matM^\perp$ that do not involve $e$.

  Recall now that the
  \emph{Cremona transform} is the birational map on projective space that
  inverts all coordinates. 
  If $\zeta_\matM$ is a matroid support polynomial on $\matM$ then
\begin{equation}\label{eqn-on-torus}
  \zeta_{\matM^\perp}(\bsx):=\bsx^E\cdot \zeta_\matM(\bsx^{-1})
  \end{equation}
is a matroid support polynomial on $\matM^\perp$ since
  \[
  \zeta_{\matM}(\bsx)= \sum_{B\in\calB_\matM}c_B\bsx^B= \bsx^E\cdot
  \sum_{B\in\calB_\matM}c_B(\bsx^{-1})^{E\minus B}=
  \bsx^E\cdot\zeta_{\matM^\perp}(\bsx^{-1}).
  \]
  We view them (non-canonically) as living in the same polynomial ring $\KK[\bsx_E]$
  and note that the corresponding hypersurfaces agree on the open
  torus $\Spec(\KK[\bsx_E^{\pm}])$.
  
  By Lemma \ref{lem-Fpure} above, for each $e\in E$ we have the
  relation $x_i(\zeta_\matM)^{p-1}\notin \bracket{\frakm}{1}$, and so
  (in particular) the morphism $R_\matM\to F_*R_\matM$ sending $1$ to
  $F_*x_i$ splits. By Lemma \ref{lem-better-glass} then, the
  hypersurface $\zeta_\matM$ will be strongly $F$-regular if it is so
  on the big torus. That in turn will follow if the dual hypersurface
  $\zeta_{\matM^\perp}$ is strongly $F$-regular. And this can be
  inferred from the deformation property provided that we can find an
  edge $e$ in $\matM^\perp$ for which the deletion matroid
  $\matM^\perp\setminus e$ is connected.

  As the deletion matroid of $\matM^\perp$ along $e$ is the
  contraction matroid along $e$ of $\matM$, we now know that
  $\zeta_\matM$ will be strongly $F$-regular whenever $\matM$ admits an
  edge $e$ for which either $\matM/ e$ or $\matM\setminus e$ is
  connected. We thus need to study the matroids for which no such
  edges exist; for this we shall use the following tools.

%  The coloops are the edges that are inside no circuit whatever, and
%  so the set of all coloops is a maximal handle. Since we assume
%  connectedness of $\matM$, no coloops exist, and so the union of
%  all circuits is $E$.
%  Consider all possible sets of the form
%  $(\bigcap_I C_i)\minus (\bigcup_J C_j)$ where $I,J$ are sets of
%  circuits for $\matM$. Then any nonempty such intersection is, by
%  definition, a (maximal) handle, and $E$ is their (disjoint) union.

  We call a filtration
  \[
  \calC_\matM\ni F_1\subsetneq \cdots\subsetneq F_k=E
  \]
  a \emph{handle decomposition} if the deletions
  $\matM\setminus(E\minus F_i)$ are connected for all $i$, and each
  $F_i\minus F_{i-1}$ is a handle of $\matM\setminus(E\minus F_i)$
  for $2\le i\le k$. Connected matroids have handle decompositions,
  and for each circuit $C$ of $\matM$ there is a handle decomposition
  that has that circuit $C$ as $F_1$, see \cite[Prop.~2.8]{DSW}.

%  We shall need the following elementary facts about handles, see
%  \cite[Lemma 2.4]{DSW}.
%  \begin{asparaenum}
%    \item If $H$ is a handle with $|H|>1$ and $e,e'\in H$ then $e'$ is
%      a coloop in $\matM\minus \{e\}$. In particular,
%      $\matM\minus\{e\}$ is disconnected.
%    \item Any handle $H$ is an independent set in $\matM$---unless
%      $\matM$ is just one circuit, in which case there is only
%      handle, the circuit $E$.
%  \end{asparaenum}
  
  Suppose a matroid has a handle decomposition where $E\neq F_1$. Let
  $H$ be the top handle $F_k$. If $|H|=1$, deleting the element of $H$
  leads to a connected matroid $\matM\setminus H$ by definition of handle
  decompositions. In that case, the deformation property implies that
  $\zeta_\matM$ is strongly $F$-regular provided that
  $\zeta_{\matM\setminus e}$ is. If, on the other hand, $|H|\geq 2$ then the
  contraction $\matM/ e$ is connected by
  \cite[Lemma~2.4]{DSW}. Since $\matM/ e$ is the dual of
  $\matM^\perp\setminus e$, the latter matroid is connected. In that
  case, strong $F$-regularity of $\zeta_{\matM^\perp}$ is implied by
  that of $\zeta_{\matM^\perp\setminus e}$ and then implies strong
  $F$-regularity of $\zeta_\matM$ by Lemma \ref{lem-Fpure} and Lemma
  \ref{lem-better-glass} and Equation \ref{eqn-on-torus}. This refers
  in both cases from a connected matroid $\matM$ to a connected
  matroid with fewer edges.

  By induction on $|E|$ it remains to investigate the case when
  the handle decomposition has $F_1=E$. In that case, $E$ is a circuit
  and as long as there are at least three edges one may contract any
  one of them to arrive at a smaller (connected) circuit matroid. But for a
  circuit of size 2, any matroid support polynomial is linear, hence
  smooth and surely strongly $F$-regular. This covers the inductive
  start and the theorem is proved.
\end{proof}

%%%%%%%%%%%%%%%%%%%%%%%%%%%%%%%%%%%%%%%%%%%%%%%%%%%%%
\section{Jets of Matroidal Polynomials}
%%%%%%%%%%%%%%%%%%%%%%%%%%%%%%%%%%%%%%%%%%%%%%%%%%%%%
\label{sec-jets}

We turn here to the $m$-jets $\scrL_{m}(\KK^{E}, X_{\zeta_\matM})$ of the matroidal polynomial hypersurface $X_{\zeta_{\matM}}$. Our intermediate goal is to prove that these are irreducible under mild hypotheses. 
Musta\c t\u a characterized equidimensionality and irreducibility of the $m$-jets of an l.c.i $X$ (Proposition \ref{prop - Mustata's characterization of irreducibility of m-jets in terms of dim lying over sing}) in terms of the dimension of the $m$-jets and the dimension of the $m$-jets lying over the singular locus  of $X$. We will use such a dimension count in order to apply Theorem \ref{thm-mustata-main}, from which we derive our main result of this section, Corollary \ref{cor-mtrdl-rat-sing}. 

The idea is induce on the size of the ground set, and reduce studying $\scrL_{m}(\KK^{E}, X_{\zeta_\matM})$ to studying $\scrL_{m}(\KK^{E\minus\{e\}}, X_{\zeta_{\matM\minus \{e\}}})$. This is accomplished as follows:

\begin{proposition} \label{prop - matroidal poly, m-jets intersect with Gamma}
Let $\matM$ be a positive rank matroid on ground set $E$ equipped with a matroidal polynomial $\zeta_\matM$. Fix $e \in E$ that is not a loop nor coloop and let $\Gamma \subseteq \scrL_{m}(\KK^{E})$ be the variety defined by the ideal $(f_{e}, D f_{e}, \dots, D^{m} f_{e}) \subseteq \KK[\{x_{f}^{(q)}\}_{\substack{f \in E \\ 0 \leq q \leq m}}].$ Then we have a natural identification
\[
\scrL_{m}(\KK^{E}, X_{\zeta_\matM}) \cap \Gamma \simeq \scrL_{m}(\KK^{E \minus \{e\}}, X_{\zeta_{\matM\setminus e}}).
\]
\end{proposition}

\begin{proof}
    Recall that $(\zeta_\matM, \dots, D^{q} \zeta_\matM)$ is the defining ideal of $\scrL_{m}(X_{\zeta_\matM})$. Because $e$ is not a loop nor coloop, by Deletion-Contraction $\zeta_\matM = \zeta_{\matM\setminus e} + f_{e} \zeta_{\matM/ e}$. For $q \in \mathbb{Z}_{\geq 0}$ we compute:
    \begin{align*}
        D^{q} \zeta_\matM 
        &= D^{q} \zeta_{\matM\setminus e} + D^{q} (f_{e} \zeta_{\matM/ e}) \\
        &= D^{q} \zeta_{\matM\setminus e} + \sum_{0 \leq p \leq q} \binom{q}{p} (D^{p} f_{e})(D^{q-p} \zeta_{\matM/ e}). 
    \end{align*}
    It quickly follows that 
    \[
    (\zeta_\matM, \dots, D^{q} \zeta_\matM) + (f_{e}, \dots, D^{q} f_{e}) = (\zeta_{\matM\setminus e}, \dots, D^{q} \zeta_{\matM\setminus e}) + (f_{e}, \dots, D^{q} f_{e}),
    \]
    which proves the claim.
\end{proof}

 Our arguments rely on bounding the dimension of $m$-jets intersected with well chosen varieties. We make frequent use of the following lemma from elementary intersection theory:

\begin{lemma} \label{lemma - basic intersection theory identity} (Proposition I.7.1 \cite{HartshorneAlgGeo}) Let $Y, \Gamma \subseteq \KK^{n}$ be affine varieties with non-empty intersection. Then
\[
\dim (Y \cap \Gamma) + \codim \Gamma \geq \dim Y.
\]  
\end{lemma}

By combining Proposition \ref{prop - matroidal poly, m-jets intersect with Gamma} and Lemma \ref{lemma - basic intersection theory identity} we deduce equidimensionality:

\begin{proposition} \label{prop - matroid poly, equidimensional jets}
    Let $\matM$ be a positive rank matroid on $E$ with attached matroidal polynomial $\zeta_\matM$. Then $\scrL_{m}(\KK^{E}, X_{\zeta_\matM})$ is an equidimensional l.c.i of dimension $(m+1)(n-1)$ for all $m \geq 0$.
\end{proposition}

\begin{proof}
By Musta\c t\u a's Proposition \ref{prop - Mustata's characterization of irreducibility of m-jets in terms of dim lying over sing}, it suffices to show that 
\begin{equation} \label{eqn - matroid poly, m-jets equidimensional dimension bound}
\dim \scrL_{m}(\KK^{E}, X_{\zeta_\matM}) \leq (m+1)(n-1) \text{ for all } m \geq 1,
\end{equation}
the case $m=0$ being immediate. We will certify this dimension bound. Before giving the inductive proof (on $|E|$), we take preliminary steps:
    
\emph{Step 1: Inductive Step.} Suppose $e \in E$ is not a loop/coloop on $\matM$ and $\zeta_{\matM\setminus e}$ satisfies \eqref{eqn - matroid poly, m-jets equidimensional dimension bound}. Let $\Gamma \subseteq \scrL_{m}(\KK^{E})$ have defining ideal $(f_{e}, D f_{e}, \dots, D^{m} f_{e})$. By Proposition \ref{prop - matroidal poly, m-jets intersect with Gamma},
\[
\scrL_{m}(\KK^{E}, X_{\zeta_\matM}) \cap \Gamma \simeq \scrL_{m}(\KK^{E \minus \{e\}}, X_{\zeta_{\matM\setminus e}}).
\]
This is non-empty since $\rank(\matM\setminus e) = \rank(\matM) \geq 1$ and so $\zeta_{\matM\setminus e} \notin \KK^{\times}$, \emph{cf.}\ Remark \ref{rmk - basic matroid polynomial observations}.\eqref{item-mat-poly-degree}. Our dimension assumption on $\scrL_{m}(\KK^{E \minus \{e\}}, X_{\zeta_{\matM\setminus e}})$ plus Lemma \ref{lemma - basic intersection theory identity} gives
\begin{align*}
\dim \scrL_{m}(\KK^{E}, X_{\zeta_\matM}) 
    &\leq \dim (\scrL_{m}(\KK^{E}, X_{\zeta_\matM}) \cap \Gamma) + \codim \Gamma \\
    &= \dim \scrL_{m}(\KK^{E \minus \{e\}}, X_{\zeta_{\matM\setminus e}}) + \codim \Gamma \\
    &\leq (m+1)(n-2) + (m+1) \\
    &= (m+1)(n-1).
\end{align*}
So \eqref{eqn - matroid poly, m-jets equidimensional dimension bound} is satisfied under these hypotheses.

\emph{Step 2: Base Case}: Suppose $\matM$ consists of only loops and coloops. Then Remark \ref{rmk - basic matroid polynomial observations}.\eqref{item-loopy} and the hypothesis $\rank(\matM) \geq 1$ shows that $\zeta_\matM$ is a Boolean arrangement. So \eqref{eqn - matroid poly, m-jets equidimensional dimension bound} follows from Remark \ref{rmk - jets of arrangements}.

\emph{The Inductive Proof}: We prove the desired dimension bound by induction on $|E|$. If $|E| = 1$, we fall into the setting of \emph{Step 2}. So assume $|E| \geq 2$ and assume an induction hypothesis. If there exists some $e \in E$ such that $e$ is not a loop/coloop, then the induction hypothesis applies to $\matM\setminus e$ and we can use \emph{Step 1}. If no such $e$ exists, then $E$ consists of only loops and coloops and we invoke \emph{Step 2}.
\end{proof}

\subsection{Rationality of Matroidal Singularities}
%%%%%%%%%%%%%%%%%%%%%%%%%%%%%%%%%%%%%%%%%%%%%%%%%%%

Now we construct machinery to prove that a matroidal polynomial $\zeta_\matM$ has rational singularities when $\matM$ is connected and has rank at least two. We will use Musta\c t\u a's \cite[Thm.~3.3]{MustataJetsLCI} and show $\scrL_{m}(\KK^{E}, X_{\zeta_\matM})$ is irreducible for all $m \geq 1$. Similarly to Proposition \ref{prop - matroid poly, equidimensional jets}, this reduces to bounding the dimension of $m$-jets, though now we must bound the dimension of $m$-jets of $X_{\zeta_\matM}$ lying over the singular locus (Proposition \ref{prop - Mustata's characterization of irreducibility of m-jets in terms of dim lying over sing}). Again, like in Proposition \ref{prop - matroid poly, equidimensional jets}, we use induction on $|E|$ to pass from $\matM$ to the matroid $\matM \setminus T$ on $E \setminus T$ for some well chosen $T \subseteq E$. As noted before,  there are connected matroids $\matM$ such that $\matM \setminus e$ is disconnected for all $e \in E$. This time, the Cremona trick cannot be used since matroidal polynomials do not have to be homogeneous. 

In sum, we have two obstacles:
\begin{enumerate}[label=(\roman*)]
    \item Adequately understand the singular locus $X_{\zeta_\matM, \Sing}$ of $X_{\zeta_\matM}$;
    \item Find a suitable way to pass from $\scrL_m(\KK^E, X_{\zeta_{\matM}}, X_{\zeta_{\matM}, \Sing})$ to jets on a ``nice'' and ``smaller'' matroidal polynomial.
\end{enumerate}
As in the characteristic $p$ leg of our journey, the existence of connective handles resolves our problems. This time, we require a mild generalization/augmentation of \cite[Cor.~3.13]{DSW}: this lets us express $\zeta_\matM$ in terms of matroidal polynomials attached to $\matM \setminus H$ and $\matM / H$.

\begin{proposition} \label{prop - new handle fml - matroidal poly}
Let $\matM$ be a matroid on $E$ admitting a proper handle $H$ that is both independent and contains no coloops. Let $\zeta_\matM$ be a matroidal polynomial. Then there exist matroidal polynomials $\zeta_{\matM / H}$ and $\zeta_{\matM \setminus H}^{h}$, for each $h \in H$, with respect to the matroids $\matM / H$ and $\matM \setminus H$ respectively, such that
\[
\zeta_\matM = \sum_{h \in H} \left( \prod_{g \in H\minus\{h\}} f_{g} \right) \zeta_{\matM \setminus H}^{h} + \left( \prod_{h \in H} f_{h} \right) \zeta_{\matM / H}.
\]
%Also: $\deg(\zeta_{\matM \setminus H}^{h}) \geq \deg(\zeta_{\matM / H})$ for all $h \in H$. 
\end{proposition}

\begin{proof}   
     We induce on $|H|$. 
     %As part of our induction we will also prove that the chosen $\zeta_{\matM / H}$ has degree at most $\deg(\zeta_\matM) - |H|$. 
     The base case $|H|=1$ is a restatement of the Deletion-Contraction axiom for non-(co)loops (i).
    
    Assume an induction hypothesis. Let $h \in H$ and set $H^{\prime} = H\minus\{h\}$. By Deletion-Contraction with respect to $h$ we have
    \begin{equation} \label{eqn - new handle fml - matroidal poly - part 1}
        \zeta_\matM = \zeta_{\matM \setminus h} + f_{h} \zeta_{\matM / h}
    \end{equation}
    for suitable matroidal polynomials $\zeta_{\matM \setminus h}$ and $\zeta_{\matM / h}$. 
    %By Remark \ref{rmk - basic matroid polynomial observations}.(g), $\deg(\zeta_{\matM \setminus h}) \geq \deg(\zeta_{\matM / h})$ and $\deg(\zeta_\matM) - 1 \geq \deg(\zeta_{\matM / h})$. 
    Since $H^{\prime}$ is a handle on $\matM / h$ satisfying the inductive set-up (Remark \ref{rmk-handle basics}), we may write
    \begin{equation} \label{eqn - new handle fml - matroidal poly - part 2}
        \zeta_{\matM / h} = \sum_{h^{\prime} \in H\minus\{h\}} \left(\prod_{g \in H \minus \{h, h^{\prime} \}} f_{g} \right) \zeta_{\matM / h \setminus H^{\prime}}^{h^{\prime}} + \left(\prod_{h^{\prime} \in H^{\prime}} f_{h^{\prime}} \right) \zeta_{\matM / h / H^{\prime}}.
    \end{equation}
    %where $\deg(\zeta_{\matM / h \setminus H^{\prime}}^{h^{\prime}}) \geq \deg(\zeta_{\matM / h / H^{\prime}})$, for all $h^{\prime} \in H^{\prime}$, and $\deg(\zeta_{\matM / h}) - (|H| - 1) \geq \deg(\zeta_{\matM / h / H^{\prime}})$. The last inequality gives $\deg(\zeta_\matM) - |H| \geq \deg(\zeta_{\matM / h / H^{\prime}})$.
    
    By commutativity of deletion and contraction for matroids, $\matM / h \setminus H^{\prime} = \matM \setminus H^{\prime} / h =\matM \setminus h^{\prime} \setminus (H^{\prime} \setminus h^{\prime}) / h$, where $h^{\prime} \in H^{\prime}$.  Since every element of $H\minus\{h\}^{\prime}$ is a coloop on $\matM \setminus h^{\prime}$, we deduce $\matM / h \setminus H^{\prime} = \matM \setminus H$. So we may set $\zeta_{\matM \setminus H}^{h^{\prime}} = \zeta_{\matM / h \setminus H^{\prime}}^{h^{\prime}}$, for each $h^{\prime} \in H^{\prime}$ and $\zeta_{\matM / H} = \zeta_{\matM / h / H^{\prime}}$. The notation change means we can combine \eqref{eqn - new handle fml - matroidal poly - part 1} and \eqref{eqn - new handle fml - matroidal poly - part 2} into
    \begin{equation} \label{eqn - new handle fml - matroidal poly - part 3}
        \zeta_\matM = \zeta_{\matM \setminus h} + \sum_{h^{\prime} \in H^{\prime}} \left( \prod_{g \in H\minus\{h^{\prime}\}} f_{g} \right) \zeta_{\matM \setminus H}^{h^{\prime}} + \left( \prod_{h \in H} f_{h} \right) \zeta_{\matM / H}.
    \end{equation}
    %where $\deg(\zeta_{\matM \setminus H}^{h^{\prime}}) \geq \deg(\zeta_{\matM / H})$ for all $h^{\prime} \in H^{\prime}$, $\deg(\zeta_\matM) - |H| \geq \deg(\zeta_{\matM / H})$, and $\deg(\zeta_{\matM / h}) - (|H| - 1) \geq \deg(\zeta_{\matM / H})$.

    Now since $H^{\prime} = \{h_{1}^{\prime}, \dots, h_{k}^{\prime}\}$ consists of coloops on $\matM \setminus h$, by repeated Deletion-Contraction for coloops we have that 
    \begin{align} \label{eqn - new handle fml - matroidal poly - part 4}
        \zeta_{\matM \setminus h} = f_{h_{1}^{\prime}} \zeta_{\matM \setminus h \setminus h^{\prime}} = f_{h_{1}}^{\prime} f_{h_{2}^{\prime}} \zeta_{\matM \setminus h \setminus h_{1}^{\prime} \setminus h_{2}^{\prime}} = \cdots = \left( \prod_{h^{\prime} \in H^{\prime}} f_{h^{\prime}} \right) \zeta_{\matM \setminus h \setminus h_{1}^{\prime} \setminus \cdots \setminus h_{k}^{\prime}}, 
    \end{align}
    for suitably chosen matroidal polynomials. Since $\matM \setminus h \setminus h_{1}^{\prime} \setminus \cdots \setminus h_{k}^{\prime} = \matM \setminus H$, the formula \eqref{eqn - new handle fml - matroidal poly - part 4} simplifies to: there exists a matroidal polynomial $\zeta_{\matM \setminus H}^{h}$ attached to the matroid $\matM \setminus H$ such that 
    \begin{equation} \label{eqn - new handle fml - matroidal poly - part 5}
        \zeta_{\matM \setminus h} =  \left( \prod_{h^{\prime} \in H^{\prime}} f_{h^{\prime}} \right) \zeta_{\matM \setminus H}^{h}.
    \end{equation}
    %Moreover, $\deg(\zeta_{\matM \setminus H}^{h}) = \deg(M \setminus h) - (|H| - 1) \geq \deg(\matM / h) - (|H| - 1) \geq \deg(\zeta_{\matM/H})$, where the first $\geq$ is the degree assumption in the Deletion-Contraction for non-(co)loops axioms. 
    Plug \eqref{eqn - new handle fml - matroidal poly - part 5} into \eqref{eqn - new handle fml - matroidal poly - part 3} to complete the induction. 
    %and noting that our produced polynomials satisfy the correct degree inequalities completes the induction.
\end{proof}

\begin{define} \label{def - handle matroidal polynomials}
    For $H \in \calI_{\matM}$ a proper handle with no coloops and $\zeta_\matM$ a matroidal polynomial on $\matM$, we call the matroidal polynomials $\zeta_{\matM / H}$ attached to $M / H$ and $\zeta_{\matM \setminus H}^{h}$ attached to $\matM \setminus H$, for each $h \in H$, from Proposition \ref{prop - new handle fml - matroidal poly} the \emph{handle matroidal polynomials}. We will often refer to the handle matroidal polynomials by just using the notation $\zeta_{\matM / H}$ or $\zeta_{\matM \setminus H}^{h}$. 
\end{define}

Proposition \ref{prop - new handle fml - matroidal poly} yields the following relationship between the handle matroidal polynomials. This is vital to our induction. %This is the only time the irreducibility axiom (iv) as well as the degree assumption in the Deletion-Contraction axiom for non-(co)loops axiom (i) in Definition \ref{dfn - new - matroidal polys} is used.

\begin{lemma} \label{lemma - matroidal poly handle fml, intersecting induced polys}
    Let $\matM$ be a connected matroid on $|E|\geq 2$ edges, and let $\zeta_\matM$ be a matroidal polynomial attached to $\matM$. Let $H \in \calI_{\matM}$ be a proper handle of $\matM$ such that $\matM \setminus H$ is connected and $\rank(\matM/ H) \geq 1$. 
    %Recall the handle matroidal polynomials $\zeta_{\matM / H}$ and $\zeta_{\matM \setminus H}^{h}$, for all $h \in H$. 
    Then for each $h \in H$,
    \begin{equation}\label{eqn-codim-bound}
        X_{\zeta_{\matM \setminus H}^{h}} \cap X_{\zeta_{\matM / H}} \subseteq \KK^{E \minus H} \text{ has dimension at most } |E| - |H| - 2.
    \end{equation}
\end{lemma}

\begin{proof}
    As $\emptyset\neq H$ is independent, it is not a circuit but must be contained in one since otherwise $\matM=(\matM_{|_H})\oplus (\matM\setminus H)$ would be disconnected. That circuit cannot be all of $E$ since else we should have $\rank(\matM / H) = 0$. It follows that  $|E\minus H|\geq 2$ and so  $X_{\zeta_{\matM \setminus H}^{h}}$ and $X_{\zeta_{\matM / H}}$ are natively hypersurfaces in the space $\KK^{E \minus H}$ of dimension at least $2$.

    Via a coordinate change we may assume that the singleton data $\sigma$ satisfies $f_{e} = x_{e}$ for all $e \in E$. By Proposition \ref{prop - initial term matroidal poly}, $\min(\zeta_{\matM / H}) \in \MSP(\matM / H)$ and $\min(\zeta_{\matM \setminus H}^{h}) \in \MSP(\matM \setminus H)$ for all $h \in H$; by Corollary \ref{cor-irred} and hypothesis, $\zeta_{\matM \setminus H}^{h} \in \KK[E \minus H]$ is irreducible for all $h\in H$.
    
    Fix  $h \in H$; as $\zeta_{\matM \setminus H}^{h}$ is irreducible, \eqref{eqn-codim-bound} can fail only if $\zeta_{\matM / H} = p_{h} \zeta_{\matM \setminus H}^{h}$ for some $p_{h} \in \mathbb{K}[E]$. This would entail
    \begin{align*}
    \MSP(\matM / H) \ni \min(\zeta_{\matM / H}) 
    &= \min(p_{h} \zeta_{\matM \setminus H}^{h}) \\
    &= \min(p_{h}) \min(\zeta_{\matM \setminus H}^{h}) \in \min(p_{h}) \cdot \MSP(\matM \setminus H).
    \end{align*}
    However, elements of $\MSP(\matM / H)$ are homogeneous of degree $\rank(\matM / H) = \rank(\matM) - |H|$; elements of $\min(p_{h}) \cdot \MSP(\matM \setminus H)$ are homogeneous of degree $\deg(\min(p_{h})) + \rank(\matM \setminus H) = \deg(\min(p_{h})) + \rank(\matM) - |H| + 1$. In consequence, \eqref{eqn-codim-bound} holds.
\end{proof}

Recall our first obstacle to showing $\scrL_m(\KK^E, X_{\zeta_\matM}, X_{\zeta_\matM, \Sing})$ is irreducible, and hence $\zeta_\matM$ has rational singularities \cite[Prop.~1.4]{MustataJetsLCI}, is to adequately ``understand'' the singular locus $X_{\zeta_\matM, \Sing}$ of $\zeta_\matM$. We do not attempt a precise account of the singular locus: its basic properties mysteriously depend on both the matroid and the particular matroidal polynomial selected (\emph{cf.}\ \cite[Ex.~5.2, 5.3]{DSW}). For us, it is enough to approximate it by certain hypersurfaces. 

\begin{lemma} \label{lemma - matroid poly, components of Sing for M connected}
Let $\matM$ be a connected matroid on $E$ and $H$ a proper handle on $\matM$. Let $\zeta_\matM$ be a matroidal polynomial and $\mathfrak{p}$ a prime containing the true Jacobian $(\partial \zeta_\matM) + (\zeta_\matM)$ of $\zeta_\matM$. Then one of the following occurs:
\begin{enumerate}[label=(\alph*)]
    \item $\mathfrak{p} \ni f_{h}$ for some $h \in H$;
    \item $\mathfrak{p} \ni \zeta_{\matM / H}$.
\end{enumerate}
(Item (b) refers to the handle matroidal polynomials derived from $\zeta_\matM$.)
\end{lemma}
\begin{proof}

For any $h \in H$ and any $A \subseteq H$, the definition of $f_{e}$ implies that
\[
\partial_{h} \bullet \prod_{g \in A} f_{g} = 
\left\{\begin{array}{rcl}
    \prod_{g \in A \minus \{h\}} f_{g}&\text{ if }& h \in A; \\
    0&\text{ if }& h \notin A.\end{array}\right.
\]
Invoking Proposition \ref{prop - new handle fml - matroidal poly} for each $h \in H$, we compute:
\begin{align} \label{eqn - matroid poly, components of Sing for M connected, eq 1}
    \zeta_\matM - f_{h} \partial_{h} \bullet \zeta_\matM 
    &= \zeta_\matM - \left( \sum_{h^{\prime} \in H\minus\{h\}} \left( \prod_{g \in H\minus\{h^{\prime}\}} f_{g} \right) \zeta_{\matM \setminus H}^{g} \right) + \left( \prod_{h \in H} f_{h} \right) \zeta_{\matM / H} \\
    &= \left( \prod_{h^{\prime} \in H\minus\{h\}} f_{h^{\prime}} \right) \zeta_{\matM \setminus H}^{h} \nonumber
\end{align}
Using Proposition \ref{prop - new handle fml - matroidal poly} and \eqref{eqn - matroid poly, components of Sing for M connected, eq 1} gives
\begin{equation} \label{eqn - matroid poly, components of Sing for M connected, eq 2}
    \zeta_\matM - \left( \sum_{h \in H} f_{h} \partial_{h} \bullet \zeta_\matM \right) = \left( \prod_{h \in H} f_{h} \right) \zeta_{\matM / H}.
\end{equation}
The lemma follows by the primality of $\mathfrak{p} \supseteq (\partial \zeta_\matM) + (\zeta_\matM).$
\end{proof}

\begin{example} \label{ex - matroid poly, case of circuit} Let $M = \matU_{n-1, n}$ be the $n$-circuit with $n \geq 2$. For a matroidal polynomial $\zeta_\matM$ and a prime $\mathfrak{p} \in \Spec \KK[E]$ containing the true Jacobian $(\partial \zeta_\matM) + (\zeta_\matM)$, we give a refined version of Proposition \ref{lemma - matroid poly, components of Sing for M connected} in this special case. Pick $g \in E$. As $\matM \setminus g$ has only coloops,
    \[
    \zeta_\matM = \zeta_{\matM \setminus g} + f_{g} \zeta_{\matM / g} = \left( \alpha \prod_{e \in E \minus  \{g\}} f_{e} \right) + f_{g} \zeta_{\matM / g}
    \]
for some constant $\alpha \in \KK^{\times}$. Then
    \[
    (\partial \zeta_\matM) + (\zeta_\matM) \ni \zeta_\matM - f_{g} \partial_{g} \bullet \zeta_\matM = \alpha \prod_{e \in E \minus \{g\}} f_{e}.
    \]
So there exists $e \in E$ such that $\mathfrak{p} \ni f_{e}.$
\end{example}

Now we come to the second obstacle in our induction puzzle: the mechanism of reduction to ``nice'' and ``smaller'' matroidal polynomial. As when proving equidimensionality of $m$-jets, we do this by intersecting $\scrL_{m}(\KK^{E}, X_{\zeta_\matM}, X_{\zeta_\matM , \Sing})$ with a well chosen complete intersection. What follows is a variant of Proposition \ref{prop - matroidal poly, m-jets intersect with Gamma}: 

\begin{lemma} \label{lemma - matroidal poly, handle fml intersect with Gamma}
    Let $\matM$ be a connected matroid on $E$ and $\zeta_\matM$ a matroidal polynomial. Suppose that $H$ is a proper handle on $\matM$. For $h \in H$ define
        \[
        Q_{h} = \left( \prod_{h^{\prime} \in H\minus\{h\}} f_{h^{\prime}} \right) \zeta_{\matM \setminus H}^{h}
        \]
    and let $\Gamma \subseteq \scrL_{m}(\KK^{E})$ be the variety cut out by $(\{D^{q} f_{h}\}_{0 \leq q \leq m})$. Then
    \begin{equation} \label{eqn - matroid poly, connected, m-jet handle intersect with Gamma}
        \scrL_{m}(\KK^{E}, X_{\zeta_\matM}) \cap \Gamma \simeq \scrL_{m}(\KK^{E\minus\{h\}}, \Var(Q_{h}))
    \end{equation}
    and 
    \begin{equation} \label{eqn - matroidal poly, connected, m-jet handle lying over intersect with Gamma}
        \scrL_{m}(\KK^{E}, X_{\zeta_\matM}, \Var(\zeta_{\matM / H})) \cap \Gamma \simeq \scrL_{m}(\KK^{E\minus\{h\}}, \Var(Q_{h}), \Var(\zeta_{\matM / H})).
    \end{equation}
    Moroever, if $\rank(\matM/ H) \geq 1$ then the intersection in \eqref{eqn - matroidal poly, connected, m-jet handle lying over intersect with Gamma} is non-empty.
\end{lemma}

\begin{proof}
    We can rewrite the handle formula Proposition \ref{prop - new handle fml - matroidal poly} as
    \begin{equation*}
        \zeta_\matM = f_{h} \left[ \sum_{g \in H\minus\{h\}} \left( \prod_{e \in H \minus \{h, g\}} f_{e} \right) \zeta_{\matM \setminus H}^{g} + \left( \prod_{h^{\prime} \in H\minus\{h\}} f_{h^{\prime}} \right) \zeta_{\matM / H} \right] + Q_{h}.
    \end{equation*}
    More succinctly, there exists a $g \in \KK[E\minus\{h\}]$ such that $\zeta_\matM = Q_{h} + f_{h} g.$ So
    \[
    D^{q} \zeta_\matM = D^{q} Q_{h} + \sum_{0 \leq p \leq q} \binom{q}{p} (D^{p} f_{h}) (D^{p-q} g),
    \]
    and \eqref{eqn - matroid poly, connected, m-jet handle intersect with Gamma} follows exactly as in Proposition \ref{prop - matroidal poly, m-jets intersect with Gamma}. Since $\scrL_{m}(\KK^{E}, X_{\zeta_\matM}, \Var(\zeta_{\matM / H}))$ is defined by $(\{D^{q} \zeta_\matM \}_{0 \leq q \leq m}) + (\zeta_{\matM / H})$, similar considerations justify \eqref{eqn - matroidal poly, connected, m-jet handle lying over intersect with Gamma}.

    Now we prove the non-emptiness of $\scrL_{m}(\KK^{E \minus \{e\}}, \Var(Q_{h}), \Var(\zeta_{\matM / H}))$ provided that $\rank(\matM/ H) \geq 1$. By Remark \ref{rmk - basic matroid polynomial observations}.\eqref{item-mat-poly-degree}, 
    \[
    \zeta_{\matM / H} \in (\{f_{e} \}_{e \in E \minus H}) \subsetneq \KK[E\minus\{h\}].
    \]
    As $\rank(\matM/ H) \geq 1$ implies that $\rank(M \setminus H) \geq 2$ (Remark \ref{rmk-handle basics}), again Remark \ref{rmk - basic matroid polynomial observations}.\eqref{item-mat-poly-degree} gives
    \[
    Q_{h} \in (\zeta_{\matM \setminus H}^{h}) \subseteq (\{f_{e} \}_{e \in E \minus H}) \subsetneq \KK[E\minus\{h\}].
    \]
    So
    \[
    \KK^{E\minus\{h\}} \supsetneq \Var(Q_{h}) \cap \Var(\zeta_{\matM / H}) \neq \emptyset. 
    \]
    The non-emptyness of $\scrL_{m}(\KK^{E\minus\{h\}}, \Var(Q_{h}), \Var(\zeta_{\matM / H}))$ then follows by Remark \ref{rmk - always a m-jet lying over a point in X} which says we can always find a $m$-jet lying over a point in $\Var(Q_{h}) \cap \Var(\zeta_{\matM / H}) \subseteq \KK^{E\minus\{h\}}$.
    \end{proof}

We have now assembled the tools in order to prove this section's main theorem:

\begin{theorem} \label{thm - matroid poly, connected implies m-jets irreducible}
Let $\matM$ be a connected matroid on $E$ of rank at least 1,  equipped with a matroidal polynomial $\zeta_\matM$. Then $\scrL_{m}(\KK^{E}, X_{\zeta_\matM})$ is an irreducible, reduced, l.c.i. of dimension $(m+1)(n-1)$ for all $m \geq 1$. Moreover $X_{\zeta_\matM}$ is normal.
\end{theorem}

\begin{proof}
We shall assume that coordinates have been adjusted so that $f_e=x_e$ in all cases.

If the rank of $\matM$ is one, then connectedness forces $\matM$ to be loopless, hence equal to $\matU_{1,n}$. The matroidal polynomials in this case have been determined in Remark \ref{rmk - basic matroid polynomial observations}.\eqref{item-U1n-mtrdl}. Up to shifting $x_e$ by $\delta_e$, they are equivalent to $1-\prod_{e\in E}x_e$ and hence smooth. We can thus focus on matroids of rank at least two.

By Proposition \ref{prop - Mustata's characterization of irreducibility of m-jets in terms of dim lying over sing} it suffices to show that $\scrL_{m}(\KK^{E}, X_{\zeta_\matM})$ is irreducible, which is equivalent to
\begin{equation} \label{eqn - matroid poly, m-jets irreducible dimension bound}
\dim \scrL_{m}(\KK^{E}, X_{\zeta_\matM}, X_{\zeta_\matM, \Sing}) < (m+1)(n-1) \text{ for all } m \geq 1.
\end{equation}
Our plan is to certify \eqref{eqn - matroid poly, m-jets irreducible dimension bound} via an induction argument on $|E|$, using the existence of a handle decomposition on $\matM$.

To this end, let $\mathfrak{p} \in \Spec \KK[E]$ be a minimal prime containing the true Jacobian $(\partial \zeta_\matM) + (\zeta_\matM)$. We will show that 
\begin{equation} \label{eqn - matroid poly, m-jets irreducible, dimension condition for minimal primes}
    \dim \scrL_{m}(\KK^{E}, X_{\zeta_\matM}, \Var(\mathfrak{p})) < (m+1)(n-1),
\end{equation}
under certain assumptions on $\frakp$; these are the main steps of the proof. 

\vspace{3ex}

\emph{Case 1}: If $\mathfrak{p} \ni f_{e}$ for some $e \in E$, then \eqref{eqn - matroid poly, m-jets irreducible, dimension condition for minimal primes} holds.

\vspace{2ex}

\noindent\emph{Proof}: Let $\Gamma \subseteq \scrL_{m}(\KK^{E})$ be cut out by $(D f_{e} , \dots, D^{m} f_{e})$. Proposition \ref{prop - matroidal poly, m-jets intersect with Gamma} implies that
\begin{equation} \label{eqn - matroid poly, case 1 intersect gamma}
    \scrL_{m}(\KK^{E}, X_{\zeta_\matM}, \Var(f_{e})) \cap \Gamma
    \simeq \scrL_{m}(\KK^{E\minus\{e\}}, X_{\zeta_{\matM\setminus e}}).
\end{equation}
This variety is non-empty since $\rank(\matM) = \rank(\matM\setminus e)$ implies that $\zeta_{\matM\setminus e}$ is not a constant.  We deduce \eqref{eqn - matroid poly, m-jets irreducible, dimension condition for minimal primes}:
\begin{align*}
    \dim \scrL_{m}(\KK^{E}, X_{\zeta_\matM}, \Var(\mathfrak{p})) 
    &\leq \dim \scrL_{m}(\KK^{E}, X_{\zeta_\matM}, \Var(f_{e})) \\
    &\leq \dim (\scrL_{m}(\KK^{E}, X_{\zeta_\matM}, \Var(f_{e})) \cap \Gamma) + \codim \Gamma \\
    &= \dim \scrL_{m}(\KK^{E\minus\{e\}}, X_{\zeta_{\matM\setminus e}}) + \codim \Gamma \\
    &\leq (m+1)(n-2) + m \\
    &= (m+1)(n-1) - 1.
\end{align*}
The justification is: the first ``$\leq$'' is $\mathfrak{p} \ni f_{e}$; the second ``$\leq$'' is Lemma \ref{lemma - basic intersection theory identity}; the first ``$=$'' is \eqref{eqn - matroid poly, case 1 intersect gamma}; the third ``$\leq$'' is Proposition \ref{prop - matroid poly, equidimensional jets}. Hence, \eqref{eqn - matroid poly, m-jets irreducible, dimension condition for minimal primes} follows in Case 1.

\vspace{3ex}

\emph{Case 2:} Suppose $H \subsetneq E$ is a proper handle such that: $\matM \setminus H$ is connected; $\rank(\matM/ H) \geq 1$; $\mathfrak{p} \ni \zeta_{\matM /H}$; there exists a $h \in H$ such that $ X_{\zeta_{\matM \setminus H}^{h}} \cap X_{\zeta_{\matM / H}} \subseteq \KK^{E\minus H}$ has dimension at most $|E| - |H| - 2$. Then if \eqref{eqn - matroid poly, m-jets irreducible dimension bound} holds for $\scrL_{m}(\KK^{E\minus H}, X_{\zeta_{\matM \setminus H}^{h}})$, then \eqref{eqn - matroid poly, m-jets irreducible, dimension condition for minimal primes} holds.

\vspace{3ex}

\noindent\emph{Proof}: Using the $h \in H$ promised in the claim's hypothesis, define $Q_{h} = \left( \prod_{h^{\prime} \in H \minus \{h\}} f_{h^{\prime}} \right) \zeta_{\matM \setminus H}^{h}$ with corresponding variety $\Var(Q_{h}) \subseteq \KK^{{E\minus\{h\}}}$. Let $\Gamma \subseteq \scrL_{m}(\KK^{E})$ have defining ideal $(f_{h}, D f_{h}, \dots, D^{m} f_{h})$. By Lemma \ref{lemma - matroidal poly, handle fml intersect with Gamma},
\[
\scrL_{m}(\KK^{E}, X_{\zeta_\matM}, \Var(\zeta_{\matM / H})) \cap \Gamma \simeq \scrL_{m}(\KK^{{E\minus\{h\}}}, \Var(Q_{h}), \Var(\zeta_{\matM / H})) \neq \emptyset.
\]
Then Lemma \ref{lemma - basic intersection theory identity} in conjunction with the claim's hypotheses yields that
\begin{align*}
\dim \scrL_{m}(\KK^{E}, X_{\zeta_\matM}, \Var(\mathfrak{p})) 
    &\leq \dim \scrL_{m}(\KK^{E}, X_{\zeta_\matM}, \Var(\zeta_{\matM / H})) \\
    &\leq \dim (\scrL_{m}(\KK^{E}, X_{\zeta_\matM}, \Var(\zeta_{\matM / H})) \cap \Gamma) + \codim \Gamma \\
    &= \dim (\scrL_{m}(\KK^{{E\minus\{h\}}}, \Var(Q_{h}), \Var(\zeta_{\matM / H})) + (m+1).
\end{align*}
We will have certified \eqref{eqn - matroid poly, m-jets irreducible, dimension condition for minimal primes} once we confirm
\begin{equation} \label{eqn - m-jets of Q h lying over zeta M / H}
\dim (\scrL_{m}(\KK^{{E\minus\{h\}}}, \Var(Q_{h}), \Var(\zeta_{\matM / H})) \leq (m+1)(n-2) - 1.
\end{equation}

If $|H| - 1 = 0$, then \eqref{eqn - m-jets of Q h lying over zeta M / H} is exactly \eqref{eqn - m-jets of Q h, vector 3} with $t=m$, while  otherwise we go through a series of simplifications that reduce to the general case of \eqref{eqn - m-jets of Q h, vector 3}.

Consider the bounded collection of integral vectors
\[
\Upsilon = \{ \mathbf{v} \in \mathbb{Z}^{H\setminus\{h\}} \mid -1 \leq v_k \leq m \text{ and } |\mathbf{v}| \leq m -  (|H| - 1) \},
\]
where $|\boldv|$ denotes the sum of entries.
To each $\mathbf{v} \in \Upsilon$, define (using the convention $D^{-1}(g) = 0$) the ideal  $J_{\mathbf{v}} \subseteq \KK[\{x_{e}^{(p)}\}_{\substack{ e \in E \minus \{h\}\\ 0 \leq p \leq m}}]$ by 
\begin{equation*}
J_{\mathbf{v}} \text{ is generated by } \{ D^{0}f_{h^{\prime}}, \dots, D^{v_{h^{\prime}}} f_{h^{\prime}}\}_{h^{\prime} \in H\minus\{h\}} \cup \{\zeta_{\matM \setminus H}^{h}, \dots, D^{m - (|H| - 1) - |\mathbf{v}|} \zeta_{\matM \setminus H}^{h} \}.
\end{equation*}
Now let $\mathfrak{q}$ be a minimal prime of $I(\scrL_m(\KK^{{E\minus\{h\}}}, \Var(Q_h)))$. Using Lemma \ref{lemma - jet schemes of products, not nec disjoint variables} repeatedly, $\mathfrak{q}$ must contain such a $J_{\mathbf{v}}$. To establish \eqref{eqn - m-jets of Q h lying over zeta M / H} it suffices to bound the dimension of $\Var(\mathfrak{q}) \cap \pi_{m,0}^{-1}(\Var(\zeta_{M / H}))$ for all such $\mathfrak{q}$; hence, in order to establish \eqref{eqn - m-jets of Q h lying over zeta M / H} it suffices to confirm that
\begin{equation} \label{eqn - m-jets of Q h, vector 1}
   \dim ( \Var(J_{\mathbf{v}}) \cap \pi_{m,0}^{-1}(\Var(\zeta_{\matM / H})) ) \leq (m+1)(n-2) - 1 \quad \forall \enspace \mathbf{v} \in \Upsilon.
\end{equation}

Fix $\mathbf{v} \in \Upsilon$. If $|\mathbf{v}| = - (|H| - 1)$, then $J_\boldv$ contains $\{\zeta_{\matM\setminus H}^h, \dots, D^m \zeta_{\matM \setminus H}^{h} \}$ and it suffices to skip to \eqref{eqn - m-jets of Q h, vector 3} and verify it with $t=m$. 
Otherwise, since $f_{h^\prime} \in \KK[h^\prime]$ for all $h^\prime \in H\minus\{h\}$ and since $\zeta_{\matM \setminus H}^{h}, \zeta_{\matM / H} \in \KK[E\minus H]$, we may check \eqref{eqn - m-jets of Q h, vector 1} after intersecting the situation with the $|\boldv|+(|H|-1)$ hyperplanes
$D^a f_{h^\prime} = 0$ for which $0 \leq a \leq v_{h^\prime}$. In the resulting vector space we are concerned with the vanishing locus of $D^{0} \zeta_{\matM \setminus H}^{h}, \dots, D^{m - (|H| - 1) - |\mathbf{v}|} \zeta_{\matM \setminus H}^{h}, \text{ and } \zeta_{\matM / H}$. These polynomials only use the variables $x_{e}^{(p)}$ for $e \in E\minus H$ and $0 \leq p \leq m$. Removing the remaining variables $\bsx^{(p)}_e$ attached to $e\in H$,  we can verify  \eqref{eqn - m-jets of Q h, vector 1} by  checking that
\begin{equation*}
    \{D^{0} \zeta_{\matM \setminus H}^{h} = \dots = D^{m - (|H| - 1) - |\mathbf{v}|} \zeta_{\matM \setminus H}^{h} = 0 = \zeta_{\matM / H} \} \subseteq \Spec \scrL_m(\KK^{E\minus H})
\end{equation*}
has dimension at most 
\begin{align*}
    (m+1)(n - 2) - 1 - &\left(\sum_{h^\prime \in H\minus\{h\}}  (m+1) - (v_{h^\prime} + 1)\right)\\
    &= (m+1)(n - |H| - 1) - 1 + (|\mathbf{v}| + |H| - 1) \\
    &= ( \bigg[ m - |\mathbf{v}| - (|H| - 1) \bigg] + 1) (n - |H| - 1) - 1 \\
    &\phantom{xxx}+ (m - \bigg[ m - |\mathbf{v}| - (|H| - 1) \bigg]) (n - |H|).
\end{align*}
To visualize this in more succinct notation, set $t =  m - |\mathbf{v}| - (|H| - 1)$. Then in order to show that \eqref{eqn - m-jets of Q h, vector 1}, it is enough to show that
\begin{align} \label{eqn - m-jets of Q h, vector 2}
    \dim \bigg( (\pi_{m,t}^{\KK^{E\minus H}})^{-1} \scrL_t(\KK^{E\minus H}, X_{\zeta_{M \setminus H}^{h}}, X_{\zeta_{M / H}}) \bigg) &\leq (t + 1)(n - |H| - 1) - 1 \\
    &+ (m - t)(n - |H|). \nonumber
\end{align}
When $t=-1$, we our conventions imply that the left-hand side of \eqref{eqn - m-jets of Q h, vector 2} is $(\pi_{m,0}^{\KK^{E\minus H}})^{-1}(X_{\zeta_{\matM / H}})$, and 
the desired inequality  is clear: the left-hand side is the dimension of a a hypersurface in $\scrL_m(\KK^{E\minus H})$; the right-hand side is $(m+1)(n - |H|) - 1$. For $0 \leq t \leq m$, the fibration $\pi_{m,t}^{\KK^{E\minus H}} : \scrL_m (\KK^{E\minus H}) \to \scrL_t (\KK^{E\minus H})$ is trivial with fiber dimension $(n-|H|)(m-t)$, and so validating \eqref{eqn - m-jets of Q h, vector 2} is tantamount to validating
\begin{equation} \label{eqn - m-jets of Q h, vector 3}
    \dim \scrL_t(\KK^{E\minus H}, X_{\zeta_{M \setminus H}^{h}}, X_{\zeta_{M / H}}) \leq (t + 1)(n - |H| - 1) - 1.
\end{equation}
By Lemma \ref{lemma - intersecting irreducible jet scheme with inverse image of avoidant variety downstairs}, the truth of \eqref{eqn - m-jets of Q h, vector 3} follows provided both: $\scrL_t(\KK^{E\minus H}, X_{\zeta_{M \setminus H}^{h}})$ is irreducible; $\dim (X_{\zeta_{M \setminus H}^{h}} \cap X_{\zeta_{M / H}}) \leq (n - |H| - 2)$. But these are (a subset of) our hypotheses in \emph{Case 2}. Thus \eqref{eqn - m-jets of Q h, vector 3} is indeed true, and unraveling all the implications/reductions, we see \eqref{eqn - matroid poly, m-jets irreducible, dimension condition for minimal primes} holds in \emph{Case 2}.

\medskip

%\emph{The Inductive Argument}: 
We now complete the proof of the theorem, showing that always one of  \emph{Case 1} and \emph{Case 2} applies.
%into an inductive argument proving \eqref{eqn - matroid poly, m-jets irreducible dimension bound}. We induce on $|E|$. 

We start with the case of circuits, when $\matM = \matU_{n-1,n}$ for $n \geq 2$. By Example \ref{ex - matroid poly, case of circuit}, for any minimal prime $\mathfrak{p}$ of the true Jacobian, there exists $e \in E$ such that $\mathfrak{p} \ni f_{e}$. Then \eqref{eqn - matroid poly, m-jets irreducible, dimension condition for minimal primes} holds for this $\mathfrak{p}$ by \emph{Case 1} and hence $\eqref{eqn - matroid poly, m-jets irreducible dimension bound}$ is validated.

%We start with the smallest $E$ admitting a rank $2$ connected matroid. Which means $|E|=3$ and $\matM = \matU_{2,3}$, the $3$-circuit. But we already handled the case of circuits. The argument appealed only to \emph{Case 1} which did not require the inductive hypothesis. So \eqref{eqn - matroid poly, m-jets irreducible dimension bound} holds for $M = \matU_{2,3}$ as well, terminating the induction and the proof.

%First assume that \eqref{eqn - matroid poly, m-jets equidimensional dimension bound} holds for all connected matroids $\matM$ of 

Now suppose that $\matM$ is an arbitrary connected matroid of rank at least $2$. We will induce on $|E|$; note that $|E|>2$ is forced. If $|E|=3$ then $\matM=\matU_{2,3}$ is a circuit and so the base case is settled.
%and all attached matroidal polynomials, on ground sets smaller than $|E|$.

Now let $\matM$ be connected of rank at least 2, but not a circuit. Assume that the theorem has already been shown for all connected minors of $\matM$ that have rank 2 or more. 
 By \cite[Lem.~2.13]{DSW} we may find a circuit $C \neq E$ of $\matM$ such that $\rank_{\matM}(C) \geq 2$. Since $\matM$ is not a circuit, by \cite[Prop.~2.8]{DSW} we know $\matM$ admits a proper handle $H$ such that $\matM \setminus H$ is connected and $\matM \setminus H \supseteq C$. So $2 \leq \rank_{\matM}(C) \leq \rank(\matM \setminus H) = \rank(\matM) - |H| + 1$ which forces $\rank(\matM / H) = \rank(\matM) - |H| \geq 1$. By Lemma \ref{lemma - matroid poly, components of Sing for M connected}, we know $\mathfrak{p} \ni f_{e}$ for some $e \in H$ or $\mathfrak{p} \ni \zeta_{\matM / H}$. If the first membership holds, \emph{Case 1} applies. If the second membership holds, then Lemma \ref{lemma - matroidal poly handle fml, intersecting induced polys} implies that there exists some $h \in H$ such that $X_{\zeta_{\matM \setminus H}^{h}} \cap X_{\zeta_{\matM / H}} \subseteq \KK^{E\minus H}$ has dimension at most $|E| - |H| - 2$, so that (since $\matM\setminus H$ is a strict minor of $\matM$)  \emph{Case 2} applies. In either setting, the cases validate \eqref{eqn - matroid poly, m-jets irreducible, dimension condition for minimal primes} and consequently also \eqref{eqn - matroid poly, m-jets irreducible dimension bound}.
\end{proof}

\begin{corollary} \label{cor-mtrdl-rat-sing}
    Suppose that $\matM$ is a connected matroid on $E$ with $\rank(\matM) \geq 2$ equipped with a matroidal polynomial $\zeta_\matM$. Then $X_{\zeta_\matM}$ has rational singularities.
\end{corollary}

\begin{proof}
By Remark \ref{rmk - basic matroid polynomial observations}.(d), the matroidal polynomial is singular. By Theorem \ref{thm - matroid poly, connected implies m-jets irreducible}, we know $\scrL_{m}(\KK^{E}, X_{\zeta_\matM})$ is irreducible for all $m \geq 1$. Thus $X_{\zeta_\matM}$ has canonical singularities by \cite[Thm.~3.3]{MustataJetsLCI}. And \cite{ElkikRationalSingularities} and \cite{FlennerRationalSingularities} demonstrate that $X_{\zeta_\matM}$ has rational singularities if and only if it has canonical singularities (because $X_{\zeta_\matM}$ is a l.c.i).
\end{proof}

\subsection{Multivariate Tutte Polynomials}
%%%%%%%%%%%%%%%%%%%%%%%%%%%%%%%%%%%%%%%%%%%%%%%%%%%%%%%%%%

Here we discuss a multivariate Tutte polynomial $Z_{\matM}$ that is a slight variation of the one appearing in \cite{SokalMultiTuttePoly}. (See loc. cit. for a general discussion of this polynomial's importance.) While a matroid can be reconstructed from its set of bases and hence its matroid basis polynomial, the multivariate Tutte polynomial encodes \emph{all} the data about a matroid simultaneously by encoding the rank of every subset of $E$. Moreover, $Z_{\matM}$ encompasses many other polynomials important in combinatorics and physics. For example, specializing $p^{- \rank M}Z_{\matM}$ to $p = q \in \mathbb{Z}_{\geq 1}$ recovers the $q$-state Potts model (see \cite{SokalMultiTuttePoly}).

\begin{define} \label{def - multi tutte poly}
    For $\matM$ a matroid on the ground set $E$ we define the \emph{multivariate Tutte polynomial} $Z_{\matM} \in \KK[E][p]$ by
    \[
    Z_{\matM} = Z_{\matM}(p, \{x_{e}\}_{e \in E}) = \sum_{A \subseteq E} p^{\rank M - \rank A} \prod_{a \in A} x_{a}.
    \]
    The \emph{multivariate Tutte hypersurface} $X_{Z_{\matM}} \subseteq \KK^{E \cup \{p\}}$ is hypersurface in $\KK^{n + 1}$ with defining ideal $(Z_{\matM}) \subseteq \KK[E][p]$. When $\matM$ is the matroid on $\emptyset$, we declare $Z_{\matM} = 1$. Note that if $\matM$ has rank at least 2, then $Z_\matM$ vanishes at the origin with order at least 2 and is hence singular there.
\end{define}

\begin{remark} \label{rmk - multivariate tutte, delete-contract}
    The multivariate Tutte polynomial satisfies the Deletion-Contraction formula
    \[ Z_{\matM} = 
    \begin{cases}
        Z_{\matM\setminus e} + x_{e} Z_{\matM/e} \enspace \text{ if $e$ not a (co)loop}; \\
        (p + x_{e}) Z_{\matM\setminus e} \enspace \enspace \enspace \text{ if $e$ is a coloop}; \\
        (1 + x_{e}) Z_{\matM\setminus e} \enspace \enspace \enspace \text{ if $e$ is a loop}.
        
    \end{cases}
    \]
\end{remark}

We show that our results about the $m$-jets of matroidal polynomials quickly imply the corresponding results for the $m$-jets of multivariate Tutte polynomials. We will prove that $X_{Z_{\matM}}$ (is singular at the origin and) has rational singularities whenever $\matM$ is connected of $\rank(\matM) \geq 2$, while of course if $\matM$ is connected with $\rank(\matM) = 1$  then $Z_{\matM} = p + x_{e}$ is smooth.

We begin with a general deformation lemma about rational singularities.

\begin{lemma} \label{lemma - deformation of rat sing}
    Let $E$, $J$ be disjoint sets. Let $f \in (\{x_j\}_{j \in J}) \cdot \KK[E \sqcup J]$. Let $g \in \KK[E]$ have rational singularities. If the ideals $(\partial(f+g))$ and $(f+g)$ are contained in the ideal generated by all variables then $f + g \in \KK[E \sqcup J]$ has rational singularities.
\end{lemma}

\begin{proof}
By Proposition \ref{prop - Mustata's characterization of irreducibility of m-jets in terms of dim lying over sing} and the justification for Corollary \ref{cor-mtrdl-rat-sing} it suffices to demonstrate:
    \begin{equation} \label{eqn - rat sing deformation, required dim bound}
        \dim \scrL_{m}(\KK^{E \sqcup J}, X_{f + g}, X_{f + g, \Sing}) < (m+1)(|E| + |J| - 1) \text{ for all } m \geq 1.
    \end{equation}
For each $j \in J$ pick $f_{j} \in \KK[E \sqcup J]$ such that $f = \sum_{j \in J} x_{j} f_{j}$. Let $\Gamma_{p} = \{x_{j}^{(q)} = 0\}_{\substack{ j \in J \\ 0 \leq q \leq p}}.$ Abbreviate $\Gamma = \Gamma_{m}$. Because 
\begin{align*}
    (D^{p} (f + g)) + I(\Gamma_{p}) 
    &= \left(D^{p} g + \sum_{j \in J} \sum_{0 \leq a_{j} \leq p} \binom{p}{a_{j}} D^{a_{j}} x_{j} D^{p - a_{j}} f_{j} \right) + I(\Gamma_{p}) \\
    &= (D^{p} g ) + I(\Gamma_{p}),
\end{align*} we have the natural isomorphism
\begin{equation} \label{eqn - rat sing deformation, intersecting jet scheme iso}
    \scrL_{m}(\KK^{E \sqcup J}, X_{f+g}) \cap \Gamma \simeq \scrL_{m}(\KK^{E}, X_{g}).
\end{equation}
Moreover, 
\begin{align} \label{eqn - rat sing deformation, intersecting jacobian containment}
    (\partial (f + g)) + (f+g) + (\{x_{j}\}_{j \in J}) 
    &\supseteq (\{\partial_{x_{e}} \bullet (f+g)\}_{e \in E}) + (f + g) + (\{x_{j}\}_{j \in J}) \\
    &= (\{\partial_{x_{e}} \bullet g\}_{e \in E}) + (g) + (\{x_{j}\}_{j \in J}) \nonumber \\
    &= (\partial g) + (g) + (\{x_{j}\}_{j \in J}). \nonumber
\end{align}

Then we obtain \eqref{eqn - rat sing deformation, required dim bound} by combining \eqref{eqn - rat sing deformation, intersecting jet scheme iso} and \eqref{eqn - rat sing deformation, intersecting jacobian containment} (for the econd ``$\leq$'') with our intersection theory identity Lemma \ref{lemma - basic intersection theory identity} (for the first ``$\leq$'') as follows:
\begin{align*}
    \dim \scrL_{m}(\KK^{E \sqcup J}, X_{f+g}, X_{f+g, \Sing}) 
    &\leq \dim (\scrL_{m}(\KK^{E \sqcup J}, X_{f+g}, X_{f + g, \Sing}) \cap \Gamma) \\
    &\qquad+ \codim \Gamma \\
    &\leq \dim \scrL_{m}(\KK^{E}, X_{g}, X_{g, \Sing}) + \codim \Gamma \\
    &= \dim \scrL_{m}(\KK^{E}, X_{g}, X_{g, \Sing}) + (m+1)|J| \\
    &< (m+1)(|E| - 1) + (m+1)|J| \\
    &= (m+1)({E} + |J| - 1).
\end{align*}
Indeed, the ``$<$'' is our assumption of rationality of $g$; the required fact $\scrL_{m}(\KK^{E \sqcup J}, X_{f+g}, X_{f+g, \Sing}) \cap \Gamma \neq \emptyset$ comes from the assumption $0$ is a singular point of $f+g$, \emph{cf.}\ Remark \ref{rmk - always a m-jet lying over a point in X}.
\end{proof}

Lemma in hand, we can prove:

\begin{theorem} \label{thm - multi tutte poly, rat sing}
    Let $Z_{\matM} \in \KK[E][p]$ be the multivariate Tutte polynomial of a connected matroid $\matM$ on $E$. Let $\rank(\matM) \geq 2$. Then $Z_{\matM}$ has rational singularities.
\end{theorem}

\begin{proof}
    If we restrict $Z_{\matM}$ to $p=0$ we get $\text{MaxR}_{\matM}$ which is a matroidal polynomial on $\matM$ (the maximal rank polynomial of Example \ref{ex-matroidal-polys}). So we may write $Z_{\matM} = (Z_{\matM} - \text{MaxR}_{\matM}) + \text{MaxR}_{\matM}$ where $Z_{\matM} - \text{MaxR}_{\matM} \in (p) \cdot \KK[E][p]$ and $\text{MaxR}_{\matM} \in \KK[E]$. Since $\text{MaxR}_{\matM}$ has rational singularities by Corollary \ref{cor-mtrdl-rat-sing}, the claim follows by Lemma \ref{lemma - deformation of rat sing}.
\end{proof}

%%%%%%%%%%%%%%%%%%%%%%%%%%%%%%%%%%%%%%%%%%%%%%%%%%%%%%%%%%
\section{Flag Matroids and Flag Matroidal Polynomials}
%%%%%%%%%%%%%%%%%%%%%%%%%%%%%%%%%%%%%%%%%%%%%%%%%%%%%%%%%%
\label{seg-flagmatroids}

We extend our strategy of combining handle style induction and jet schemes to study the singularities of sums of matroid support polynomials attached to certain matroids. We give a handle formula, use it (via suitable hypersurfaces containing the true Jacobian) to show that the $m$-jets are equidimensional under minimal assumptions, and derive irreducibility of $m$-jets under  connectedness conditions. 

The main application, obtained via \cite{MustataJetsLCI} and using vocabulary as yet to be defined, is Corollary \ref{cor-flag-mtrdl-ratsing}: flag matroidal polynomials on (terminally strict and terminally connected) flag matroids $\scrM$ of terminal rank at least two have rational singularities. This result is quite robust, \emph{cf.}\ Corollary \ref{cor - polys whose monomial support are ind sets are rational}.

\subsection{Primer on Matroid Quotients}
%%%%%%%%%%%%%%%%%%%%%%%%%%%%%%%%%%%%%%%%%%%%%%%%%%%%%%%%%%
\label{sec-flag-matroids}
%We are going to study sums of polynomials, each one attached to a matroid $M_{i}$ on a shared ground set $E$. The polynomials will be the matroid support polynomials. As for the collection $\{M_{i}\}$ of matroids on $E$, we will only consider ones where each $M_{i}$ is a \emph{quotient} of another. 

Matroid quotients originate from linear algebra. Suppose $\matM$ is representable over $\KK$, as a collection of vectors $v_{e} \in \KK^{s}$ where the independent sets $I$ of $\matM$ correspond to the subsets $\{v_{e}\}_{e \in I}$ that are linearly independent. For $\rho \in \KK^{s}$, let $\KK^{s} \to \KK^{s} / \KK \rho \simeq \KK^{s-1}$ be the canonical quotient map. Then the elements 
\[
\{ \overline{v}_{e} \in \KK^{s}/\KK \rho \}_{e \in E}
\]
represent a matroid again. Special cases arise when fixing a subset $E'\subseteq E$ and choosing $\rho$ to be a generic $\KK$-linear functional on $\KK^{E'}$ that vanishes on $\KK^{E\minus E'}$; then the bases of the new matroid are
the sets
\[
\{B = \{\overline{v}_{e_{1}}, \dots, \overline{v}_{e_{\rank(\matM) - 1}}\}\mid \exists e^{\prime} \in E^{\prime} \text{ such that }  B \cup \overline{v}_{e^{\prime}} \text{ is linearly independent}\}.
\]

The general definition does not require representability:

\begin{define} \label{def - matroid quotient}
    Let $\matM$ be a matroid on a ground set $E$. We say a matroid $\matN$ on $E$  \emph{is a quotient of} $\matM$, denoted $\matM \twoheadrightarrow \matN$, if every circuit of $\matM$ is a union of circuits of $\matN$. 
\end{define}
This indeed is a generalization: a (minimal) linear dependency descends to the quotient by $\KK\rho$; it might just become a sum of smaller minimal linear dependencies.
\begin{example} \label{ex - truncations are quotients}
    Given a matroid $\matM$ on $E$, its \emph{truncation} is the matroid $\tau(\matM)$ on $E$ whose independent sets are
    \[
    \calI_{\tau(\matM)} = \{ I \in \calI_{\matM} \mid \rank_{\matM}(I) \leq \rank(\matM) - 1\}.
    \]
    The truncation $\tau(\matM)$ is a quotient of $\matM$ (\cite[pg.~277]{OxleyMatroidTheory}) and corresponds to the case $E=E'$ in the quotient construction in the representable case. If $\matM$ is loopless and $\rank(\matM) \geq 2$, then any two elements $e,f\in E$ are in a common circuit of $\tau(\matM)$ and hence  $\tau(\matM)$ is connected. (Indeed,  $\{e, f\}$ is not already a circuit, hence in $\calI_\matM$ then any $\matM$-basis containing it will be a circuit in $\tau(\matM)$). 
    
    Applying the truncation operation $t$-times gives the \emph{$t$-trunctation} matroid $\tau^{t}(\matM):=\tau(\tau^{t-1}(\matM))$ on $E$ characterized by
    \[
    \calI_{\tau^{t}(\matM)} = \{I \in \calI_{\matM} \mid \rank_{\matM}(I) \leq \rank(\matM) - t \}.
    \]
\end{example}

\begin{remark} \label{rmk-mat-quots} Let $\matN$ be a quotient of $\matM$ on ground set $E$.
    
\noindent 
\begin{enumerate}[label=(\alph*)]
    \item Let $I \in \calI_{\matN}$. If $I$ were dependent on $\matM$ then it would contain a circuit of $\matM$. But then $I$ should contain a circuit of $\matN$, so that 
        \[
        \calI_{\matN} \subseteq \calI_{\matM}.
        \]
        \item\label{item-mat-quot-b} In particular,
        \[
        \rank(\matN) \leq \rank(\matM),
        \]
        with equality precisely if $\matM=\matN$ by 
        \cite[Cor.~7.3.4]{OxleyMatroidTheory}.
        \item\label{item-mat-quot-c} Let $H$ be a handle of $\matN$ that meets a circuit $C$ of $\matM$. Since $C$ is a union of circuits of $\matN$, $H$ must meet (hence, be inside) one of these $\matN$-circuits, and in particular inside $C$:
        \[
        \{\text{handles of } \matN\} \subseteq \{\text{handles of }\matM\} \]
        \item This is true in particular for coloops:
        \[
        \{\text{coloops of } \matN \} \subseteq \{\text{coloops of } \matM \}
        \]
        \item Loops (of $\matM$) are circuits, and thus: 
        \[
        \{\text{loops of } \matN\} \supseteq \{\text{loops of } \matM\}.
        \]
        \item\label{item-dual-DC} By  \cite[Lem.~7.3.3]{OxleyMatroidTheory}, $\matN$ is a quotient of $\matM$ if and only if for some set $E'$  there exists a matroid $\matQ$ on $E \sqcup E^{\prime}$ where: $\rank(\matQ) = \rank(\matM)$; $E^{\prime} \in \calI_{\matQ}$; $\matM = \matQ \setminus E^{\prime}$; $\matN = \matQ / E$. We call $\matQ$ a \emph{lift} of the quotient $M \twoheadrightarrow \matN$. By duality of Deletion-Contraction,
        \[
        \matN^{\perp} = (\matQ / E^{\prime})^{\perp} = \matQ^{\perp} \setminus E^{\prime} \quad \text{ and } \quad \matM^{\perp} = (\matQ \setminus E^{\prime})^{\perp} = \matQ^{\perp} / E^{\prime}.
        \]
        %Further: $\rank(\matN^{\perp}) = \rank(\matQ^{\perp} \setminus E^{\prime}) = \rank_{\matQ^{\perp}}(E) = |E| + \rank_{\matQ}(E^{\prime}) -\rank(\matQ) = \rank(\matQ^{\perp})$. And: $\rank_{\matQ^{\perp}}(E^{\prime}) = |E^{\prime}| + \rank_{\matQ}((E \setminus E^{\prime}) - \rank(\matQ) = |E^{\prime}|$. That is, 
        One checks that this yields:
        \[
        \matM^{\perp} \text{ is a quotient of } \matN^{\perp} \text{ with lift }\matQ^{\perp}.
        \]
        \item\label{item-quot-del-contr}
        Let $\matN$ be a matroid quotient of $\matM$ on the ground set $E$. For $e \in E$, one may easily check with the help of \cite[Prop.~7.3.6]{OxleyMatroidTheory} that
    \[
    \matN / e \text{ is a quotient of } \matM/e \quad \text{\normalfont and } \quad \matN \setminus e \text{ is a quotient of } \matM\setminus e.
    \]
    \end{enumerate}
\end{remark}

%That the quotient relationship is well-behaved with Deletion-Contraction is well known, but we highlight it due to its importance:
%
%\begin{proposition} \label{prop - quotients with respect to delete, contract}
%    Let $\matN$ be a matroid quotient of $\matM$ on the ground set $E$. For $e \in E$,
%    \[
%    N / e \text{ is a quotient of } \matM/e \quad \text{\normalfont and } \quad \matN \setminus e \text{ is a quotient of } \matM\setminus e.
%    \]
%\end{proposition}
%
%\begin{proof}
%    By Proposition 7.3.6 \cite{OxleyMatroidTheory}, $\matN$ is a matroid quotient of $\matM$ on ground set $E$ if and only 
 %   \begin{equation} \label{eqn - matroid quotient, rank of subsets charaterization}
 %   [X \subseteq Y \subseteq E] \implies [\rank_{\matM}(Y) - \rank_{\matM}(X) \geq \rank_{\matN}(Y) - \rank_{\matN}(X)].
 %   \end{equation}
  %  That $\matN/e$ is a quotient of $\matM/e$ follows from the rank formula for contractions: $\rank_{\matM/e}(X) = \rank_{\matM}(X \cup e) - \rank_{\matM}(e)$. That $\matN \setminus e$ is a quotient of $\matM\setminus e$ follows similarly from \eqref{eqn - matroid quotient, rank of subsets charaterization}.
%\end{proof}

\subsection{Flag Matroids and Flag Matroidal Polynomials}
%%%%%%%%%%%%%%%%%%%%%%%%%%%%%%%%%%%%%%%%%%%%%%%%%%%%%%%%%%

\begin{define} \label{def - flag matroid}
    A sequence of matroids $\mathscr{M} = (\matM_{k}, \dots, \matM_{1})$ on the shared ground set $E$ with $\rank(\matM_1)\geq 1$ is a \emph{flag matroid of length $k$} provided that $\matM_{i}$ is a quotient of $\matM_{i+1}$ for all $i \in \ZZ_{\geq 1}$, 
    \[
    \matM_{k} \twoheadrightarrow \matM_{k-1} \twoheadrightarrow \cdots \twoheadrightarrow \matM_{2} \twoheadrightarrow \matM_{1}.
    \]
    The flag matroid $\mathscr{M}$ is  \emph{terminally strict} if $k=1$, or if
    \[
    \rank(\matM_{2}) > \rank(\matM_{1}).
    \]

    When $\matM_{1}$ is connected, $\mathscr{M}$ is a \emph{terminally connected} flag matroid; when $\matM_{1}$ is loopless, $\mathscr{M}$ is a \emph{terminally loopless} flag matroid. The \emph{$t$-abutment} of $\mathscr{M}$ is the length $t$ flag matroid $\mathscr{M}_{\leq t} = (\matM_{t}, \dots, \matM_{1})$.

    We define the class of \emph{flag matroidal polynomials} on the flag matroid $\mathscr{M}$ by
    \[
    \FMP(\mathscr{M}) = \{ \zeta_{\mathscr{M}} = \sum_{1 \leq i \leq k} \zeta_{\matM_{i}} \mid \zeta_{\matM_{i}} \text{ a matroid support polynomial on } \matM_{i} \}.
    \]
    Flag matroidal polynomials on $\mathscr{M}$ are usually denoted by $\zeta_{\mathscr{M}}$ as above. The \emph{$t$-abutment} of a flag matroidal polynomial is the flag matroidal polynomial $\zeta_{\mathscr{M}_{\leq t}} = \sum_{i \leq t} \zeta_{\matM_{i}} \in \FMP(\mathscr{M}_{\leq t})$ on $\mathscr{M}_{\leq t} = (\matM_{t}, \matM_{t-1}, \dots, \matM_{1})$.

    Finally, we let $X_{\zeta_{\mathscr{M}}} \subseteq \KK^{E}$ be the hypersurface defined by $\zeta_{\mathscr{M}}.$
\end{define}

\begin{remark}
\begin{asparaenum}
\item For $S \subseteq E$ the flag matroid $\mathscr{M}$ induces flag matroids
    \[
    \mathscr{M} \setminus S = (\matM_{k} \setminus S, \dots, \matM_{1} \setminus S) \quad \text{ and } \quad \mathscr{M} / S = (\matM_{k} / S, \dots, \matM_{1} / S),
    \]
    see Remark \ref{rmk-mat-quots} \ref{item-quot-del-contr}. 
    \item 
    The  monomial support of $\zeta_{\mathscr{M}} \in \FMP(\mathscr{M})$ is certainly contained in the union $\calB_{\matM_{i}}$, but the containment may be strict, due to possible cancellation. For example, when $\scrM = (\matM, \matM)$, the difference of matroid basis polynomials  $\Psi_{\matM} - \Psi_{\matM} = 0$ belongs to $\FMP(\mathscr{M})$. Note that when $\mathscr{M}$ is terminally strict and $\zeta_{\mathscr{M}} \in \FMP(\mathscr{M})$, we have $\min \zeta_{\mathscr{M}} = \zeta_{\matM_{1}} \neq 0$, since each $\zeta_{\matM_{i}}$ is homogeneous of degree $\rank(M_{i})$.
    \end{asparaenum}
\end{remark}

%Recall our main objective is to prove that under mild assumptions, the flag matroidal polynomial $\zeta_{\mathscr{M}}$ attached to a flag matroid $\mathscr{M}$ has rational singularities. 
\begin{example} \label{ex - repeated trunctations, monomial support on independent sets}
    Let $\matM$ be a matroid on $E$. Then for all $1 \leq s \leq \rank(\matM) - 1$ we have a length $s+1$ terminally strict flag matroid $\mathscr{M}$ given by repeated truncation:
    \[
    \mathscr{M} = (\matM, \tau(\matM), \tau^{2}(\matM), \dots, \tau^{s}(\matM)).
    \]
%    Let 
%    \[
%    \rho_{\tau^{s+1 - i} (\matM)} = \sum_{B \in \calB_{\tau^{s+1 - i} (\matM)}} \alpha_{B, i} \bsx^{B} = \sum_{\substack{I \in \calI_{\matM} \\ \rank(I) = \rank(\matM) - (s+1) - i}} \alpha_{I, i} \bsx^{I}
%    \]
%    be an arbitrary matroid support polynomial on $\tau^{s+1 - i} (\matM)$. Then
%    \[
%    \sum_{0 \leq i \leq s} \rho_{\tau^{s+1 - i} (\matM)} = \sum_{\rank(\matM) - s \leq \ell \leq \rank(\matM)} \sum_{\substack{I \in \calI_{\matM} \\ \rank_{\matM}(I) = \ell}} \alpha_{\rank(\matM) - (s+1) - \ell, I} \bsx^{I}
 %   \]
%    is a flag matroidal polynomial on $\mathscr{M}$. Write 
%    \[
%    \calI_{\matM}^{\geq \rank(\matM) - s} = \{I \in \calI_{\matM} \mid \rank(\matM) - s \leq \rank(I) \leq \rank(\matM) \}.
%    \]
%   The above shows that when $\mathscr{M}$ is given by $s$ repeated truncations of $\matM$, 
%    \[
%    \FMP(\mathscr{M}) = \{\text{polynomials whose monomial support is } \calI_{\matM}^{\geq \rank(\matM) - s} \}.
%    \]
    A matroid support polynomial on $\tau^t(\matM)$ has monomial support $\{I \in \calI(\matM) \mid \rank(I) = \rank(\matM) - t\}$. So the sum of matroid support polynomials on $\matM, \tau(\matM), \dots, \tau^s(\matM)$ has monomial support $\calI_\matM^{\geq \rank(\matM) - s}$. Conversely, any polynomial with monomial support $\calI_\matM^{\geq \rank(\matM) - s}$ can be reverse enginered as a sum of matroid support polynomials on $\matM, \tau(\matM), \dots, \tau^s(\matM)$. We conclude that 
    \[
    \FMP(\mathscr{M}) = \{\text{polynomials whose monomial support is } \calI_{\matM}^{\geq \rank(\matM) - s} \}.
    \]
\end{example}

\begin{remark} \label{rmk - basics about flag matroidal polys} Let $\mathscr{M}$ be a length $k$ flag matroid, $\zeta_{\mathscr{M}}$ a flag matroidal polynomial.
\begin{enumerate}[label=(\alph*)]
    \item Assume that $\mathscr{M}$ is terminally strict and $\rank(\matM_{1}) \geq 2$. Since $\rank(\matM_{i}) \geq 2$ for all $i$ and $\zeta_{\matM_{i}}$ is homogeneous of degree at least $2$,  $\zeta_{\mathscr{M}} \in (\{x_{e}\}_{e \in E})^{2}$. As $\min \zeta_{\mathscr{M}} = \zeta_{\matM_{1}} \neq 0$, we see that $\zeta_{\mathscr{M}}$ is singular at $0$.
 
    \item\label{item-flag-mat-b} Assume that $\mathscr{M}$ is terminally strict. Then $\min(\zeta_{\mathscr{M}}) = \zeta_{\matM_{1}}$ and $\zeta_{\mathscr{M}} - \zeta_{\matM_{1}}$ is a sum of squarefree monomials ranging in degree from $\rank(\matM_{2})$ to $\rank(\matM_{k})$. Thus, $\rank \matM_{1} \leq \deg \zeta_{\mathscr{M}} \leq \rank(\matM_{k})$,  and $\rank \matM_{1} = \deg \zeta_{\mathscr{M}}$ exactly when $\zeta_{\mathscr{M}} - \zeta_{\matM_{1}} = 0$.

    \item\label{item-flag-mat-c} If $\mathscr{M}$ is terminally connected and terminally strict, then $\zeta_{\mathscr{M}}$ is irreducible. Indeed, let $ab = \zeta_{\mathscr{M}}$ be a factorization. Then, $\min(a)\min(b) = \min(\zeta_{\mathscr{M}}) = \zeta_{\matM_{1}}$ is irreducible and hence we may assume $\min(a) = \zeta_{\matM_{1}}$ and $\min(b) = 1$. Since $\zeta_{\mathscr{M}}$ consists of square-free monomials, the factorization $ab$ induces a partition $A \sqcup B \subseteq E$ where $a \in \KK[A]$ and $b \in \KK[B]$, \emph{cf.}\ Remark \ref{rmk - basic matroid polynomial observations}.\eqref{item-partition}. Since $\matM_{1}$ is connected, and $\zeta_{\matM_{1}} \in \MSP(\matM_{1})$, every $e \in E$ appears in one of the monomials of $\zeta_{\matM_{1}}$. It folows that  $A = E$ which then forces $b\in\KK^\times$.
 
    \item\label{item-strict-contr/del} Assume that $\mathscr{M}$ is terminally strict. If $I \in \calI_{\matM_{1}}$, then both $\mathscr{M} / I$ and $\scrM\setminus I$ are terminally strict. Indeed, we find $I \in \calI_{\matM_{2}}$ and 
    \[
    \rank(\matM_{2} / I) = \rank(\matM_{2}) - |I| > \rank(\matM_{1}) - |I| = \rank(\matM_{1} / I),
    \]
     while a basis $B_1$ of $\matM_1$ that contains $I$ can be extended to a basis of $\matM_2$, implying terminal strictness of the deletion.
    
%    \item\label{item-strict-deletion} We have Deletion-Contraction on flag matroidal polynomials. We state one case. Take $e \in E$ with $e \in \calI_{\matM_{1}}$. Let $r$ be the smallest index such that $e$ is a coloop in $\matM_{k}$; if no such index exists set $r = k+1$. Then we can rewrite
%    \begin{align} \label{eqn - delete contract, flag matroidal polynomials}
%    \zeta_{\mathscr{M}} 
%    &= \sum_{1 \leq i \leq r-1} \zeta_{\matM_{i}} + \sum_{r \leq t \leq k} \zeta_{\matM_{t}} \\
%    &= \underbrace{\left(\sum_{1 \leq i \leq r-1} \zeta_{\matM_{i} \setminus e} + x_{e}\zeta_{\matM_{i} / e} \right)}_{
%    =:\zeta_{(\mathscr{M} \setminus e)_{\leq r-1}}}    
%    + x_e\cdot \underbrace{\left( \sum_{r \leq t \leq k} \zeta_{\matM_{t} / e} \right)}_{
%    =:\zeta_{\mathscr{M} / e}}    \nonumber %\\
 %%   &= \zeta_{(\mathscr{M} \setminus e)_{\leq r-1}} + x_{e} \zeta_{\mathscr{M} / e} \nonumber
%    \end{align}
%   since $\calI_{\matM_{i}} \subseteq \calI_{\matM_{i+1}}$. Here, $\zeta_{\matM_{i} \setminus e}$ and $\zeta_{\matM_{i} / e}$ are the matroidal polynomials induced by Deletion-Contraction identities and $\zeta_{\mathscr{M} \setminus e}$.
    %and $\zeta_{\mathscr{M} / e}$ are the flag matroidal polynomials induced by summing the induced matroidal polynomials. 
 %   
 %   Note that if $\mathscr{M}$ as above happened to be  terminally strict, then independence of $e$ forces  $(\mathscr{M} \setminus e)_{\leq r-1}$ and $\scrM/e$ to be terminally strict  as well.
\end{enumerate}
\end{remark}

\begin{define}\label{dfn-flag-poly-contr/del}
Let $\mathscr{M}$ be a flag matroid of length $k$ and let $\zeta_{\mathscr{M}} \in \FMP(\mathscr{M})$. Take $e \in E$ with $e \in \calI_{\matM_{1}}$. Let $r$ be the smallest index such that $e$ is a coloop in $\matM_k$; if no such index exists set $r = k+1$. Each $\zeta_{\matM_i}$ satisfies a Deletion-Contraction identity with respect to $e$. Taken together this yields a Deletion-Contraction identity for $\zeta_{\mathscr{M}}$ with respect to $e$, on the level of flag matroidal polynomials:
    \begin{align} \label{eqn - delete contract, flag matroidal polynomials}
    \zeta_{\mathscr{M}} 
    &= \sum_{1 \leq i \leq r-1} \zeta_{\matM_{i}} + \sum_{r \leq t \leq k} \zeta_{\matM_{t}}\\ 
    &= \underbrace{\left(\sum_{1 \leq i \leq r-1} \zeta_{\matM_{i} \setminus e} + x_{e}\zeta_{\matM_{i} / e} \right)}_{\displaystyle
    =:\zeta_{(\mathscr{M} \setminus e)_{\leq r-1}}}    
    + x_e\cdot \underbrace{\left( \sum_{r \leq t \leq k} \zeta_{\matM_{t} / e} \right)}_{\displaystyle
    =:\zeta_{\mathscr{M} / e}}.   \nonumber %\\
    \end{align}
We refer to $\zeta_{(\matM \setminus e)_{\leq r - 1}} \in \FMP((\mathscr{M} \setminus e)_{\leq r - 1})$ and $\zeta_{(\matM / e)} \in \FMP(\mathscr{M} / e)$ as the \emph{induced deletion} (resp.\ \emph{contraction) flag matroidal polynomials}. 

When $\mathscr{M}$ happens to be terminally strict, the independence of $e$ forces $(\mathscr{M} \setminus e)_{\leq r- 1}$ and $\mathscr{M} / e$ to be terminally strict as well, by Remark \ref{rmk - basics about flag matroidal polys}.\ref{item-strict-contr/del}. 
\end{define}

%\begin{define} 
%    Let $\mathscr{M}$ be a flag matroid of length $k$ and let $\zeta_{\mathscr{M}}$ be a flag matroidal polynomial on $\mathscr{M}$. Let $e \in \calI_{\matM_{1}}$ and set $r$ to be the smallest index such that $e$ is a coloop on $\matM_{i}$; if no such index exists set $r = k+1$. The flag matroidal polynomials $\zeta_{(\mathscr{M} \setminus e)_{\leq r-1}} \in \FMP((\mathscr{M} \setminus e)_{\leq r-1})$ and $\zeta_{\mathscr{M} / e} \in \FMP(\mathscr{M} / e)$ that appear in \eqref{eqn - delete contract, flag matroidal polynomials} are called the \emph{induced deletion (resp. contraction) flag matroidal polynomials}. 
%\end{define}

Our first nontrivial result is that the $m$-jets of a flag matroidal polynomial attached to a terminally strict and terminaly loopless flag matroid are equidimensional for all $m \geq 1$. Our argument is as in the case of matroidal polynomials: we use the intersection theory identity Lemma \ref{lemma - basic intersection theory identity} along with a straightforward induction powered by Deletion-Contraction.

\begin{proposition} \label{prop - flag matroidal poly, delete contract intersect with gamma}
    Let $\zeta_{\mathscr{M}}$ be a flag matroidal polynomial on a length $k$ flag matroid $\mathscr{M}$. Select $e \in E$ that is neither a loop nor coloop on $\matM_{1}$ and let $r$ be the smallest index such that $e$ is a coloop on $\matM_{r}$; if no such index exists set $r = k+1$. Let $\Gamma \subseteq \scrL_{m}(\KK^{E})$ be the variety cut out by $(\{ D^{q}x_{e} \}_{0 \leq q \leq m})$. Then
    \[
    \scrL_{m}(\KK^{E}, X_{\zeta_{\mathscr{M}}}) \cap \Gamma \simeq \scrL_{m}(\KK^{E\minus\{e\}}, Y)
    \]
    where $Y \subseteq \KK^{E\minus\{e\}}$ is the hypersurface cut out by $\zeta_{(\mathscr{M} \setminus e )_{\leq r - 1}}.$
\end{proposition}

\begin{proof}
    By the Deletion-Contraction identity \eqref{eqn - delete contract, flag matroidal polynomials}, there is a $g \in \KK[E \setminus e]$ such that 
    \[
    \zeta_{\mathscr{M}} = \zeta_{(\mathscr{M} \setminus e )_{\leq r - 1}} + x_{e} g. 
    \]
    The result follows exactly as in Proposition \ref{prop - matroidal poly, m-jets intersect with Gamma}.
\end{proof}

\begin{proposition} \label{prop - flag matroids, jets equidimensional}
    Let $\zeta_{\mathscr{M}}$ be a flag matroidal polynomial on a terminally strict and terminally loopless flag matroid $\mathscr{M}$. Then $\scrL_{m}(\KK^{E}, X_{\zeta_{\mathscr{M}}})$ is an equidimensional l.c.i of dimension $(m+1)(n-1)$ for all $m \geq 0$.
\end{proposition}

\begin{proof}
    By Proposition \ref{prop - Mustata's characterization of irreducibility of m-jets in terms of dim lying over sing} it suffices to prove that 
    \begin{equation} \label{eqn - flag matroidal poly, equidimensional jets, dim criterion}
        \dim \scrL_{m}(\KK^{E}, X_{\zeta_{\mathscr{M}}}) \leq (m+1)(n-1) \text{ for all } m \geq 1,
    \end{equation}
    the case $m=0$ being self-evident. Suppose that $e\in E$ is not a (loop nor) coloop of $\matM_{1}$; then $\rank(\matM_1\setminus e)\geq 1$. Let $r$ be the smallest index such that $e$ is a coloop on $\matM_{r}$; if no such index exists let $r = k+1$.  Let $Y \subseteq \KK^{E\minus\{e\}}$ be the variety cut out by $\zeta_{(\mathscr{M} \setminus e)_{\leq r-1}}$ and let $\Gamma \subseteq \scrL_{m}(\KK^{E})$ be the variety cut out by $(\{D^{q} x_{e} \}_{0 \leq q \leq m})$. Then
    \begin{align} \label{eqn - flag matroidal poly, equidimensional jets, part 1}
        \dim \scrL_{m}(\KK^{E}, X_{\zeta_{\mathscr{M}}})
        &\leq \dim \left( \scrL_{m}(\KK^{E}, X_{\zeta_{\mathscr{M}}}) \cap \Gamma \right) + \codim \Gamma \\
        &= \dim \scrL_{m}(\KK^{E\minus\{e\}}, Y) + \codim \Gamma, \nonumber
    \end{align}
    by Lemma \ref{lemma - basic intersection theory identity} (which applies since $\scrL_{m}(\KK^{E\minus\{e\}}, Y) \neq \emptyset)$ and  Proposition \ref{prop - flag matroidal poly, delete contract intersect with gamma}. 
    
    If \eqref{eqn - flag matroidal poly, equidimensional jets, dim criterion} holds for $Y$, then we can extend \eqref{eqn - flag matroidal poly, equidimensional jets, part 1} by adding the line 
    \[
    \leq (m+1)(n-2) + (m+1) = (m+1)(n-1).
    \]
    So if $\eqref{eqn - flag matroidal poly, equidimensional jets, dim criterion}$ holds for such a $Y$ then it holds for $X_{\zeta_{\mathscr{M}}}$ as well. 
    
    The above gives the inductive step of a proof of \eqref{eqn - flag matroidal poly, equidimensional jets, dim criterion} via induction on $|E|$ since looplessness of $\matM_{1}$ persists under deletion, and  $\mathscr{M}$ being terminally strict implies the same for $(\mathscr{M} \setminus e)_{\leq r-1}$.
    
    The base case for this induction occurs  when $\matM_{1}$ is Boolean: then 
    $|E|=\rank(\matM_1)\le\rank(\matM_i)\le|E|$ for all $i$, forcing all $\matM_i$ to be equal; by strictness, $r=1$ and
     $\zeta_{\mathscr{M}} = c\bsx^{E}$ for $c \in \KK^{\times}$. That $\zeta_{\mathscr{M}}$ satisfies \eqref{eqn - flag matroidal poly, equidimensional jets, dim criterion} in this case is Remark \ref{rmk - jets of arrangements}.
\end{proof}

\subsection{Rationality of Flag Matroidal Polynomials}
%%%%%%%%%%%%%%%%%%%%%%%%%%%%%%%%%%%%%%%%%%%%%%%%%%%%%%%%%%

We are now ready to prove flag matroidal polynomials have rational singularities given mild conditions on the associated flag matroid. As before we will use a handle style induction to show that the $m$-jets are irreducible, with the handles coming from the terminal matroid $\matM_{1}$ of the flag $\mathscr{M}$. 

The first step is to give a handle formula for $\zeta_{\mathscr{M}}$ based on a handle $H$ of $\matM_{1}$. 

\begin{proposition} \label{prop - new handle fml, flag matroidal poly}
    Let $\zeta_{\mathscr{M}}$ be a flag matroidal polynomial on the length $k$ flag matroid $\mathscr{M}$. Suppose that $H \in \calI_{\matM_{1}}$ is a proper handle on $\matM_{1}$ such that $\rank(\matM_{1} / H) \geq 1$. Let $r$ be the smallest index such that $H$ contains a coloop on $\matM_{r}$; if no such index exists let $r = k+1$. Then there exist flag matroidal polynomials $\zeta_{\mathscr{M} / H} \in \FMP(\mathscr{M} / H)$ and $\zeta_{(\mathscr{M} \setminus H)_{< r}}^{h} \in \FMP((\mathscr{M} \setminus H)_{< r})$, for each $h \in H$, such that
    \begin{equation} \label{eqn - new handle fml, flag matroidal poly}
        \zeta_{\mathscr{M}} = \left( \sum_{h \in H} \bsx^{H\minus\{h\}} \zeta_{(\mathscr{M} \setminus H)_{< r}}^{h} \right) + \bsx^{H} \zeta_{\mathscr{M} / H}
    \end{equation}
\end{proposition}

\begin{proof}
    By Remark \ref{rmk-mat-quots}.\eqref{item-mat-quot-c}, $H$ is a handle on each $\matM_{i}$. Thus, as soon as $H$ contains a coloop on $\matM_{i}$ it consists of only coloops on $\matM_{i}$. Thus: if $i < r$, the handle $H \in \calI_{\matM_{i}}$ on $\matM_{i}$ contains no coloops; for $i \geq r$, the set $H$ contains only coloops on $\matM_{i}$. By using Proposition \ref{prop - new handle fml - matroidal poly} we compute
    \begin{equation} \label{eqn - new handle fml, flag matroidal poly, 1}
        \sum_{i < r} \zeta_{\matM_{i}} = \sum_{i < r} \left[ \left( \sum_{h \in H} \bsx^{H\minus\{h\}} \zeta_{\matM_{i} \setminus H}^{h} \right) + \bsx^{H} \zeta_{\matM_{i} / H} \right],
    \end{equation}
    where $\zeta_{\matM_{i} / H}$ and $\zeta_{\matM_{i} \setminus H}^{h}$, for each $h \in H$, are the handle matroidal polynomials of Definition \ref{def - handle matroidal polynomials}. As for $i \geq r$, by using Deletion-Contraction for matroidal polynomials with respect to coloops, we obtain
    \begin{equation} \label{eqn - new handle fml, flag matroidal poly, 2}
        \sum_{i \geq r} \zeta_{\matM_{i}} = \sum_{i \geq r} \bsx^{H} \zeta_{\matM_{i} / H}
    \end{equation}
    where each $\zeta_{\matM_{i}/ H} \in \MSP(\matM_{i} / H)$ is obtained by repeated application of contraction. Set, with $h \in H$, 
    \begin{equation} \label{eqn - handle flag matroidal polys}
        \zeta_{(\mathscr{M} \setminus H)_{< r}}^{h} := \sum_{1 \leq i < r} \zeta_{\matM_{i} \setminus H}^{h} \quad \text{ AND } \quad \zeta_{\mathscr{M} / H} := \sum_{1 \leq i \leq k} \zeta_{\matM_{i} / H}.
    \end{equation}
    Then $\zeta_{(\mathscr{M} \setminus H)_{< r}}^{h} \in \FMP((\mathscr{M} \setminus H)_{< r})$, for all $h \in H$, and $\zeta_{\mathscr{M} / H} \in \FMP(\mathscr{\matM}/H)$ (since $\rank(\matM_1 \setminus H) \geq \rank(\matM_1 / H) \geq 1$ by hypothesis).  As $\zeta_{\mathscr{M}} = \sum_{1 \leq i \leq r} \zeta_{\matM_{i}}$, combining \eqref{eqn - handle flag matroidal polys}, \eqref{eqn - new handle fml, flag matroidal poly, 1} and \eqref{eqn - new handle fml, flag matroidal poly, 2} yields \eqref{eqn - new handle fml, flag matroidal poly}.
\end{proof}

\begin{define} \label{def - handle flag matroidal polys}
    In the set-up of Proposition \ref{prop - new handle fml, flag matroidal poly}, we call the flag matroidal polynomials $\zeta_{\mathscr{M} / H} \in \FMP(\mathscr{M}/H)$ and $\zeta_{(\mathscr{M} \setminus H)_{<r}}^{h} \in \FMP((\mathscr{M} \setminus H)_{<r})$, that appear in \eqref{eqn - handle flag matroidal polys} the \emph{handle flag matroidal polynomials} (with respect to $H$) of $\zeta_\scrM$. 
\end{define}

We next prove an analogue to Lemma \ref{lemma - matroidal poly handle fml, intersecting induced polys}. 

\begin{lemma} \label{lemma - flag matroidal poly handle fml, intersecting induced polys}
    Let $\zeta_{\mathscr{M}}$ be a flag matroidal polynomial on the length $k$ flag matroid $\mathscr{M}$. Assume that $\mathscr{M}$ is terminally connected and terminally strict. Also assume $H \in \calI_{\matM_{1}}$ is a proper handle on $\matM_{1}$ such that $\matM_{1} \setminus H$ is connected and $\rank(\matM_{1} / H) \geq 1$. Then, for each $h \in H$, we have the following dimension bound for hypersurfaces attached to handle flag matroidal polynomials:
    \[
    X_{\zeta_{(\mathscr{M} \setminus H)_{< r}}^{h}} \cap X_{\zeta_{\mathscr{M} / H}} \subseteq \KK^{E\minus H} \text{ has dimension at most } |E| - |H| - 2.
    \]
\end{lemma}

\begin{proof}
    Since $1 \leq \rank(\matM_{1} / H) = \rank(\matM_{1}) - |H|$  we cannot have $|E\minus H| \leq 1$, lest $\rank(M_{1}) = |E|$ in violation of the connectedness of $\matM_{1}$. 
    
    Fix $h \in H$. By construction, $\mathscr{M} / H$ and $(\mathscr{M} \setminus H)_{< r}$ are both terminally strict, \emph{cf.}\ Remark \ref{rmk - basics about flag matroidal polys}.\ref{item-strict-contr/del} and the definition of $r$; by hypothesis, $(\mathscr{M} \setminus H)_{<r}$ is terminally connected. So $\min \zeta_{(\mathscr{M} \setminus H)_{< r}}^{h} = \zeta_{(\matM_{1} \setminus H)^{h}}^{h}, \min \zeta_{\mathscr{M} / H} = \zeta_{\matM_{1} / H}$, and $\zeta_{(\mathscr{M} \setminus H)_{< r}}^{h}$ is irreducible, \emph{cf.}\ Remark \ref{rmk - basics about flag matroidal polys}.\ref{item-flag-mat-b},\ref{item-flag-mat-c}. By irreducibility, the only way the dimension bound can fail is if $\zeta_{\mathscr{M} / H}$ is divisible by $\zeta_{(\mathscr{M} \setminus H)_{< r}}^{h}$. But then $\min \zeta_{\mathscr{M} / H}$ is divisible by $\min \zeta_{(\mathscr{M} \setminus H)_{< r}}^{h}$ which is impossible as the former has degree $\rank(\matM_{1} / H) = \rank(\matM_{1}) - |H|$ and the latter has degree $\rank(\matM_{1} \setminus H) = \matM_{1} - (|H| - 1)$. So the dimension bound does hold.
\end{proof}

As in the case of matroidal polynomials, instead of studying the dimension of $\scrL_{m}(X_{\zeta_{\mathscr{M}}}, \Var(\mathfrak{p}))$, for $\mathfrak{p}$ a minimal prime of the true Jacobian ideal of $\zeta_{\mathscr{M}}$, it will be both easier and sufficient to replace $\Var(\mathfrak{p})$ with certain hypersurfaces. This is a straightforward application of our new handle formula.

\begin{proposition} \label{prop - min primes of jacobian, flag matroidal poly}
    Let $\zeta_{\mathscr{M}}$ be a matroidal polynomial on a length $k$ flag matroid $\mathscr{M}$. Let $H \in \calI_{\matM_{1}}$ be a proper handle on $\matM_{1}$ that contains no coloops on $\matM_{1}$. Furthermore, suppose that $\rank(\matM_{1} / H) \geq 1$. If $\mathfrak{p} \in \Spec \KK[E]$ is a prime containing the true Jacobian $(\partial \zeta_{\mathscr{M}}) + (\zeta_{\mathscr{M}})$ of $\zeta_{\mathscr{M}}$, then we have one of the following memberships:
    \begin{enumerate}[label=(\alph*)]
        \item $\mathfrak{p} \ni x_{h} \text{ for some } h \in H$;
        \item $\mathfrak{p}$ contains the handle flag polynomial  $\zeta_{\mathscr{M} / H}$ of $\zeta_\scrM$.
    \end{enumerate}
\end{proposition}

\begin{proof}
    Let $r$ be the smallest index such that $H$ contains a coloop on $\matM_r$; if no such index exists set $r = k+1$. By Proposition \ref{prop - new handle fml, flag matroidal poly}, for each $h \in H$ we compute
    \[
    \mathfrak{p} \supseteq (\partial \zeta_{\mathscr{M}}) + (\zeta_{\mathscr{M}}) \ni \zeta_{\mathscr{M}} - (x_{h} \partial_{x_{h}} \bullet \zeta_{\mathscr{M}}) = \bsx^{H\minus\{h\}} \zeta_{(\mathscr{M} \setminus H)_{< r}}^{h}.
    \]
    So again using Proposition \ref{prop - new handle fml, flag matroidal poly},
    \[
    \mathfrak{p} \ni \zeta_{\mathscr{M}} - \left(\sum_{h \in H} \zeta_{\mathscr{M}} - (x_{h} \partial_{x_{h}} \bullet \zeta_{\mathscr{M}}) \right) = \bsx^{H} \zeta_{\mathscr{M} / H}.
    \]
    Now use primality of $\mathfrak{p}$.
\end{proof}

As in the case of matroidal polynomials, we next intersect $\scrL_{m}(\KK^{E}, X_{\zeta_{\mathscr{M}}})$ with a well chosen variety based on whatever handle we are working with, leading to an inductive set-up.

\begin{lemma} \label{lemma - jets, flag matroidal polynomial, handle intersect with gamma formula}
    Let $\zeta_{\mathscr{M}}$ be a flag matroidal polynomial on the length $k$ terminally strict flag matroid $\mathscr{M}$. Suppose that $H \in \calI_{\matM_{1}}$ is a proper handle on $\matM_{1}$ that contains no coloops on $\matM_{1}$ while $\rank(\matM_{1} / H) \geq 1$. Fix $h \in H$ and let $\xi_{h} = \bsx^{H\minus\{h\}} \zeta_{(\mathscr{M} \setminus H)_{< r}}^{h} \in \KK[{E\minus\{h\}}].$ Furthermore, let $\Gamma \subseteq \scrL_{m}(\KK^{E})$ be the variety cut out by $(\{D^{q} x_{h} \}_{0 \leq q \leq m})$. Then
    \begin{equation} \label{eqn - jets, flag matroidal poly, handle intersect with gamma fml}
    \scrL_{m}(\KK^{E}, X_{\zeta_{\mathscr{M}}}, \Var(\zeta_{\mathscr{M} / H})) \cap \Gamma \simeq \scrL_{m}(\KK^{{E\minus\{h\}}}, \Var(\xi_{h}), \Var(\zeta_{\mathscr{M}/H})) \neq \emptyset.
    \end{equation}
    Here $\zeta_{\mathscr{M} / H}, \zeta_{(\mathscr{M} \setminus H)_{< r}}^{h}$ are the handle flag matroidal polynomials from Definition \ref{def - handle flag matroidal polys}, while $r$ is the smallest index such that $H$ contains a coloop on $\matM_{r}$ (and is $k+1$ if no such index exists).
\end{lemma}

\begin{proof}
    First note that because $H$ contains no coloops on $\matM_{1}$, we know that $r > 1$. In particular, $\zeta_{h}$ is well-defined. By the handle formula for flag matroidal polynomials, \emph{cf.}\ Proposition \ref{prop - new handle fml, flag matroidal poly}, there exists a polynomial $g \in \KK[E]$ such that $\zeta_{\mathscr{M}} = \xi_{h} + x_{h} g$.  The justification for the ``$\simeq$'' part of \eqref{eqn - jets, flag matroidal poly, handle intersect with gamma fml} follows by similar argument as in Lemma \ref{lemma - matroidal poly, handle fml intersect with Gamma}. 
    
    As for the ``$\neq \emptyset$'' part of \eqref{eqn - jets, flag matroidal poly, handle intersect with gamma fml}, because $\rank(\matM_{1} / H) \geq 1$ and because $\mathscr{M} / H$ is terminally strict (Remark \ref{rmk - basics about flag matroidal polys}.(e)), we have
    \begin{equation} \label{eqn - flag matroidal poly, handle intersect with gamma, 1}
    0 \neq \zeta_{\mathscr{M} / H} \in (\{x_{e}\}_{e \in E\minus H}) \subsetneq \KK[{E\minus\{h\}}].
    \end{equation}
    Because $(\mathscr{M} \setminus H)_{< r}$ is terminally strict by definition of $r$ and because $\rank(\matM_{1} \setminus H) \geq \rank(\matM_{1} / H) \geq 1$, we see that
    \begin{equation} \label{eqn - flag matroidal poly, handle intersect with gamma, 2}
        0 \neq \xi_{h} \in (\zeta_{(\mathscr{M} \setminus H)_{< r}}^{h}) \subseteq (\{x_{e}\}_{e \in E\minus H}) \subsetneq \KK[{E\minus\{h\}}]. 
    \end{equation}
    Together \eqref{eqn - flag matroidal poly, handle intersect with gamma, 1} and \eqref{eqn - flag matroidal poly, handle intersect with gamma, 2} imply that 
    \begin{equation*} \label{eqn - flag matroidal poly, handle intersect with gamma, 3}
        \KK^{{E\minus\{h\}}} \supsetneq \Var(\xi_{h}) \cap \Var(\zeta_{\mathscr{M} / H}) \neq \emptyset.
    \end{equation*}
    The non-emptiness of $\scrL_{m}(\KK^{{E\minus\{h\}}}, X_{\zeta_{h}}, \Var(\zeta_{\mathscr{M} / H}))$ follows by Remark \ref{rmk - always a m-jet lying over a point in X}.
\end{proof}

Finally, we can use all this handle set-up to give our inductive proof the $m$-jets of $X_{\zeta_{\mathscr{M}}}$ are irreducible under mild hypotheses. While the nature of flag matroidal polynomials makes the argument more involved than its counterpart Theorem \ref{thm - matroid poly, connected implies m-jets irreducible}, the architecture of the two proofs are of a kind.

\begin{theorem} \label{thm - flag matroidal poly, m-jets irreducible}
    Consider a flag matroidal polynomial $\zeta_{\mathscr{M}}$ attached to a terminally connected and terminally strict flag matroid $\mathscr{M}$ of length $k$ on the ground set $E$. If $\rank(M_{1}) \geq 2$, then $\scrL_{m}(\KK^{E}, X_{\zeta_{\mathscr{M}}})$ is an irreducible, reduced l.c.i of dimension $(m+1)(n-1)$ for all $m \geq 1$. Moreover, $X_{\zeta_{\mathscr{M}}}$ is normal.
\end{theorem}

\begin{proof}
    By Proposition \ref{prop - Mustata's characterization of irreducibility of m-jets in terms of dim lying over sing} it suffices to certify that
    \begin{equation} \label{eqn - dim characterization irreducible m-jets, flag matroidal poly}
    \dim \scrL_{m}(\KK^{E}, X_{\zeta_{\mathscr{M}}}, X_{\zeta_{\mathscr{M}}, \Sing}) < (m+1)(n-1) \text{ for all } m \geq 1.
    \end{equation}
    We will induce on $|E|$. First we establish some cases preparing the inductive story. With this in mind, let $\mathfrak{p}$ be a minimal prime containing the true Jacobian $(\partial \zeta_{\mathscr{M}}) + (\zeta_{\mathscr{M}})$ of $\zeta_{\mathscr{M}}$. We must check that
    \begin{equation} \label{eqn - flag matroidal poly, jets irreducible, dim characterization in terms of min primes}
        \dim \scrL_{m}(\KK^{E}, X_{\zeta_{\mathscr{M}}}, \Var(\mathfrak{p})) < (m+1)(n-1).
    \end{equation}

    \vspace{3ex}

    \emph{Case 1}: If $\mathfrak{p} \ni x_{e}$ for some $e \in E$, then \eqref{eqn - flag matroidal poly, jets irreducible, dim characterization in terms of min primes} holds.

    \vspace{2ex}

    \emph{Proof}: Let $\Gamma \subseteq \scrL_{m}(\KK^{E})$ be the codimension $m$ variety defined by $(\{D^{q} x_{e} \}_{1 \leq q \leq m})$.  Let $r$ be the smallest index such that $e$ is a coloop on $\matM_r$; if no such index exists, then set $r = k+1$. Note that $r>1$ by hypothesis. Proposition \ref{prop - flag matroidal poly, delete contract intersect with gamma} and  Lemma \ref{lemma - basic intersection theory identity}
    \begin{align} \label{eqn - flag matroid poly, jets irreducible, case 1 part 1}
        \dim \scrL_{m}(\KK^{E}, X_{\zeta_{\mathscr{M}}}, \Var(\mathfrak{p})) 
        &\leq \dim \scrL_{m}(\KK^{E}, X_{\zeta_{\mathscr{M}}}, \Var(x_{e})) \\
        &\leq \dim \left( \scrL_{m}(\KK^{E}, X_{\zeta_{\mathscr{M}}}, \Var(x_{e})) \cap \Gamma \right) + \codim \Gamma \nonumber \\
        &= \dim \scrL_{m}(\KK^{E\minus\{e\}}, X_{\zeta_{(\mathscr{M} \setminus e)_{\leq r-1}}}) + \codim \Gamma. \nonumber
    \end{align}
    By the definition of $r$, the flag matroid $(\mathscr{M} \setminus e)_{\leq r - 1}$ is terminally strict; since $\rank(\matM_{1}) \geq 2$, certainly $\rank(\matM_{1} \setminus e) \geq 2$.
    %\dan{ r=2 reduces to matroidal polynomials so in principle that is okay as well} \dan{I think Case 1 is always fine for $\rank \matM_1 = 1$: here the argument just uses the equidimensionality of jets and you never have to wrestle with $\mathscr{M} / e$. The problem is Case 2 where you do have to deal with flag matroidal polys on $\mathscr{M} / e$. In \eqref{eqn - matroidal poly jets irreducible, case 2, part 1} one uses the intersection theory lemma Lemma \ref{lemma - basic intersection theory identity} to bound dimension. For this to be legal you have to be sure the two varieties you are intersecting meet. This is the \eqref{eqn - jets, flag matroidal poly, handle intersect with gamma fml} in Lemma \ref{lemma - jets, flag matroidal polynomial, handle intersect with gamma formula}. When $\rank(\matM_1 / H) \geq 1$, both $\zeta_h$ and $\zeta_{\mathscr{M} / H}$ vanish at 0 so we know this intersection is non-empty. When $\matM_1$ has rank 1, $\rank(\matM_1 / H) = 0$ and I don't know that this intersection is non-empty. I wouldn't know how to fix this: I have no control over $\mathscr{M} / H$ in this case since its terminal matroid is dependent. So to do the $\rank (\matM_1) = 1$ case you would either: find a way to guarantee Case 2 never happens; find a different argument for Case 2. Both frighten.} 
    In particular, the usage of Lemma \ref{lemma - basic intersection theory identity} is valid. Proposition \ref{prop - flag matroids, jets equidimensional} lets us extend \eqref{eqn - flag matroid poly, jets irreducible, case 1 part 1} with the line 
    \[
    = (m+1)(n-2) + \codim \Gamma = (m+1)(n-2) + m = (m+1)(n-1) - 1,
    \]
    verifying \eqref{eqn - flag matroidal poly, jets irreducible, dim characterization in terms of min primes}.   (When $r=2$, we are discussing a matroid support polynomial on $\scrM = \matM_1$, which is terminally strict by definition. These estimates were already established in the matroidal case, but Proposition \ref{prop - flag matroids, jets equidimensional} also applies.)

    \vspace{3ex}

    \emph{Case 2:} Suppose there exists a proper handle $H \in \calI_{\matM_{1}}$ such that $\matM_{1} \setminus H$ is connected and $\rank(\matM_{1} / H) \geq 1$. Assume that $\mathfrak{p} \ni \zeta_{\mathscr{M} / H}$. Let $r$ be the smallest index such that $H$ contains a coloop on $\matM_{r}$; if no such index exists, set $r = k+1$. By hypothesis, $r>1$. 
    %\uli{again, worry when $r=2$?} \dan{see my previous comment; same issue. I will note when $\matN = \matU_{1,n}$ the special handle part still exists in our set-up. Still the issue of Case 2 exists, but maybe more tractable in this situation.} 
    Finally assume that, for some $h \in H$, we know that \eqref{eqn - dim characterization irreducible m-jets, flag matroidal poly} holds for $\zeta_{(\mathscr{M} \setminus H)_{< r}}^{h}$ and that the dimension of $X_{\zeta_{(\mathscr{M} \setminus H\minus\{h\})_{< r}}^{h}} \cap X_{\zeta_{\mathscr{M} / H}} \subseteq \KK^{E\minus H}$ is at most $|E| - |H| - 2$. Then we claim \eqref{eqn - flag matroidal poly, jets irreducible, dim characterization in terms of min primes} holds for $\mathfrak{p}$. (Here we use the handle flag matroidal polynomials of Definition \ref{def - handle flag matroidal polys}.)

    \vspace{2ex}
    
    \emph{Proof}: This time let $\Gamma = \cap_{0 \leq q \leq m} \{D^{q} x_{h} = 0 \} \subseteq \mathscr{L}_m(\KK^E)$ for the particular $h \in H$ described in \emph{Case 2}'s hypotheses. So $\codim \Gamma = m+1$. Again by the intersection theory identity Lemma \ref{lemma - basic intersection theory identity} and Lemma \ref{lemma - jets, flag matroidal polynomial, handle intersect with gamma formula} we have
    \begin{align} \label{eqn - matroidal poly jets irreducible, case 2, part 1}
        \dim \scrL_{m}(\KK^{E}, X_{\zeta_{\mathscr{M}}}, \Var(\mathfrak{p})) 
        &\leq \dim \scrL_{m}(\KK^{E}, X_{\zeta_{\mathscr{M}}}, \Var(\zeta_{\mathscr{M}/H})) \\
        &\leq \dim \left( \scrL_{m}(\KK^{E}, X_{\zeta_{\mathscr{M}}}, \Var(\zeta_{\mathscr{M}/H})) \cap \Gamma \right) + \codim \Gamma \nonumber \\
        &= \dim  \scrL_{m}(\KK^{E\minus\{e\}}, \Var(\xi_{h}), \Var(\zeta_{\mathscr{M}/H})) + \codim \Gamma \nonumber.
    \end{align}
    Here $\xi_{h} = \bsx^{H\minus\{h\}} \zeta_{(\mathscr{M} \setminus H)_{< r}}^{h} \in \mathbb{C}[{E\minus\{h\}}]$ is as defined in Lemma \ref{lemma - jets, flag matroidal polynomial, handle intersect with gamma formula}, and the usage of Lemma \ref{lemma - basic intersection theory identity} is legal by Lemma \ref{lemma - jets, flag matroidal polynomial, handle intersect with gamma formula}.

    It suffices to prove that
    \begin{align} \label{eqn - matroidal poly jets irreducible, case 2, part 2}
        \dim  \scrL_{m}(\KK^{{E\minus\{h\}}}, \Var(\xi_{h}), \Var(\zeta_{\mathscr{M}/H})) < (m+1)(n-2),
    \end{align}
    for then substituting \eqref{eqn - matroidal poly jets irreducible, case 2, part 2} into \eqref{eqn - matroidal poly jets irreducible, case 2, part 1} gives \eqref{eqn - flag matroidal poly, jets irreducible, dim characterization in terms of min primes}. As $\bsx^{H\minus\{h\}} \in \KK[H\minus\{h\}]$ and $\zeta_{(\mathscr{M} \setminus H)_{< r}}^{h} \in \KK[E\minus H]$, we are in the same situation as \emph{Case 2} of Theorem \ref{thm - matroid poly, connected implies m-jets irreducible}. (For the conversion, replace $\xi_h$ with the $Q_h$ of \emph{Case 2}, Theorem \ref{thm - matroid poly, connected implies m-jets irreducible}.) The reasoning therein 
    %to Lemma \ref{lemma - m-jet codim for products of eqs disjoint variables} gives an explicit description of the irreducible components of $\scrL_{m}(\KK^{{E\minus\{h\}}}, X_{\xi_{h}})$. As in \emph{Case 2} of Theorem \ref{thm - matroid poly, connected implies m-jets irreducible} this description 
    implies that in order to show \eqref{eqn - matroidal poly jets irreducible, case 2, part 2}, it suffices to show (for all $0 \leq t \leq m$) that
    \begin{align} \label{eqn - matroidal poly jets irreducible, case 2, part 3}
        \dim \left( \scrL_{t}(\KK^{E\minus H}, X_{\zeta_{(\mathscr{M} \setminus H)_{<r}}^{h}}) \cap \pi_{t}^{-1} \Var( \zeta_{\mathscr{M} / H}) \right) 
        < (t+1)(n-|H| - 1), 
    \end{align}
    an intersection that occurs inside $\scrL_{t}(\KK^{E\minus H})$. By Lemma \ref{lemma - basic intersection theory identity} and the inductive hypothesis on $ \zeta_{(\mathscr{M} \setminus H)_{< r}}^{h}$, to prove \eqref{eqn - matroidal poly jets irreducible, case 2, part 3} it suffices to prove that
    \begin{align} \label{eqn - matroidal poly jets irreducible, case 2, part 4}
        X_{\zeta_{(\mathscr{M} \setminus H)_{< r}}^{h}} \cap X_{\zeta_{\mathscr{M} / H}} \subseteq \KK^{E\minus H} \text{ has dimension at most } |E| - |H| - 2. 
    \end{align}
    But \eqref{eqn - matroidal poly jets irreducible, case 2, part 4} was one of the assumptions in \emph{Case 2}, given our choice of $h \in H$. So unraveling all the implications we see that \eqref{eqn - flag matroidal poly, jets irreducible, dim characterization in terms of min primes} holds finishing \emph{Case 2}. 
 
\medskip

    \emph{The Inductive Argument:}

    Fix a flag matroidal polynomial $\zeta_{\mathscr{M}}$ on a terminally connected and terminally strict flag matroid $\matM$ on shared ground set $E$ and assume that \eqref{eqn - dim characterization irreducible m-jets, flag matroidal poly} holds for all suitable flag matroidal polynomials on smaller ground sets. Consider a minimal prime $\mathfrak{p}$ containing the true Jacobian of $\zeta_{\mathscr{M}}$. 

    First assume that $\matM_{1}$ is not a circuit. By \cite[Lemma 2.13]{DSW}, there exists a circuit $C \neq E$ of $\matM_{1}$ such that $\rank_{\matM_{1}}(C) \geq 2$. By \cite[Prop.\ 2.8]{DSW}, we may find a proper handle $H \in \calI_{\matM_{1}}$ of $\matM_{1}$ such that $\matM_{1} \setminus H$ is connected and $\matM_{1} \setminus H \supseteq C$. Thus, $2 \leq \rank(\matM_{1} \setminus H) = \rank(\matM_{1}) - |H| + 1$ forces $\rank(\matM_{1} / H) = \rank(\matM_{1}) - |H| \geq 1$. By Proposition \ref{prop - min primes of jacobian, flag matroidal poly}, either $\mathfrak{p} \ni x_{h}$ for some $h \in H$ or $\mathfrak{p} \ni \zeta_{\mathscr{M} / H}$ (or both). If $\mathfrak{p} \ni x_{h}$, then \emph{Case 1} shows that $\mathfrak{p}$ satisfies \eqref{eqn - flag matroidal poly, jets irreducible, dim characterization in terms of min primes}. If $\mathfrak{p} \ni \zeta_{\mathscr{M} / H}$, then \emph{Case 2} applies and shows that $\mathfrak{p}$ satisfies \eqref{eqn - flag matroidal poly, jets irreducible, dim characterization in terms of min primes}. (We may invoke \emph{Case 2} since we checked that: $\rank(\matM_{1} / H) \geq 1$; $\rank(\matM_{1} \setminus H) \geq 2$; for $r$ as defined in \emph{Case 2}, the flag matroid $(\mathscr{M} \setminus H)_{< r}$ is terminally strict (Remark \ref{rmk - basics about flag matroidal polys}.\ref{item-strict-contr/del}, \cite[Lemma 2.4.(c)]{DSW}); the flag matroid $(\mathscr{M} \setminus H)_{< r}$ is terminally connected by choice of $H$; by Lemma \ref{lemma - flag matroidal poly handle fml, intersecting induced polys}, there exists a $h \in H$ such that $X_{\zeta_{(\mathscr{M} \setminus H)_{<r}}^{h}} \cap X_{\zeta_{\mathscr{M} / H}} \subseteq \KK^{E\minus H}$ has dimension at most $|E| - |H| - 2$. Thus, with the particular $h \in H$ just referenced, $\zeta_{(\mathscr{M} \setminus H)_{<r}}^{h}$ falls into the inductive set-up of \emph{Case 2}.) Since $\mathfrak{p}$ satisfies \eqref{eqn - flag matroidal poly, jets irreducible, dim characterization in terms of min primes} we conclude that \eqref{eqn - dim characterization irreducible m-jets, flag matroidal poly}.

    Now assume that $\matM_{1}$ is a circuit. If the length of the flag is one then apply Theorem \ref{thm - matroid poly, connected implies m-jets irreducible} to conclude \eqref{eqn - dim characterization irreducible m-jets, flag matroidal poly}. If $\mathscr{M}$ has length at least two, then terminal strictness forces $\matM_{2}$ to be the free matroid on $E$ and so by Remark \ref{rmk-mat-quots}.\ref{item-mat-quot-b}, $\matM_{i}$ is the free matroid on $E$ for all $i \geq 2$. Since each summand $\zeta_{\matM_{i}}$ of $\zeta_{\mathscr{M}}$ is a matroid support polynomial,
    \begin{equation} \label{eqn - flag matroidal poly, jets irreducible, case of circuit}
    \zeta_{\mathscr{M}} = \left(\sum_{e \in E} c_{E\minus\{e\}} \bsx^{E\minus\{e\}} \right) + b \bsx^{E},
    \end{equation}
    where $c_{E\minus\{e\}} \in \KK^{\times}$, $b \in \KK$, and $\zeta_{\mathscr{M}} - b\bsx^{E} = \zeta_{\matM_{1}}$. If $b = 0$, then use Theorem \ref{thm - matroid poly, connected implies m-jets irreducible} to conclude \eqref{eqn - dim characterization irreducible m-jets, flag matroidal poly}. If $b \neq 0$ then
    \[
    \mathfrak{p} \ni \zeta_{\mathscr{M}} - x_{e} \partial_{x_{e}} \zeta_{\mathscr{M}} = c_{E\minus\{e\}} \bsx^{E\minus\{e\}}
    \]
    and so there exists an $e^{\prime} \in E \minus \{e\}$ such that $\mathfrak{p} \ni x_{e^{\prime}}$. So we may invoke \emph{Case 1} to confirm \eqref{eqn - flag matroidal poly, jets irreducible, dim characterization in terms of min primes} and hence \eqref{eqn - dim characterization irreducible m-jets, flag matroidal poly}.

    That completes the inductive step. The base case is the with the smallest $E$ possible admitting a connected matroid of rank $2$. This forces $|E| = 3$ and $\matM_{1} = \matU_{2,3}$, a circuit. Note that when we verified \eqref{eqn - dim characterization irreducible m-jets, flag matroidal poly} for circuits, the only case we invoked was \emph{Case 1} which only appeals to Proposition \ref{prop - flag matroids, jets equidimensional} and does not require the inductive setup. 
    \end{proof}

Deducing rationality via \cite{MustataJetsLCI} is now trivial:

\begin{corollary} \label{cor-flag-mtrdl-ratsing}
Consider a flag matroidal polynomial $\zeta_{\mathscr{M}}$ attached to a terminally connected and terminally strict flag matroid $\mathscr{M}$. If $\rank(\matM_{1}) \geq 2$, then $\zeta_{\mathscr{M}}$ has rational singularities.
\end{corollary}

\begin{proof}
Since $\rank(\matM_{1}) \geq 2$, $\zeta_{\mathscr{M}}$ is not smooth by Remark \ref{rmk - basics about flag matroidal polys}.\eqref{item-mat-poly-degree}. Now use Theorem \ref{thm - flag matroidal poly, m-jets irreducible} to argue as in Corollary \ref{cor-mtrdl-rat-sing} mutatis mutandis.
\end{proof}

We finish with a nice application of Corollary \ref{cor-flag-mtrdl-ratsing}, with the goal of giving a flavor of the large family of polynomials we now know to have rational singularities.

\begin{corollary} \label{cor - polys whose monomial support are ind sets are rational}
    Let $\matM$ be a loopless matroid on $E$ with $\rank(\matM) \geq 3$. For $1 \leq s \leq \rank(\matM) - 2$, set
    \[
    \calI_{\matM}^{\geq \rank(\matM) - s} = \{I \in \calI_{\matM} \mid \rank(\matM) - s \leq \rank(I) \leq \rank(\matM)\}.
    \]
    Then any polynomial $g \in \KK[E]$ with monomial support exactly $\calI_{\matM}^{\geq \rank(\matM) - s}$ has rational singularities. 
\end{corollary}

\begin{proof}
    Let $\mathscr{M}$ be the length $s+1$ flag matroid given by $s$ repeated truncations, as in Example \ref{ex - repeated trunctations, monomial support on independent sets}. Then $\mathscr{M}$ is terminally strict (since $s \leq \rank(\matM) - 1)$ and terminally connected (since $\rank(\matM) \geq 2$ and $\matM$ is loopless, \emph{cf.}\ Example \ref{ex - truncations are quotients}). The rank of the terminal matroid of $\mathscr{M}$ is at least two since $s \leq \rank(\matM) - 2$. Example \ref{ex - repeated trunctations, monomial support on independent sets} shows that $g \in \FMP(\mathscr{M})$. Now appeal to Corollary \ref{cor-flag-mtrdl-ratsing}.
\end{proof}

%%%%%%%%%%%%%%%%%%%%%%%%%%%%%%%%%%%%%%%%%%%%%%%%%%%%%%%%%%
\section{Feynman Diagrams and Feynman Integrands}
%%%%%%%%%%%%%%%%%%%%%%%%%%%%%%%%%%%%%%%%%%%%%%%%%%%%%%%%%%
\label{sec-Feynman}

This section is motivated by Quantum Field Theory,  specifically the study of the Lee--Pomeransky form of Feynman integrals attached to Feynman diagrams. We postpone, for narrative purposes, an explicit definition of either the Feynman diagram or its associated polynomials until Subsection \ref{subsec-FInt}.

We will apply our techniques of handle induction and $m$-jet analysis to a more general class of polynomials than just those arising from Feynman diagrams: the \emph{Feynman integrands} (Definition \ref{def - feynman integrands}). The definition is motivated by Feynman diagrams of course, and we will recover the \emph{Feynman diagram polynomial} (Definition \ref{def - feynman diagram poly}) as a special case of a Feynman integrand (Proposition \ref{prop-FDiagPoly=FInt}).

 When casting a Feynman integral in Lee--Pomeransky form, one takes a Mellin transformation of what we call a Feynman diagram polynomial, assembled entirely out of data contained in the Feynman diagram. Theorem \ref{thm-feyn-poly-ratsing} asserts that under rather natural assumptions on the Feynman diagram the Feynman diagram polynomial has rational singularities. Certain distinctions have to be drawn in the theorem for scalar vs.\ vector-valued QFTs.

\subsection{Basics on Feynman Integrands}
%%%%%%%%%%%%%%%%%%%%%%%%%%%%%%%%%%%%%%%%%%%%%%%%%%%%%%%%%%

We follow the narrative of the previous two sections: we give a handle formula for Feynman integrands and use it to study their $m$-jets. We will eventually prove the $m$-jets are irreducible under a mild connectivity and rank assumption; hence such Feynman integrands have rational singularities. As such, the beats of this section's story echo those of its predecessors.

\begin{define} \label{def - feynman integrands}
    Let $\matN$ be a quotient of a matroid $\matM$ on shared ground set $E$ such that $\rank(\matM) - \rank(\matN) = 1$. We \emph{always} assume $\rank(\matN) \geq 1$. Let $\zeta_{\matN}$ and $\xi_{\matM}$ be matroid support polynomials on $\matN$ and $\matM$ respectively. Now consider a set $F \supseteq E$ and define $\massm = (m_{f})_{f \in F} \in \KK^{F}$ and $\Delta_{\massm}^{F} = \sum_{f \in F} m_{f} x_{f}$. Then we call the polynomial
    \[
    \Feyn = \Feyn(\zeta_{\matN}, \Delta_{\massm}^{F}, \xi_{\matM}) := \zeta_{\matN}(1 + \Delta_{\massm}^{F}) + \xi_{\matM} \in \KK[F]
    \]
    a \emph{Feynman integrand}. We call the hypersurface cut out by $\Feyn$ the Feynman hypersurface; we denote it by $X_{\Feyn} \subseteq \KK^{F}$. In physical applications, $F$ and $E$ usually coincide. We allow $\massm$ to have zero and nonzero entries.

    We also consider the \emph{momentumless Feynman integrand}
    \[
    \Feyn(\zeta_{\matN}, \Delta_{\massm}^{F}, 0) = \zeta_{\matN}(1 + \Delta_{\massm}^{F}).
    \]
    Sometimes we will want to use Deletion-Contraction. So for any $S \subseteq F$, write $\Delta_{\massm}^{F \setminus S} = \sum_{f \in F \setminus S} m_{f} x_{f}$.
\end{define}

\begin{example}
    Proposition \ref{prop-FDiagPoly=FInt} says that Feynman diagram polynomials are instances of Feynman integrands. 
\end{example}

\begin{remark} \label{rmk - feynman integrand basic facts} Suppose that $\Feyn = \Feyn(\zeta_{\matN}, \Delta_{\massm}^{F}, \zeta_\matM)$ is a Feynman integrand.
    \begin{enumerate}[label=(\alph*)]
        \item\label{item-Feyn-min} We have $\min(\Feyn) = \zeta_{\matN}$. If $\Delta_{\massm}^{F} = 0$ then $\deg \Feyn = \deg \xi_{\matM} = \deg \zeta_{\matN} + 1$. If $\Delta_{\massm}^{F} \neq 0$ and $\matN$ is loopless, then $\Feyn$ contains a monomial of degree $\deg \zeta_{\matN} + 1$ that is either not squarefree or not in $\KK[E]$. (The latter is possible when $F \supsetneq E$.) Since $\xi_{\matM} \in \KK[E]$ is squarefree, $\deg \Feyn = \deg \zeta_{\matN} + 1$ in this case as well.
        \item\label{item-Feyn-irred} If $\matN$ is connected, then $\Feyn$ is irreducible. Indeed, let $ab = \Feyn$ be a factorization. Then $\min(a) \min(b) = \zeta_{\matN}$ is irreducible, and we may assume that $\min(a) = \zeta_{\matN}$ and $\min(b) = 1$. If $\deg a = \deg \Feyn = \deg \zeta_{\matN} + 1$ then $b = 1$ and the factorization is trivial. If $\deg a = \deg \zeta_{\matN}$, then $a = \zeta_{\matN}$. This implies $\zeta_{\matN} \mid \xi_{\matM}$ and so we have a factorization $\zeta_{\matN} c = \xi_{\matM}$. This induces a partition $A \sqcup C \subseteq E$ where $\zeta_{\matN} \in \KK[A]$ and $c \in \KK[C]$ (Remark \ref{rmk - basic matroid polynomial observations}.\eqref{item-partition}). But every element of $E$ lies in a basis of $\matN$, so $A = E$. This forces $c \in \KK^{\times}$ and $\zeta_{\matN} = \xi_{\matM}$ up to a constant. But this is impossible: $\deg \zeta_{\matN} = \rank(\matN) < \rank(\matM) = \deg \xi_{\matM}$.
        \item\label{item-Feyn-nonempty} If $\rank(\matN) \geq 2$, then $\Feyn \in  (\{x_{e}\}_{e \in E})^{2} \cdot \KK[F]$ is singular at the origin. If $\rank(\matN) \geq 1$, then $\zeta_{\matN} \neq 0$ and $\Feyn$ is not a scalar. 
        \item\label{item-Feyn-DC} $\Feyn$ satisfies Deletion-Contraction identities as induced by those of $\zeta_{\matN}$, $\Delta_{\massm}^{F}$, and $\xi_{\matM}$ since $\Delta_{\massm}^{F}$ is a matroid support polynomial on $U_{1, |\widetilde{F}|}$ where $\widetilde{F} =\{f\in F\mid m_f\neq 0\}$. The identities are cumbersome to write down so we record the simplification we require.

        Suppose that $e \in E$ is neither a loop nor coloop on $\matN$. If it is also neither a loop nor coloop on $\matM$, then there is a $g \in \KK[F]$ such that
        \[
        \Feyn(\zeta_{\matN}, \Delta_{\massm}^{F}, \xi_{\matM}) = \Feyn(\zeta_{N \setminus e}, \Delta_{\massm}^{F \minus \{e\}}, \xi_{\matM\setminus e}) + x_{e} g.
        \]
        If $e$ is a coloop on $\matM$ then there is a $g \in \KK[F]$ such that
        \[
        \Feyn(\zeta_{\matN}, \Delta_{\massm}^{F}, \xi_{\matM}) = \Feyn(\zeta_{N \setminus e}, \Delta_{\massm}^{F \minus \{e\}}, 0) + x_{e} g.
        \]
    \end{enumerate}
\end{remark}

As in previous sections, our first result jet theoretic result is to show that $\scrL_{m}(\KK^{F}, X_{\Feyn})$ is equidimensional under mild assumptions on $\matN$. The proof idea is also the same as before: use Deletion-Contraction to reduce studying the dimension of $\scrL_{m}(\mathbf{K}^{F}, X_{\Feyn})$ to studying the $m$-jets of a Feynman integrand on a smaller ground set. 

\begin{proposition} \label{prop - feynman integrand, jet scheme intersect with gamma basic fml}
    Let $\Feyn = \Feyn(\zeta_{\matN}, \Delta_{\massm}^{F}, \xi_{\matM})$ be a Feynman integrand. Select $e \in E$. Let $\Gamma \subseteq \scrL_{m}(\KK^F)$ be the variety cut out by $\{D^{q}x_{e}\}_{0 \leq q \leq m}$. If $e \in E$ is not a loop/coloop on $\matM$ nor $\matN$ then
    \[
    \scrL_{m}(\KK^{F}, X_{\Feyn}) \cap \Gamma \simeq \scrL_{m}(\KK^{F \minus \{e\}}, \Var( \Feyn(\zeta_{\matN \setminus e}, \Delta_{\massm}^{F \minus \{e\}}, \xi_{\matM\setminus e}) )).
    \]
    If $e$ is not a loop/coloop on $\matN$ but is a coloop on $\matM$, then
    \[
    \scrL_{m}(\KK^{F}, X_{\Feyn}) \cap \Gamma \simeq \scrL_{m}(\KK^{F \minus \{e\}}, \Var( \Feyn(\zeta_{\matN \minus \{e\}}, \Delta_{\massm}^{F \minus \{e\}}, 0) )).
    \]
\end{proposition}

\begin{proof}
    Use Remark \ref{rmk - feynman integrand basic facts}.\ref{item-Feyn-DC} to argue as in Proposition \ref{prop - matroidal poly, m-jets intersect with Gamma}.
\end{proof}

Since a non-coloop $e \in \calI_{\matN}$ can be a coloop on $\matM$, we have to prove equidimensionality of $m$-jets in two steps: first the momentumless Feynman integrands; second the Feynman integrands themselves.

\begin{proposition} \label{prop - easiest Feynman form, equidimensional jets}
    Consider the momentumless Feynman integrand $\Feyn = \Feyn(\zeta_{\matN}, \Delta_{\massm}^{F}, 0)$. Then $\scrL_{m}(\KK^{F}, X_{\Feyn})$ is an equidimensional l.c.i. of dimension $(m+1)(|F| - 1)$) for all $m \geq 1$.
\end{proposition}

\begin{proof}
    By Proposition \ref{prop - Mustata's characterization of irreducibility of m-jets in terms of dim lying over sing} it suffices to verify
    \begin{equation} \label{eqn - easiest Feynman form, equidimensional jets dim criterion}
        \dim \scrL_{m}(\KK^{F}, X_{\Feyn}) \leq (m+1)(|F| - 1) \text{ for all } m \geq 1.
    \end{equation}
    We do this by induction. Suppose that $e \in E$ is neither a loop nor coloop on $\matN$. Let $\Gamma$ be the variety with defining ideal $(\{D^{q}x_{e}\}_{0 \leq q \leq m})$. Then by Lemma \ref{lemma - basic intersection theory identity} and Proposition \ref{prop - feynman integrand, jet scheme intersect with gamma basic fml},
    \begin{align} \label{eqn - easiest Feynman form, equidimensional jets part 1}
        \dim \scrL_{m}(\KK^{F}, X_{\Feyn}) 
        &\leq \dim \left( \scrL_{m}(\KK^{F}, X_{\Feyn}) \cap \Gamma) \right) + \codim \Gamma \\
        &= \dim \scrL_{m}(\KK^{F \minus \{e\}}, \Var( \Feyn(\zeta_{N \setminus e}, \Delta_{\massm}^{F \minus \{e\}}, 0)) ) + \codim \Gamma. \nonumber
    \end{align}
    (The usage of Lemma \ref{lemma - basic intersection theory identity} is legal since $\rank(\matN\setminus e) = \rank(\matN) \geq 1$ by definition.) So if \eqref{eqn - easiest Feynman form, equidimensional jets dim criterion} holds for the $m$-jets attached to $\Feyn(\zeta_{N \setminus e}, \Delta_{\massm}^{F \setminus e}, 0)$, then \eqref{eqn - easiest Feynman form, equidimensional jets part 1} shows that\eqref{eqn - easiest Feynman form, equidimensional jets dim criterion} holds for $X_{\Feyn}$ as well. 

    Proceeding inductively, it enough to prove \eqref{eqn - easiest Feynman form, equidimensional jets dim criterion} holds when $\matN$ consists of only loops and coloops. In that case,
    \[
    \Feyn = c \bsx^{E^{\prime}} (1 + \sum_{f \in F} m_{f} x_{f})
    \]
    where $c \in \KK^{\times}$ and $E^{\prime} \subseteq E$ consists of the coloops of $\matN$; remark that $E^{\prime} \neq \emptyset$ since $\rank(\matN) \geq 1$. Thus, $\Feyn$ is a simple normal crossing divisor and \eqref{eqn - easiest Feynman form, equidimensional jets dim criterion} follows by a local computation using Remark \ref{rmk - jets of arrangements}.
\end{proof}

\begin{proposition} \label{prop - feynman integrand, equidimensional jets}
    Consider the Feynman integrand $\Feyn = \Feyn(\zeta_{\matN}, \Delta_{\textbf{m}}^{F}, \xi_{\matM})$. Assume that $\matN$ is loopless. Then $\scrL_{m}(\KK^{F}, X_{\Feyn})$ is an equidimensional l.c.i. of dimension $(m+1)(|F| - 1)$ for all $m \geq 1$.
\end{proposition}

\begin{proof}
    Again it suffices by Proposition \ref{prop - Mustata's characterization of irreducibility of m-jets in terms of dim lying over sing} to check that
    \begin{equation} \label{eqn - feynman integrand, equidimensional jets dimension criterion}
        \dim \scrL_{m}(\KK^{F}, X_{\Feyn}) \leq (m+1)(|F| - 1) \text{ for all } m \geq 1
    \end{equation}
    Again we proceed by induction. As before, pick $e \in E$ and let $\Gamma$ be cut out by $\{D^{q}(x_{e})\}_{0 \leq q \leq m}$. 

\vspace{2ex}

    \emph{Case 1}: $e$ is not a (co)loop on $\matN$ nor on $\matM$.

\vspace{1ex}

    \noindent By  Lemma \ref{lemma - basic intersection theory identity} and Proposition \ref{prop - feynman integrand, jet scheme intersect with gamma basic fml},
    \begin{align} \label{eqn - feynman integrand, equidimensional jets, part 1}
        \dim \scrL_{m}(\KK^{F}, X_{\Feyn}) 
        &\leq \dim \left( \scrL_{m}(\KK^{F}, X_{\Feyn}) \cap \Gamma \right) + \codim \Gamma \\
        &= \dim \scrL_{m}(\KK^{F \minus \{e\}}, \Var( \Feyn(\zeta_{\matN \setminus e}, \Delta_{\massm}^{F \minus \{e\}}, \xi_{\matM\setminus e}))) + \codim \Gamma; \nonumber
    \end{align}
    since $\rank(\matN \setminus e) = \rank(\matN) \geq 1$, the usage of Lemma \ref{lemma - basic intersection theory identity} is valid. So if \eqref{eqn - feynman integrand, equidimensional jets dimension criterion} holds for the $m$-jets attached to $\Feyn(\zeta_{\matN \setminus e}, \Delta_{\massm}^{F \minus \{e\}}, \xi_{\matM\setminus e})$ then \eqref{eqn - feynman integrand, equidimensional jets, part 1} shows \eqref{eqn - feynman integrand, equidimensional jets dimension criterion} holds for $X_{\Feyn}$ as well. 

\vspace{2ex}

    \emph{Case 2}: $e$ is not a (co)loop on $\matN$ but is a coloop on $\matM$. 

\vspace{1ex}

    Argue as in \emph{Case 1}, but now using the coloop formula in Proposition \ref{prop - feynman integrand, jet scheme intersect with gamma basic fml}:    \begin{align*}
        \dim \scrL_{m}(\KK^{F}, X_{\Feyn}) \leq \dim \scrL_{m}(\KK^{F \minus \{e\}}, \Var( \Feyn(\zeta_{\matN \setminus e}, \Delta_{\massm}^{F \minus \{e\}}, 0))) + \codim \Gamma.
    \end{align*}
    By Proposition \ref{prop - easiest Feynman form, equidimensional jets}, we know 
    \[
    \dim \scrL_{m}(\KK^{F \minus \{e\}}, \Var( \Feyn(\zeta_{\matN \setminus e}, \Delta_{\massm}^{F \minus \{e\}}, 0))) \leq (m+1)(|F| - 2).
    \]
    So \eqref{eqn - feynman integrand, equidimensional jets dimension criterion} holds for $X_{\Feyn}$.

\vspace{2ex}

    \emph{The Inductive Argument}: 

\vspace{1ex}
    
    We prove that $\Feyn$ satisfies \eqref{eqn - feynman integrand, equidimensional jets dimension criterion} by induction on $|E|$. Assume that the claim holds on smaller ground sets. If $\matN$ admits a non-(co)loop $e$ then either: $\rank(\matM\setminus e) = \rank(\matN \setminus e) + 1$ happens and \emph{Case 1} and the inductive hypothesis apply, or $\rank(\matM\setminus e) = \rank(\matN \setminus e)$ happens and \emph{Case 2} applies since looplessness persists under deletions. If $\matN$ admits no non-(co)loops then since $\matN$ is loopless, $\matN$ must be a free matroid on $E$. But then $\rank(\matM) = \rank(\matN) + 1 = |E| + 1$ is impossible. 

    The base case of the induction is the smallest $E$ admitting a matroid quotient $\matM \twoheadrightarrow \matN$ where: $\rank(\matN) \geq 1$; $\rank(\matM) = \rank(\matN) + 1$; $\matN$ is loopless. This forces: $|E| = 2$; $\matM$ is the free matroid on $E$; $\matN$ is the circuit matroid $\matU_{1,2}$ on $E$. Here \emph{Case 2} applies and implies \eqref{eqn - feynman integrand, equidimensional jets dimension criterion}. As \emph{Case 2} does not appeal to any inductive hypothesis, the induction terminates. 
\end{proof}

\subsection{A Special Handle}
%%%%%%%%%%%%%%%%%%%%%%%%%%%%%%%%%%%%%%%%%%%%%%%%%%%%%%%%%%

The inductive tools necessary to study the rationality of Feynman integrals involve, as before, a handle style induction that exploits connective handles $H$ on $\matN$. Our situation is similar to that of flag matroidal polynomials: $(\matM, \matN)$ is a terminally strict flag matroid after all. 

In the flag matroidal polynomial setting, if $H$ contained a coloop on $\matM$, the induction reduced studying the flag $(\matM,\matN)$ to the matroid $\matN\setminus H$, and for connected $\matN\setminus H$, any attached matroidal polynomial was smooth or rational thanks to Corollary \ref{cor-mtrdl-rat-sing}. 

In the present case, if $H$ has a coloop on $\matM$ the analogous approach would be to reduce studying the Feynman integrand $\Feyn(\zeta_{\matN}, \Delta_{\massm}^{E}, \xi_{\matM})$ attached to $\matN$ and $\matM$ to the momentum-free Feynman integrand $\Feyn(\zeta_{\matN \setminus H}, \Delta_{\massm}^{E\minus H}, 0)$. This is problematic since  momentum-free Feynman integrands often \emph{do not} have rational singularities. For instance, on the circuit $\matN=\matU_{2,3}$ the momentum-free Feynman integrand $\Psi_{\matN}\cdot(1 + x_{1} + x_{2} + x_{3})$ is not even irreducible at the point $(x_{1} = -1, x_{2} = 0 = x_{3})$. 

Thus, not only will our connective handle $H$ on $\matN$ need to satisfy all our previous requirements, it will also need to avoid all coloops on $\matM$. In other words, we will require $\rank(\matM \setminus H) > \rank(\matN \setminus H)$. The following result guarantees that we can find such a special handle whenever necessary.

\begin{proposition}  \label{prop - special handle wrt matroid quotients}
    Let $\matN$ be a quotient of $\matM$ on common ground set $E$ where $\rank(\matM) > \rank(\matN) \geq 2$. If $\matN$ is connected and not a circuit, then there exists a proper handle $H \in \calI_{\matN}$ with the following properties:
    \begin{enumerate}[label=(\alph*)]
        \item\label{item-connected} $\matN\setminus H$ is connected;
        \item\label{item-rank-stuff} $\rank(\matM \setminus H) > \rank(\matN \setminus H) \geq 2$
        \item\label{item-more-rank} $\rank(\matN/ H) \geq 1$.
    \end{enumerate}
\end{proposition}

\begin{proof}
    First note that for any proper handle $H$ of $\matN$, the inequality $2 \leq \rank(\matN \setminus H) = \rank(\matN) - |H| + 1$ is equivalent to $1 \leq \rank(\matN/ H) = \rank(\matN) - |H|$, \emph{i.e.}\ to Condition \ref{item-more-rank}. Since $\rank(\matN) \geq 2$, by \cite[Lemma 2.13]{DSW} we can find a circuit $C$ of $\matN$ with $\rank_{\matN}(C) \geq 2$. Since $\matN$ is not a circuit, $C \neq E$. By \cite[Prop.~2.8]{DSW} we may find a proper handle $H \in \calI_{\matN}$ such that $\matN\setminus H$ is connected and $\matN\setminus H \supseteq C$. If $\rank(\matM \setminus H) > \rank(\matN \setminus H)$ we are done since $\rank(\matN \setminus H) \geq \rank_{\matN}(C) \geq 2$.

    So assume $\rank(\matM \setminus H) = \rank(\matN \setminus H)$ and pick $h \in H$. Since $H$ is a handle on both $\matM$ and $\matN$, by \cite[Lemma 2.4.(c)]{DSW}, $H\minus\{h\}$ consists of only coloops on both $\matM \setminus h$ and $\matN\setminus h$. So
    \begin{align*}
        \rank(\matM \setminus h) - |H| + 1 
        &= \rank(\matM \setminus H) \\
        &= \rank(\matN \setminus H) \\
        &=\rank(\matN \setminus h) - |H| + 1 \\
        &= \rank(\matN) - |H| + 1.
    \end{align*}
    So $\rank(\matM \setminus H) = \rank(\matN \setminus H)$ if and only if $\rank(\matM) = \rank(\matN) + 1$ and $h$ is a coloop on $\matM$. By the handle property, this means that $H$ consists of only coloops on $\matM$. As $\matM \setminus H = \matN \setminus H$ by Remark \ref{rmk-mat-quots}.\eqref{item-mat-quot-b}, certainly $\matM \setminus H$ is connected. So every element of $\matM \setminus H$ lies in a circuit of $\matM \setminus H$ and hence every element of $\matM \setminus H$ lies in a circuit of $\matM$. In summary:
    \[
    H \subseteq \{\text{coloops of $\matM$}\} \subseteq H.
    \]

    Since $H \in \calI_{\matN}$ is a handle on $\matN$ and $\matN$ is connected, the set 
    \[
    S = \{C^{\prime} \mid C^{\prime} \text{ is a circuit of $\matN$}, C^{\prime} \supseteq H \}
    \]
    is non-empty. In fact, there must be a $C^{\prime} \in S$ such that $\rank_{\matN}(C^{\prime}) \geq 2$. (Otherwise, $H$ should then be a singleton $\{e\}$ which, by the connectedness of $\matN$, would yield
    $
    S = \{ \{e, g\} \mid g \in E \setminus e \}
    $ and so $\{e\}\in\calB_\matN$, in contradiction to our hypotheses).
    
    So we may find a circuit $C^{\prime}$ of $\matN$ such that $C^{\prime} \supseteq H$ and $\rank_{\matN}(C^{\prime}) \geq 2$. Since $\matN$ is not a circuit, $\matN\neq C^{\prime}$. By \cite[Prop.~2.8]{DSW}, we may find a proper handle $H^{\prime} \in \calI_{\matN}$ of $\matN$ such that: $\matN\setminus H^{\prime}$ is connected; $\matN\setminus H^{\prime} \supseteq C^{\prime} \supseteq H = \{\text{coloops of $\matM$} \}$. Since $H^{\prime}$ is disjoint from the coloops of $\matM$, we have already seen that $\rank(\matM \setminus H^{\prime}) > \rank(\matN \setminus H^{\prime})$. By the choice of $C^{\prime}$, we have $\rank(\matN \setminus H^{\prime}) \geq \rank_{\matN}(C^{\prime}) \geq 2$. Thus $H^{\prime}$ satisfies \ref{item-connected} and \ref{item-rank-stuff}. That $H^{\prime}$ satisfies \ref{item-more-rank} as well comes from the proof's first sentence.
\end{proof}

\subsection{Rationality of Feynman Integrands}
%%%%%%%%%%%%%%%%%%%%%%%%%%%%%%%%%%%%%%%%%%%%%%%%%%%%%%%%%%

We return to the now familiar process of building up our handle style induction: we give a handle formula for $\Feyn$, use it to find special hypersurfaces containing the true Jacobian of $\Feyn$ and use them to show that the $m$-jets of $\Feyn$ are irreducible under mild assumptions. We begin with the handle formula: 

\begin{proposition} \label{prop - new handle fml, Feynman}
    Let $\Feyn = \Feyn(\zeta_{\matN}, \Delta_{\massm}^{F}, \xi_{\matM})$ be a Feynman integrand. Let $H \in \calI_{\matN}$ be a proper handle such that $\rank(\matN/ H) \geq 1$ and $H$ contains no coloops on $\matN$ nor on$\matM$. Then there exist Feynman integrands 
    \[
    \Feyn_{/ H} := \Feyn(\zeta_{\matN/ H}, \Delta_{\massm}^{F}, \zeta_{\matM / H}) \quad \text{and} \quad \Feyn_{\setminus H}^{h}:= \Feyn(\zeta_{\matN\setminus H}^{h}, \Delta_{\massm}^{F \minus\{h\}}, \zeta_{\matM \setminus H}^{h})
    \]
    where $h \in H$ and $\zeta_{\matN/ H}, \zeta_{\matN\setminus H}^{h}$ (resp. $\xi_{\matM / H}, \xi_{\matM \setminus H}^{h}$) are the handle matroidal polynomials of $\zeta_{\matN}$ (resp. $\xi_{\matM}$) of Definition \ref{def - handle matroidal polynomials} such that
    \begin{equation} \label{eqn - new handle fml, Feynman}
    \Feyn = \left( \sum_{h \in H} \bsx^{H\minus\{h\}} \Feyn_{\setminus H}^{h} \right) + \bsx^{H} \left( \Feyn_{/H} + \sum_{h \in H} m_{h} \zeta_{\matN\setminus H}^{h} \right).
    \end{equation}
\end{proposition}

\begin{proof}
    Applying the handle formula Proposition \ref{prop - new handle fml - matroidal poly} to both $\zeta_{\matN}$ and $\zeta_\matM$ gives \eqref{eqn - new handle fml, Feynman}:
    \begin{align} \label{eqn - new handle fml, Feynman, pf, 1}
        \Feyn 
        &= \bigg( \big( \sum_{h \in H} \bsx^{H\minus\{h\}} \zeta_{\matN\setminus H}^{h} \big) + \bsx^{H} \zeta_{\matN/ H} \bigg) (1 + \Delta_{\massm}^{F}) \\
        &+ \bigg( \big(\sum_{h \in H} \bsx^{H\minus\{h\}} \xi_{\matM \setminus H}^{h} \big) + \bsx^{H} \xi_{\matM / H} \bigg) \nonumber \\
        &= \sum_{h \in H} \bsx^{H\minus\{h\}} \bigg( \zeta_{\matN\setminus H}^{h} (1 + \Delta_{\massm}^{F\minus\{h\}}) + \xi_{\matM \setminus H}^{h} \bigg) \nonumber \\
        &+ \big( \bsx^{H} \sum_{h \in H} m_{h} \zeta_{\matN\setminus H}^{h} \big) +  \bsx^{H} \bigg( \zeta_{\matN/ H} (1 + \Delta_{\massm}^{F}) + \xi_{\matM / H} \bigg)\nonumber \\
        &= \bigg( \sum_{h \in H} \bsx^{H\minus\{h\}} \Feyn_{\setminus H}^{h} \bigg) + \bsx^{H} \bigg( \Feyn_{/H} + \sum_{h \in H} m_{h} \zeta_{\matN\setminus H}^{h} \bigg) \nonumber.
    \end{align}
\end{proof}

\begin{define} \label{def - handle feynman integrands}
    In the setting of Proposition \ref{prop - new handle fml, Feynman}, we call the polynomials
    \[
    \Feyn_{/ H} := \Feyn(\zeta_{\matN/ H}, \Delta_{\massm}^{F}, \zeta_{\matM / H}) \quad \text{and} \quad \Feyn_{\setminus H}^{h}:= \Feyn(\zeta_{\matN\setminus H}^{h}, \Delta_{\massm}^{F\minus\{h\}}, \zeta_{\matM \setminus H}^{h})
    \]
    from \eqref{eqn - new handle fml, Feynman} the \emph{handle Feynman integrands}.  Note that $\Feyn_{/ H} \in \KK[F]$ and $\Feyn_{\setminus H}^{h} \in \KK[F\minus\{h\}]$. 
\end{define}

The first part of setting up the handle style induction is to study the minimal primes of the true Jacobian of a Feynman integrand, in s manner prallel to the previous cases. Here, the presence of $\Delta_{\massm}^{F}$ is an inconvenience in terms of argumentation, but surmountable in terms of results.

\begin{proposition} \label{prop - feynman integrand, minimal primes of true jacobian}
    Let $\Feyn = \Feyn(\zeta_{\matN}, \Delta_{\massm}^{F}, \xi_{\matM})$ be a Feynman integrand. Any prime ideal $\mathfrak{p}$ containing the true Jacobian $(\partial \Feyn) + (\Feyn)$ contains both  $\zeta_{\matN}, \xi_{\matM}$. 

    Moreover, if $H \in \calI_{\matN}$ is a proper handle on $\matN$ such that $\rank(\matN/ H) \geq 1$ and $H$ contains no coloops on $\matN$ nor $\matM$, then one of the following occurs:
    \begin{enumerate}[label=(\alph*)]
        \item\label{item-p-h} $\mathfrak{p} \ni x_{h}$ for some $h \in H$;
        \item\label{item-p-H} $\mathfrak{p}$ contains the handle Feynman integrand  $\Feyn_{/H}$.
    \end{enumerate}
\end{proposition}

\begin{proof}
    Let $\eps := \sum_{e \in E} x_{e} \partial_{x_{e}}$ be the Euler derivation. Since $\min(\Feyn) = \zeta_{\matN}$ is homogeneous of degree $\rank(\matN)$ and $\Feyn - \zeta_{\matN}$ is homogeneous of degree $\rank(\matN) + 1$ (or is $0$),  we observe that (the bullet signifies application of a vector field to a function)
    \begin{align*}
    \mathfrak{p} \ni (\rank(\matN) + 1)\Feyn - \eps \bullet \Feyn 
    &= (\rank(\matN) + 1) \Feyn - \eps \bullet (\Feyn - \zeta_{\matN}) - \eps \bullet \zeta_{\matN} \\
    &= \zeta_{\matN}.
    \end{align*}
    Since $\mathfrak{p} \ni \Feyn = \zeta_{\matN}(1 + \Delta_{\massm}^{E}) + \xi_{\matM}$ we deduce that $\mathfrak{p} \ni \xi_{\matM}$ as well.

    Now pick $h \in H$. By the handle formula Proposition \ref{prop - new handle fml - matroidal poly}
    \begin{align*}
        \zeta_{\matN} - x_{h} \partial_{x_{h}} \bullet \zeta_{\matN} = \bsx^{H\minus\{h\}}  \zeta_{\matN\setminus H}^{h} \quad \text{ and } \quad \xi_{\matM} - x_{h} \partial_{x_{h}} \bullet \xi_{\matM} = \bsx^{H\minus\{h\}} \xi_{\matM \setminus H}^{h}.
    \end{align*}
    Hence,
    \begin{align*}
        \mathfrak{p} 
        &\ni \Feyn - x_{h} \partial_{x_{h}} \bullet \Feyn \\
        &= (\zeta_{\matN} + \zeta_{\matN} \Delta_{\massm}^{F} + \xi_{\matM}) - x_{h} \partial_{x_{h}} \bullet (\zeta_{\matN} + \zeta_{\matN} \Delta_{\massm}^{F} + \xi_{\matM}) \\
        &= (\zeta_{\matN} - x_{h} \partial_{x_{h}} \bullet \zeta_{\matN}) + (\zeta_{\matN} - x_{h} \partial_{x_{h}} \bullet \zeta_{\matN}) \Delta_{\massm}^{F} - \zeta_{\matN} (x_{h} \partial_{x_{h}} \bullet \Delta_{\massm}^{F}) + (\xi_{\matM} - x_{h} \partial_{x_{h}} \bullet \xi_{\matM}) \\
        &= \bsx^{H\minus\{h\}} \zeta_{\matN\setminus H}^{h} + \bsx^{H\minus\{h\}} \zeta_{\matN\setminus H}^{h} \Delta_{\massm}^{F} - m_{h} x_{h} \zeta_{\matN} + \bsx^{H\minus\{h\}} \xi_{\matM \setminus H}^{h} \\
        &= \bsx^{H\minus\{h\}} \Feyn(\zeta_{\matN\setminus H}^{h}, \Delta_{\massm}^{F}, \xi_{\matM \setminus H}^{h}) - m_{h} x_{h} \zeta_{\matN}.
    \end{align*}
    Since $\mathfrak{p} \ni \zeta_{\matN}$, we find for all $h \in H$ that
    \begin{align*}
    \mathfrak{p} 
    \ni \bsx^{H\minus\{h\}} \Feyn( \zeta_{\matN\setminus H}^{h}, \Delta_{\massm}^{F}, \xi_{\matM \setminus H}^{h}) 
    =  \bsx^{H\minus\{h\}} \left( \zeta_{\matN\setminus H}^{h}  (1 + \Delta_{\massm}^{F}) +   \xi_{\matM \setminus H}^{h} \right).
    \end{align*}
    By the first ``$=$'' of \eqref{eqn - new handle fml, Feynman, pf, 1},
    \begin{equation*} \label{eqn - feynman integrand, min prime true jacobian, useful membership in coloop on M case}
    \mathfrak{p} \ni \bsx^{H} \left( \zeta_{\matN/ H} (1 + \Delta_{\massm}^{F}) + \zeta_{\matM / H} \right) = \bsx^{H} \Feyn(\zeta_{\matN/ H}, \Delta_{\massm}^{F}, \zeta_{\matM / H}) = \bsx^{H} \Feyn_{/ H}.
    \end{equation*}
    Primality of $\mathfrak{p}$ implies that either \ref{item-p-h} or \ref{item-p-H} must hold.
    \end{proof}

\begin{example} \label{example - feynman integrand, min primes jacobian, N a circuit}
    Let $\matN$ be a circuit on $E$ that is a quotient of $\matM$ where $\matN\neq \matM$, whence $\matM$ is the free matroid on $E$. Consider a Feynman integrand $\Feyn = \Feyn(\zeta_{\matN}, \Delta_{\massm}^{E}, \xi_{\matM})$ and let $\mathfrak{p}$ be a minimal prime of the true Jacobian $(\partial \Feyn) + (\Feyn)$. By Proposition \ref{prop - feynman integrand, minimal primes of true jacobian}, $\mathfrak{p} \ni \xi_{\matM}\in\KK^\times\cdot\bsx^E$. Since $\matM$ is free,  there exists an $e \in E$ such that $\mathfrak{p} \ni x_{e}$.
\end{example}

Now we turn to setting up a contrivance for falling into the inductive set-up. The first Lemma is familiar to its matroidal polynomial and flag matroidal polynomial counterparts; recall Definition \ref{def - handle feynman integrands}.

\begin{lemma} \label{lemma - feynman integrand, jet scheme intersect with gamma, lying over}
    Let $\Feyn = \Feyn(\zeta_{\matN}, \Delta_{\massm}^{F}, \xi_{\matM})$ be a Feynman integrand. Let $H \in \calI_{\matN}$ be a proper handle on $\matN$ such that $\rank(\matN/ H) \geq 1$ and $H$ contains no coloops on $\matN$ nor on $\matM$. Pick $h \in H$ and let $\Gamma \subseteq \scrL_{m}(\KK^{E})$ be the variety defined by $(\{D^{q} x_{h} \}_{0 \leq q \leq m})$. Then
    \begin{equation} \label{eqn - feynman integrand, jet scheme intersect with gamma}
        \scrL_{m}(\KK^{F}, X_{\Feyn}) \cap \Gamma \simeq \scrL_{m}(\KK^{F\minus\{h\}}, \Var(\bsx^{H\minus\{h\}} \Feyn_{\setminus H}^{h}))
    \end{equation}
    and 
    \begin{align} \label{eqn - feynman integrand, jet scheme intersect with gamma, lying over}
    \scrL_{m}  ( \KK^{F}, &X_{\Feyn}, \Var ( \Feyn_{/H}) ) \cap \Gamma \\
    &\simeq \scrL_{m} \big(\KK^{F\minus\{h\}}, \Var( \bsx^{H\minus\{h\}} \Feyn_{\setminus H}^{h}), \Var( \Feyn(\zeta_{\matN/ H}, \Delta_{\massm}^{F\minus\{h\}}, \xi_{\matM / H})) \big) \nonumber \\
    &\neq \emptyset. \nonumber
    \end{align}
\end{lemma}

\begin{proof}
    By Proposition \ref{prop - new handle fml, Feynman}, there exists  $g \in \KK[F]$ such that 
    \[
    \Feyn = \bsx^{H\minus\{h\}} \Feyn_{\setminus H}^{h} + x_{h} g.
    \]
    As the first summand lies in $\KK[F\minus\{h\}]$ we can prove \eqref{eqn - feynman integrand, jet scheme intersect with gamma} by mimicking the argument for \eqref{eqn - matroid poly, connected, m-jet handle intersect with Gamma} in Lemma \ref{lemma - matroidal poly, handle fml intersect with Gamma}. The justification for the ``$\simeq$'' of \eqref{eqn - feynman integrand, jet scheme intersect with gamma, lying over} follows similarly to the justification for \eqref{eqn - matroidal poly, connected, m-jet handle lying over intersect with Gamma} in Lemma \ref{lemma - matroidal poly, handle fml intersect with Gamma} with the additional observation that evaluating $\Feyn_{/H}$ at $x_{h}=0$ yields $\Feyn(\zeta_{\matN/ H}, \Delta_{\massm}^{F\minus\{h\}}, \xi_{\matM / H})$. 
    
    As for the ``$\neq \emptyset$'' part of \eqref{eqn - feynman integrand, jet scheme intersect with gamma, lying over}, as $\rank(\matN/ H) \geq 1$ we know that $\rank(\matN \setminus H) \geq 2$. Together these inequalities imply that $\bsx^{H\minus\{h\}}\Feyn_{\setminus H}^{h}$ and $\Feyn(\zeta_{\matN/ H}, \Delta_{\massm}^{F\minus\{h\}}, \xi_{\matM / H})$ are both nonzero and vanish at $0 \in \KK^{F\minus\{h\}}$. Then,  
    \[
    \scrL_{m} (\KK^{F\minus\{h\}}, \Var(\bsx^{H\minus\{h\}} \Feyn_{\setminus H}^{h}), \Var(\Feyn( \zeta_{\matN/ H}, \Delta_{\massm}^{F\minus\{h\}}, \xi_{\matM / H}))) \neq \emptyset
    \]
    follows by Remark \ref{rmk - always a m-jet lying over a point in X}.
\end{proof}

Lemma \ref{lemma - feynman integrand, jet scheme intersect with gamma, lying over} exhibits another obstacle $\Delta_{\massm}^{F}$ causes. In previous inductive arguments, we first passed from $\scrL_{m}(\KK^{F}, X_{\Feyn})$ to $\scrL_{m}(\KK^{F\minus\{h\}}, \Var(\bsx^{H\minus\{h\}} \Feyn_{\setminus H}^{h}))$ and then reduced to $\scrL_m(\KK^{F \minus H}, \Feyn_{\setminus H}^{h})$. Alas, here this reduction is nonsensical. First of all, $\Feyn_{\setminus H}^{h}$ may not reside in $\KK[F \minus H]$; because its definition involves $\Delta_{\massm}^{F \minus\{h\}}$ it may use $H \minus\{h\}$ variables. Secondly, it is less obvious how to extract the $\bsx^{H\minus\{h\}}$ term from $\bsx^{H\minus\{h\}}\Feyn_{\setminus H}^{h}$ since the multiplicands do not use disjoint variable sets. In previous arguments we had such disjointedness.
%In previous sections, w and then use Lemma \ref{lemma - m-jet codim for products of eqs disjoint variables} in order to reduce further to  $\scrL_{m}(\KK^{F \minus H}, \Feyn_{\setminus H}^{h})$.  But we \emph{cannot} use Lemma \ref{lemma - m-jet codim for products of eqs disjoint variables} because of $\Delta_{\massm}^{F}$: the equation $\bsx^{H\minus\{h\}} \Feyn_{\setminus H}^{h}$ is not a product of a polynomial in the $H\minus\{h\}$ variables and one in the $F \setminus H$ variables. Indeed: $\Feyn_{\setminus H}^{h}$ may use $H\minus\{h\}$ variables as its definition involves $\Delta_{\massm}^{F\minus\{h\}}$. 
So we need another way to understand $\scrL_{m}(\KK^{F}, \Var(\bsx^{H\minus\{h\}} \Feyn_{\setminus H}^{h}))$, hence the following:

\begin{lemma} \label{lemma - feynman, intersecting with nabla}
    Let $E \subseteq F$ with $H^{\prime} = F \minus E$. Consider the Feynman integrand $\Feyn(\zeta_{\matN}, \Delta_{\massm}^{F}, \xi_{\matM})$. Consider the bounded integral vectors
    \[
    \Upsilon = \{ \mathbf{v} \in \mathbb{Z}^{|H^{\prime}|} \mid -1 \leq v_{k} \leq m \text{ and } |\mathbf{v}| \leq m - |H^{\prime}| \} ,
    \]
    and for each $\mathbf{v} \in \Upsilon$ define (with the convention $D^{-1}(g) = 0$) the ideal $J_{\mathbf{v}} \subseteq \KK[\{x_{f}^{(q)} \}_{\substack{ f \in F \\ 0 \leq q \leq m}}]$ by 
    \[
    J_{\mathbf{v}} \text{ is generated by } \{D^{p} x_{h^{\prime}} \}_{\substack{h^{\prime} \in H^{\prime} \\ 0 \leq p \leq v_{h^{\prime}}}} \cup \{D^{j} \Feyn(\zeta_{\matN}, \Delta_{\massm}^{F}, \xi_{\matM}) \}_{0 \leq j \leq m - |H^{\prime}| - |\mathbf{v}|}. 
    \]
    If $\mathfrak{q}$ is a minimal prime of $I(\scrL_{m}(\KK^{F}, \Var(\bsx^{H^{\prime}} \Feyn(\zeta_{\matN}, \Delta_{\massm}^{F}, \xi_{\matM}))))$, then
    \begin{equation} \label{eqn - feynman, intersect with nabla, min primes}
    \text{ there exists $\mathbf{v} \in \Upsilon$ such that } \mathfrak{q} \supseteq J_{\mathbf{v}}.
    \end{equation}

    Now fix $\mathbf{v} \in \Upsilon$ and set $t = m - |H^{\prime}| - |\mathbf{v}|$. The ideal $\nabla_{\mathbf{v}} \subseteq \KK[\{x_{f}^{(q)} \}_{\substack{ f \in F \\ 0 \leq q \leq m}}]$ defined by $\nabla_{\mathbf{v}} = (\{ D^{\ell} x_{h^{\prime}} \}_{\substack{h^{\prime} \in H^{\prime} \\ v_{h^{\prime}} + 1 \leq \ell \leq m}}) $ is a complete intersection of codimension $(m+1)|H^{\prime}| - (m-t)$. Moreover, we have a natural isomorphism
    \begin{equation} \label{eqn - feynman, intersect with nabla, embedding}
        \Var(J_{\mathbf{v}}) \cap \Var(\nabla_{\mathbf{v}}) \simeq (\pi_{m, t}^{\KK^{E}})^{-1} \bigg( \scrL_{t}(\KK^{E}, \Var(\Feyn(\zeta_{\matN}, \Delta_{\massm}^{E}, \xi_{\matM}))) \bigg) \subseteq \scrL_{m}(\KK^{E}).
    \end{equation}
    
    \end{lemma}

    \begin{proof}
    The proof of \eqref{eqn - feynman, intersect with nabla, min primes} is a repeated application of Lemma \ref{lemma - jet schemes of products, not nec disjoint variables}. As for the second paragraph, $\nabla_{\mathbf{v}}$ clearly is a complete intersection, easily seen to be of the specified codimension; it remains to check \eqref{eqn - feynman, intersect with nabla, embedding}. By construction,
    \begin{equation} \label{eqn - feynman, intersect with nabla, nabla plus lambda}
        J_{\mathbf{v}} + \nabla_{\mathbf{v}} = \bigg( \{x_{h^{\prime}}^{(q)}\}_{\substack{h^{\prime} \in H^{\prime} \\ 0 \leq q \leq m}} \bigg) + \bigg( \{D^{j} \Feyn(\zeta_\matN, \Delta_{\massm}^{F}, \zeta_\matM) \}_{0 \leq j \leq t} \bigg).
    \end{equation}
    When $t=-1$ the truth of \eqref{eqn - feynman, intersect with nabla, embedding} is trivial:
    \[
    \Var(J_{\mathbf{v}}) \cap \Var(\nabla_{\mathbf{v}}) = \Var( \{ x_{h^\prime}^{(q)} \}_{\substack{h^\prime \in H^{\prime} \\ 0 \leq q \leq m}} ) \simeq \scrL_{m}(\KK^{E}). 
    \]

    So we may assume $t \geq 0$. Let $0 \leq j \leq t \leq m$. We extract the monomials of $D^{j} \Feyn(\xi_\matN, \Delta_{\massm}^{F}, \xi_{\matM})$ containing members of $\{x_{f}\}_{f \in F \minus E}$. This (note the change in superscript of $\Delta_{\massm}^{(-)}$) exchanges $D^{j} \Feyn(\zeta_\matN, \Delta_{\massm}^{F}, \xi_\matM)$ with $D^{j} \Feyn(\zeta_\matN, \Delta_{\massm}^{E}, \xi_\matM)$, modulo a controllable error term. To be exact,
    \begin{align} \label{eqn - feynman, intersect with nabla, computation}
        D^{j} &\Feyn(\zeta_{\matN}, \Delta_{\massm}^{F}, \xi_{\matM}) \\
        &= D^{j} \zeta_{\matN} + D^{j}(\zeta_{\matN} \Delta_{\massm}^{F}) + D^{j} \xi_{\matM} \nonumber \\
        &= D^{j} \zeta_{\matN} + \left( \sum_{0 \leq a \leq j} \binom{m}{a}(D^{a} \zeta_{\matN}) (D^{m-a} \Delta_{\massm}^{F}) \right)+ D^{j} \xi_{\matM} \nonumber \\
        &= D^{j} \zeta_{\matN} + \left( \sum_{0 \leq a \leq j} \binom{m}{a} (D^{a} \zeta_{\matN}) \left[ (D^{m-a} \Delta_{\massm}^{E}) + \sum_{h^{\prime} \in H^{\prime}} m_{h^{\prime}} D^{m-a}x_{h^{\prime}} \right] \right)+ D^{j} \xi_{\matM} \nonumber \\
        &= D^{j} \Feyn(\zeta_{\matN}, \Delta_{\massm}^{E}, \xi_{\matM}) + g \nonumber
    \end{align}
    where $g$ is in the ideal of $\KK[\{x_{f}^{(q)}\}_{\substack{ f \in F \\ 0 \leq q \leq m}}]$ generated by $\{D^{\ell} x_{h^{\prime}} \}_{\substack{h^{\prime} \in H^{\prime} \\ 0 \leq \ell \leq m}}.$ Combining \eqref{eqn - feynman, intersect with nabla, nabla plus lambda} and \eqref{eqn - feynman, intersect with nabla, computation} we deduce
    \[
    J_{\mathbf{v}} + \nabla_{\mathbf{v}} = \bigg( \{x_{h^{\prime}}^{(q)}\}_{\substack{h^{\prime} \in H^{\prime} \\ 0 \leq q \leq m}} \bigg) + \bigg( \{ D^{j} \Feyn(\zeta_{\matN}, \Delta_{\massm}^{E}, \xi_{\matM}) \}_{0 \leq j \leq t} \bigg),
    \]
    and hence
    \[
    \Var(J_{\mathbf{v}}) \cap \Var(\nabla_{\mathbf{v}}) \simeq (\pi_{m,t}^{\KK^{E}})^{-1} \bigg( \scrL_{t}(\KK^{E}, \Var( \Feyn(\zeta_{\matN}, \Delta_{\massm}^{E}, \xi_{\matN}))) \bigg).
    \]
    \end{proof}

Next we give the main result of this section: when $\matN$ is connected and $\rank(\matN) \geq 2$ then  the $m$-jets of $\Feyn(\zeta_{\matN}, \Delta_{\massm}^{E}, \xi_{\matM})$ are irreducible. Note that this is a polynomial in $\KK[E]$; while our result doesn't speak to Feynman integrands where $F \supsetneq E$, the proof requires dealing with such cases. There are two key differences compared to the corresponding Theorems \ref{thm - matroid poly, connected implies m-jets irreducible}, \ref{thm - flag matroidal poly, m-jets irreducible} caused by $\Delta_{\massm}^{E}$. First we will have to employ our special handle guaranteed by Proposition \ref{prop - special handle wrt matroid quotients}; second, we will have to carefully combine Lemma \ref{lemma - feynman integrand, jet scheme intersect with gamma, lying over} and Lemma \ref{lemma - feynman, intersecting with nabla} to reduce to an inductive hypothesis.

\begin{theorem} \label{thm - feynman integrand, irreducible jets}
    Let $\Feyn = \Feyn(\zeta_{\matN}, \Delta_{\massm}^{E}, \xi_{\matM}) \in \KK[E]$ be a Feynman integrand. Suppose that $\matN$ is connected of positive rank. Then $\scrL_{m}(\KK^{E}, X_{\Feyn})$ is an irreducible, reduced, l.c.i. of dimension $(m+1)(n-1)$ for all $m \geq 1$. Moreover, $X_{\Feyn}$ is normal.
\end{theorem}

\begin{proof}
    When $\rank(\matN) = 1$ and $\Feyn$ is smooth, this is immediate; when $\rank(\matN) = 1$ and $\Feyn$ is singular, use the the second paragraph of Corollary \ref{cor-Feyn-ratsing}'s proof and Musta\c t\u a's Theorem \ref{thm-mustata-main}. So we may assume that $\rank(\matN) \geq 2$. 
    
    By Proposition \ref{prop - Mustata's characterization of irreducibility of m-jets in terms of dim lying over sing}, it is enough to demonstrate the dimension bound
    \begin{equation} \label{eqn - feynman, irreducible jets, dimension criterion}
        \dim \scrL_{m}(\KK^E, X_{\Feyn}, X_{\Feyn, \Sing}) < (m+1)(n-1) \text{ for all } m \geq 1.
    \end{equation}
    We induce on $|E|$, via a handle induction on $\matN$, and resolve cases before giving the whole induction argument. 

    Let $\mathfrak{p}$ be a minimal prime of of the true Jacobian $(\partial \Feyn) + (\Feyn)$. To show \eqref{eqn - feynman, irreducible jets, dimension criterion} we must show
    \begin{equation} \label{eqn - feynman, irreducible jets, dimension bound for general min prime}
        \dim \scrL_{m}(\KK^E, X_{\Feyn}, \Var(\mathfrak{p})) < (m+1)(n-1) \text{ for all } m \geq 1.
    \end{equation}
    We first establish two cases.

    \vspace{2ex}

    \emph{Case 1}: If there exists an $e \in E$ such that $x_{e} \in \Var(\mathfrak{p})$, then \eqref{eqn - feynman, irreducible jets, dimension bound for general min prime} holds.

\vspace{1ex}

    \noindent\emph{Proof}: Let $\Gamma \subseteq \scrL_{m}(\KK^{E})$ be the codimension $m$ variety cut out by $(\{D^{q} x_{e} \}_{1 \leq q \leq m})$. Then
    \begin{align} \label{eqn - feynman, irreducible jets, case 1}
        \dim \scrL_{m}(\KK^E, X_{\Feyn}, \Var(\mathfrak{p})) 
        &\leq \dim \scrL_{m}(\KK^E, X_{\Feyn}, \Var(x_{e})) \\
        &\leq \dim \left(\scrL_{m}(\KK^E, X_{\Feyn}, \Var(x_{e})) \cap \Gamma \right) + \codim \Gamma \nonumber \\
        &= \dim \scrL_m(\KK^{E \minus \{e\}}, Y) + m, \nonumber
    \end{align}
    where $Y$ is one of two hypersurfaces. If $e$ is not a coloop on $\matM$, then $Y$ is cut out by $\Feyn(\zeta_{\matN \setminus e}, \Delta_{\mathfrak{m}}^{E\minus\{e\}}, \xi_{\matM\setminus e})$; if $e$ is a coloop on $\matM$, then $Y$ is cut out by $\Feyn(\zeta_{\matN \setminus e}, \Delta_{\mathfrak{m}}^{E\minus\{e\}}, 0).$ (The first ``$\leq$'' is the assumption of \emph{Case 1}; the second ``$\leq$'' is Lemma \ref{lemma - basic intersection theory identity}; the ``$=$'' is Proposition \ref{prop - feynman integrand, jet scheme intersect with gamma basic fml} for either choice of $Y$.) Since $\matN$ is loopless and deletions preserve looplessness, $\dim \scrL_{m}(Y) = (m+1)(n-2)$ by Proposition \ref{prop - feynman integrand, equidimensional jets} or Proposition \ref{prop - easiest Feynman form, equidimensional jets}, depending on the choice of $Y$. So \eqref{eqn - feynman, irreducible jets, case 1} gives $\dim \scrL_{m}(X_{\Feyn}, \Var(\mathfrak{p})) \leq (m+1)(n-1) - 1$, verifying \eqref{eqn - feynman, irreducible jets, dimension bound for general min prime}.

    \vspace{2ex}

    \emph{Case 2}: Suppose that $H \in \calI_{\matN}$ is a proper handle on $\matN$ such that: $\matN\setminus H$ is connected; $H$ has no coloops on $\matM$; $\rank(\matN/ H) \geq 1$; $\mathfrak{p} \ni \Feyn_{/H}$. Furthermore, assume there exists a $h \in H$ such that $\Feyn(\zeta_{\matN\setminus H}^{h}, \Delta_{\massm}^{E\minus H}, \xi_{\matM \setminus H}^{h})$ satisfies \eqref{eqn - feynman, irreducible jets, dimension criterion}. Then \eqref{eqn - feynman, irreducible jets, dimension bound for general min prime} holds.

\vspace{1ex}

    \noindent\emph{Proof}: Select $h \in H$ and consider $\bsx^{H\minus\{h\}} \Feyn_{\setminus H}^{h} \in \KK[{E\minus\{h\}}]$. Let $\Gamma \subseteq \scrL_{m}(\KK^{E})$ be the codimension $m+1$ variety cut out by $\{D^{q} x_{h} \}_{0 \leq q \leq m}.$ Then
    \begin{align} \label{eqn - feynman, irreducible jets, case 2 part 1}
        \dim &\scrL_{m}(\KK^{E}, X_{\Feyn}, \Var(\mathfrak{p})) \\
        &\leq \dim \scrL_{m}(\KK^{E}, X_{\Feyn}, \Var(\Feyn_{/H})) \nonumber \\
        &\leq \dim \left( \scrL_{m}(\KK^{E}, X_{\Feyn}, \Var(\Feyn_{/H})) \cap \Gamma \right) + \codim \Gamma \nonumber \\
        &= \dim \left( \scrL_{m}(\KK^{{E\minus\{h\}}}, \Var(\bsx^{H\minus\{h\}} \Feyn_{\setminus H}^{h}), \Var(\Feyn(\zeta_{\matN/ H}, \Delta_{\massm}^{{E\minus\{h\}}}, \xi_{\matM / H}))) \right)\nonumber\\
        &\qquad+ \codim \Gamma, \nonumber
    \end{align}
    where the first ``$\leq$'' is part of the \emph{Case 2} assumption, the second ``$\leq$'' is the intersection theory identity Lemma \ref{lemma - basic intersection theory identity}; the ``$=$'' is \eqref{eqn - feynman integrand, jet scheme intersect with gamma, lying over} of Lemma \ref{lemma - feynman integrand, jet scheme intersect with gamma, lying over}. (The legality of Lemma \ref{lemma - basic intersection theory identity} is guaranteed by the non-emptiness asserted in Lemma \ref{lemma - feynman integrand, jet scheme intersect with gamma, lying over}.)

    Let $\mathfrak{q}$ be a minimal prime of $\scrL_{m}(\KK^{E\minus\{h\}}, \Var(\bsx^{H\minus\{h\}} \Feyn_{\setminus H}^{h}))$. Since $\codim \Gamma = m+1$, to show \eqref{eqn - feynman, irreducible jets, dimension bound for general min prime} it suffices by \eqref{eqn - feynman, irreducible jets, case 2 part 1} to show that
    \begin{equation} \label{eqn - feynman, irreducible jets, case 2 part 2}
        \dim \left( \Var(\mathfrak{q}) \cap (\pi_{m,0}^{\KK^{{E\minus\{h\}}}})^{-1} \big( \Var(\Feyn( \zeta_{\matN/ H}, \Delta_{\massm}^{{E\minus\{h\}}}, \zeta_{\matM / H})) \big) \right) \leq (m+1)(n-2) - 1.
    \end{equation}
    
    Apply Lemma \ref{lemma - feynman, intersecting with nabla} with $H^{\prime} := ({E\minus\{h\}}) \minus (E \minus H)= H\minus\{h\}$.  Then, firstly, we may find an ideal $J_{\mathbf{v}} \subseteq \KK[\{x_{f}^{(q)}\}_{\substack{ f \in E \minus \{e\} \\ 0 \leq q \leq m}}]$ where $\Var(\mathfrak{q}) \subseteq \Var(J_{\mathbf{v}})$ and so validating $\eqref{eqn - feynman, irreducible jets, case 2 part 2}$ reduces  to validating
    \begin{equation} \label{eqn - feynman, irreducible jets, case 2, part 3}
        \dim \bigg( \Var(J_{\mathbf{v}}) \cap (\pi_{m,0}^{\KK^{{E\minus\{h\}}}})^{-1} \big( \Var(\Feyn( \zeta_{\matN/ H}, \Delta_{\massm}^{{E\minus\{h\}}}, \zeta_{\matM / H})) \big) \bigg) \leq (m+1)(n-2) - 1.
    \end{equation}
    Secondly, we can find an integer $-1 \leq t \leq m$ and a complete intersection $\nabla_{\mathbf{v}} \subseteq \KK[ \{x_{f}^{(q)}\}_{\substack{ f \in E \minus \{e\} \\ 0 \leq q \leq m}}]$ of codimension $(m+1)(|H| - 1) - (m+1) - (t+1)$ such that
    \begin{align} \label{eqn - feynman, irreducible jets, case 2, part 4}
        \Var(J_{\mathbf{v}}) \cap \Var(\nabla_{\mathbf{v}}) 
        \simeq (\pi_{m,t}^{\KK^{E \minus H}})^{-1} \bigg( \scrL_{t}(\KK^{E \minus H}, \Var(\Feyn(\zeta_{\matN\setminus H}^{h}, \Delta_{\massm}^{E \minus H}, \xi_{\matM \setminus H}^{h}))) \bigg).
    \end{align}
    (The right hand side in \eqref{eqn - feynman, irreducible jets, case 2, part 4} is never empty since $\rank(\matN \setminus H) \geq 2$; when $t=-1$ we convened this term denotes $\scrL_{m}(\KK^{E\minus H})$.) 

    Now we finish out \emph{Case 2} by validating \eqref{eqn - feynman, irreducible jets, case 2, part 3} in two cases, based on the choice of $-1 \leq t \leq m$.

    \vspace{5mm}
    
    \emph{SubClaim 2.(a): If $t \geq 0$, then \eqref{eqn - feynman, irreducible jets, case 2, part 3} holds.} 
    
    In this setting, we deduce that
    \begin{align} \label{eqn - feynman, irreducible jets, case 2, part 5}
        \dim &\bigg( \Var(J_{\mathbf{v}}) \cap (\pi_{m,0}^{\KK^{{E\minus\{h\}}}})^{-1} \big( \Var(\Feyn( \zeta_{\matN/ H}, \Delta_{\massm}^{{E\minus\{h\}}}, \zeta_{\matM / H})) \big) \bigg) \\
        &\leq \dim \bigg( \Var(J_{\mathbf{v}}) \cap \Var(\nabla_{\mathbf{v}}) \cap (\pi_{m,0}^{\KK^{{E\minus\{h\}}}})^{-1} \big( \Var(\Feyn( \zeta_{\matN/ H}, \Delta_{\massm}^{{E\minus\{h\}}}, \zeta_{\matM / H})) \big) \bigg)\nonumber\\
        & \phantom{xxx}+\codim \Var(\nabla_{\mathbf{v}}) \nonumber \\
        &= \dim \left[ (\pi_{m,t}^{\KK^{E\minus H}})^{-1} \left[ \scrL_{t} \bigg( \KK^{E\minus H}, \Var(\Feyn(\zeta_{\matN\setminus H}^{h}, \Delta_{\massm}^{E\minus H}, \xi_{\matM \setminus H}^{h})), \Var(\Feyn(\zeta_{\matN/ H}, \Delta_{\massm^{E\minus H}}, \zeta_{\matM / H})) \bigg) \right] \right]   \nonumber \\
        &\phantom{xxx}+ \codim \Var(\nabla_{\mathbf{v}}) \nonumber \\
        &= \dim \left[ \scrL_{t} \bigg( \KK^{E\minus H}, \Var(\Feyn(\zeta_{\matN\setminus H}^{h}, \Delta_{\massm}^{E\minus H}, \xi_{\matM \setminus H}^{h})), \Var(\Feyn(\zeta_{\matN/ H}, \Delta_{\massm^{E\minus H}}, \zeta_{\matM / H})) \bigg) \right] \nonumber \\
        &\phantom{xxx}+ \codim \Var(\nabla_{\mathbf{v}}) + (m-t)(|E| - |H|) \nonumber
    \end{align}
    Here: the ``$\leq$ is Lemma \ref{lemma - basic intersection theory identity}; the first ``$=$'' is \eqref{eqn - feynman, irreducible jets, case 2, part 4} plus the fact that intersecting with $\Var(\nabla_{\mathbf{v}})$ has the effect of evaluating $\Feyn(\zeta_{\matN/ H}, \Delta_{\massm}^{{E\minus\{h\}}}, \xi_{\matM / H})$ at $x_{h^{\prime}}^{(q)} = 0$, for all $h^{\prime} \in H\minus\{h\}$ and all $0 \leq q \leq m$, since $J_{\mathbf{v}} + \nabla_{\mathbf{v}} \supseteq \{x_{h^{\prime}}^{(q)}\}_{\substack{h^{\prime} \in H\minus\{h\} \\ 0 \leq q \leq m}}$; the third ``$=$'' is the fact $\pi_{m,t}^{\KK^{E\minus H}}$ is a trivial projection with fibers $\KK^{(m-t)(|E| - |H|)}$. The $t$-jets after the last ``$=$'' are non-empty as $\rank(\matN \setminus H) > \rank(\matN/ H) \geq 1$ by Remark \ref{rmk - feynman integrand basic facts}.\ref{item-Feyn-nonempty}; hence, Lemma \ref{lemma - basic intersection theory identity} is used correctly.
    
    Since 
    \begin{align*}
        (m &+ 1)(|E|-2) - 1 - 
        \bigg(\codim \Var(\nabla_{\mathbf{v}}) + (m-t)(|E| - |H|) \bigg) \\
        &= (m+1)(|E|-2) - 1 - \bigg((m+1)(|H| - 1) - (m-t) + (m-t)(|E| - |H|) \bigg) \\
        &= (t+1)(|E| - |H| - 1) - 1, 
    \end{align*} 
    the bound \eqref{eqn - feynman, irreducible jets, case 2, part 5} shows that in order to confirm \eqref{eqn - feynman, irreducible jets, case 2, part 3}, it suffices to confirm that
    \begin{align} \label{eqn - feynman, irreducible jets, case 2, part 6}
        \dim & \scrL_{t} \bigg( \KK^{E\minus H}, \Var(\Feyn(\zeta_{\matN\setminus H}^{h}, \Delta_{\massm}^{E\minus H}, \xi_{\matM \setminus H}^{h})), \Var(\Feyn(\zeta_{\matN/ H}, \Delta_{\massm^{E\minus H}}, \zeta_{\matM / H})) \bigg) \\
        &\leq (t+1)(|E| - |H| - 1) - 1. \nonumber
    \end{align}
    By \emph{Case 2}'s assumptions, $\scrL_{t}(\KK^{E\minus H}, \Var(\zeta_{\matN\setminus H}^{h}, \Delta_{\massm}^{E\minus H}, \xi_{\matM \setminus H}^{h}))$ is irreducible of dimension $(t+1)(|E| - |H| - 1)$. By Lemma \ref{lemma - intersecting irreducible jet scheme with inverse image of avoidant variety downstairs},  \eqref{eqn - feynman, irreducible jets, case 2, part 6} would follow from 
    \begin{equation} \label{eqn - feynman, irreducible jets, case 2, part 7}
        \dim \bigg( \Var(\Feyn(\zeta_{\matN\setminus H}^{h}, \Delta_{\massm}^{E\minus H}, \xi_{\matM \setminus H}^{h})) \cap \Var(\Feyn(\zeta_{\matN/ H}, \Delta_{\massm^{E\minus H}}, \zeta_{\matM / H}) ) \bigg) \leq |E|-|H| - 2.
    \end{equation}
    
    Recall that $\matN\setminus H$ is connected  \emph{Case 2}. Consequently, Remark \ref{rmk - feynman integrand basic facts}.\ref{item-Feyn-irred} implies that $\Feyn(\zeta_{\matN\setminus H}^{h}, \Delta_{\massm}^{E\minus H}, \xi_{\matM \setminus H}^{h})$ is irreducible. So the only way \eqref{eqn - feynman, irreducible jets, case 2, part 7} could fail is if there were an $\alpha \in \KK[E\minus H]$ such that $\alpha \Feyn(\zeta_{\matN\setminus H}^{h}, \Delta_{\massm}^{E\minus H}, \xi_{\matM \setminus H}^{h}) = \Feyn(\zeta_{\matN/ H}, \Delta_{\massm}^{E\minus H}, \xi_{\matM / H})$. In particular, the leading term $\zeta_{\matN\setminus H}^{h}$ of the first Feynman integrand should divide the leading term $\zeta_{\matN/ H}$ of the second Feynman integrand. As $\zeta_{\matN\setminus H}^{h}$ is homogeneous of degree $|E| - |H| + 1$ and $\zeta_{\matN/ H}$ is homogeneous of degree $|E| - |H|$, \emph{cf.}\ Remark \ref{rmk - feynman integrand basic facts}.\eqref{item-Feyn-min} and Remark \ref{rmk-handle basics}, this is impossible. Hence \eqref{eqn - feynman, irreducible jets, case 2, part 7} is true, and \emph{Subclaim 2.(a)} follows, completing \emph{Case 2} under the choice $t \geq 0$.

    \vspace{2ex}

    \emph{SubClaim 2.(b): If $t = -1$, then \eqref{eqn - feynman, irreducible jets, case 2, part 3} holds.} 
    
    The argument the same as in \emph{SubClaim 2.(a)}, except that our convention for $t=-1$ breaks down for \eqref{eqn - feynman, irreducible jets, case 2, part 5}. Here $\codim \Var(\nabla_{\mathbf{v}}) = (m+1)(|H| - 2)$ and the correct version of \eqref{eqn - feynman, irreducible jets, case 2, part 5} is
    \begin{align} \label{eqn - feynman, irreducible jets, case 2, part 8}
        \dim &\bigg( \Var(J_{\mathbf{v}}) \cap (\pi_{m,0}^{\KK^{{E\minus\{h\}}}})^{-1} \big( \Var(\Feyn( \zeta_{\matN/ H}, \Delta_{\massm}^{{E\minus\{h\}}}, \zeta_{\matM / H})) \big) \bigg) \\
        &\leq \dim \bigg( \Var(\Feyn(\zeta_{\matN/ H}, \Delta_{\massm}^{E\minus H}, \xi_{\matM / H}))) \bigg) \nonumber \\
        &\phantom{xxx}+ (m+1)(|H| - 2) + m(|E| - H|). \nonumber
    \end{align}
    (The Feynman hypersurface after the ``$\leq$'' in \eqref{eqn - feynman, irreducible jets, case 2, part 8} lives in $\KK^{E\minus H}$.) Then \eqref{eqn - feynman, irreducible jets, case 2, part 3} will follow from
    \begin{equation} \label{eqn - feynman, irreducible jets, case 2, part 9}
        \Var(\Feyn(\zeta_{\matN/ H}, \Delta_{\massm}^{E\minus H}, \xi_{\matM / H})) \subseteq \KK^{E\minus H} \text{ has dimension at most } |E| - |H| - 1,
    \end{equation}
    whcih is immediate since 
    $\rank(\matN/ H) \geq 1$ by assumption. So \emph{SubClaim 2.(b)} holds, completing \emph{Case 2}  entirely.

    \vspace{2ex}

    \emph{The Inductive Argument}. 
    
    \vspace{1ex}
    
    We prove that $\Feyn$ satisfies \eqref{eqn - feynman, irreducible jets, dimension criterion} by induction on $|E|$. So fix a ground set $E$ and Feynman integrand $\Feyn = \Feyn(\zeta_{\matN}, \Delta_{\massm}^{E}, \xi_{\matM})$ where $\matN$ is connected and $\rank(\matN) \geq 2$. We first tackle the inductive step. So assume that the claim holds for all suitable Feynman integrands over ground sets strictly smaller than $E$. We must show that for any minimal prime $\mathfrak{p}$ of the true Jacobian $(\partial \Feyn) + (\Feyn)$ the inequality \eqref{eqn - feynman, irreducible jets, dimension bound for general min prime} holds.
    
    First assume that $\matN$ is a circuit. By Example \ref{example - feynman integrand, min primes jacobian, N a circuit}, there exists  $e \in E$ such that $\mathfrak{p} \ni x_{e}$. \emph{Case 1} applies verifying \eqref{eqn - feynman, irreducible jets, dimension bound for general min prime} and hence \eqref{eqn - feynman, irreducible jets, dimension criterion}. 

    Now assume that $\matN$ is not a circuit. By Proposition \ref{prop - special handle wrt matroid quotients} we may find a proper handle $H \in \calI_{\matM}$ such that: $\matN\setminus H$ is connected; $\rank(\matM \setminus H) > \rank(\matN \setminus H) \geq 2$; $\rank(\matN/ H) \geq 1$. Since, for any $h \in H$, every element of $(H\minus\{h\})$ is a coloop on $\matM$ and $\matN$, the fact that $\rank(\matM \setminus H) > \rank(\matN \setminus H)$ implies that $H$ contains no coloops on $\matM$. Pick $h \in H$; then $\Feyn( \zeta_{\matN\setminus H}^{h}, \Delta_{\massm}^{E\minus H}, \xi_{\matM \setminus H}^{h})$ falls into the inductive setup. By the preceding analysis we may invoke Proposition \ref{prop - feynman integrand, minimal primes of true jacobian}: either there is some $g \in H$ such that $\mathfrak{p} \ni x_{g}$, or $\mathfrak{p} \ni \Feyn_{/ H}$. If the former membership holds we can use \emph{Case 1}; if the latter membership holds we can use \emph{Case 2}. Either way we certify \eqref{eqn - feynman, irreducible jets, dimension bound for general min prime} and consequently \eqref{eqn - feynman, irreducible jets, dimension criterion} as well.

    So the inductive step holds. The base case occurs when $E$ is the smallest possible ground set admitting a connected matroid $\matN$ of rank $2$. This forces $|E| = 3$ and $\matN = \matU_{2,3}$, \emph{i.e.}\ $\matN$ is the $2$-circuit. But we already handled the case of a circuit via \emph{Case 1}, without any inductive setup. This completes the induction and also the proof.
    \end{proof}

    Rationality of $\Feyn$ follows from irreducibility of $\scrL_{m}(\KK^{E}, X_{\Feyn})$ as before, though we can weaken the rank hypothesis:

    \begin{corollary} \label{cor-Feyn-ratsing}
        Let $\Feyn = \Feyn(\zeta_{\matN}, \Delta_{\massm}^{E}, \xi_{\matM}) \in \KK[E]$ be a Feynman integrand. If $\matN$ is connected and $\rank(\matN) \geq 2$, then $\Feyn$ has rational singularities; if $\matN$ is connected and $\rank(\matN) = 1$, then $\Feyn$ has rational singularities or is smooth. 
    \end{corollary}

    \begin{proof}
        First assume $\rank (\matN) \geq 2$. Since $\matN$ is connected, $\Feyn$ is not smooth, see Remark \ref{rmk - feynman integrand basic facts}.\ref{item-Feyn-nonempty}. Now use Theorem \ref{thm - feynman integrand, irreducible jets} and argue as in Corollary \ref{cor-mtrdl-rat-sing}, \emph{mutatis mutandis}.

        Now assume that $\rank(\matN) = 1$. Then $\Feyn \in \mathbb{K}[E]$ is a quadric. Via a linear coordinate change, $\Feyn$ can be brought into one of the standard forms $1+\sum_{e'\in E'}x_{e'}^2$ or $x_e+\sum_{e'\in E'}x_{e'}^2$ or $\sum_{e'\in E'}x_{e'}^2$, where $E'\subseteq E$ and $e\notin E'$. In the first and second case, $\Feyn$ is smooth. In the third case, $|E'|\geq 3$ since, by Remark \ref{rmk - feynman integrand basic facts}.\ref{item-Feyn-irred}, $\Feyn$ is irreducible. It follows that in this third case $\Feyn$ has rational singularities since its Bernstein--Sato polynomial is $(s+1)(s+|E'|/2)$, compare Remark \ref{rmk-rat-sing-properties}. 
    \end{proof}

\subsection{Feynman Diagrams}\label{subsec-FInt}
%%%%%%%%%%%%%%%%%%%%%%%%%%%%%%%%%%%%%%%%%%%%%%%%%%%%%%%%%

We briefly discuss here the background for Feynman diagrams, and then show how our results on Feynman integrands imply rationality of singularities for Feynman diagrams. For more details, see \cite{BognerWeinzierl,HT,W-Feynman}

\medskip

Let $G$ be a graph with vertices $V:=V_G$ and edges $E:=E_G$. 
Denote by $\calT_G^i$ its set of $i$-forests, so $F\subseteq E$ is in $\calT_G^i$ precisely when it is circuit-free and the graph on the set of vertices of $G$ with the set of edges $F$ has exactly $(i-1)$ more connected components than $G$ does. In a connected graph, an $i$-forest has exactly $i$ connected components and a 1-forest is often called a \emph{spanning tree}.

We identify some of the vertices to be \emph{external vertices} and label them as $\VExt(G)$. To each external vertex $v \in \VExt(G)$ we associate a \emph{momentum} $p(v)$ where we require that 
\[
\sum_{v \in \VExt(G)} p(v) = 0,
\]
and we assume that  a \emph{mass function}
\[
\massm\colon E\to \KK
\]
has been chosen (where typically one assumes the range of the mass function to be $\RR_{\geq 0}$).  
One  incorporates the mass data into a linear polynomial:
\[
\Delta_{\massm}^{E} = \sum_{e \in E} (m_e)^2 x_{e} \in \KK[E]
\]
where $m_e=\massm(e)$.

The graph along with the data $G(V,E,\massm,p)$ of external vertices, momenta, and masses comprise a \emph{Feynman diagram}. 

\medskip

Feynman diagrams are a tool to model certain probabilities in scattering theory. The edges of a Feynman diagram represent particles that interact "at the vertices". External vertices possess also an "external leg", which represents an observable particle with momentum $p(v)$, "entering" or "exiting" the scattering process. For a chosen set of external vertices, the infinitely many possible Feynman diagrams that exhibit the chosen set as external represent the various ways in which the scattering may have proceeded. In the formulation of Lee--Pomeransky \cite{LeePomeransky}, (suitably normalized) integrals as displayed in \eqref{eqn-FeynmanIntegral} model the various probabilities according to which the system might behave; the relevant quantities are given as integrals over certain rational functions.

In Quantum Field Theory it turns out one can reduce the information stored in complicated Feynman diagram to ones in which the underlying graphic matroid is connected, without losing any information on the scattering.

\medskip

The key polynomials associated to a Feynman diagram $G$ are defined as follows and all live in $\KK[E] = \KK[\{x_{e}\}_{e \in E}]$. The \emph{first Symanzik polynomial} of $G$ is the matroid basis polynomial of the cographic matroid of $G$:
\[
\scrU :=\Psi_{(\matM_G)^\perp}=\sum_{T \in\calT^1_G} \bsx^{E\minus T} \in \KK[E].
\]

The \emph{second Symanzik polynomial} $\scrF_{0}^{W}$ of $G$ is given by enumerating the complements of all spanning $2$-forests of $G$ and then weighting each complement by a constant depending on our momentum data:
\[
\scrF_{0}^{W} := \sum_{F = T_{1} \sqcup T_{2} \in\calT^2_G} |p(T_{1})|^{2} \enspace {\bsx^{E\minus F}} \in \KK[E];
\]
here, $F = T_{1} \sqcup T_{2}$ is the decomposition of $F$ into its two tree components, and $p(T_{1}) = \sum_{v \in \VExt \cap T_{1}} p(v)$. (Since $p(T_{1}) + p(T_{2}) = \sum_{v \in \VExt(G)} p(v) = 0$, the term $p(T_{1})^{2}$ is independent of the choice of connected component of $F$.)

Quantum Field Theory investigates Feynman integrals when the momenta are vectors in Minkowski space $\RR^{1,3}$ with norm $|p|^2=p_0^2-(p_1^2+p_2^2+p_3^2)$. The superscript in $\scrF^W_0$ refers to \emph{Wick rotation}, multiplying momentum coordinates by $\sqrt{-1}$, and thus exchanging Minkowski norm by Euclidean norm, but also moving the entire situation out of the real into the complex domain. 

 A simplified "scalar" version arises when the range of $p$ is $\RR$. 
From the matroidal point of view these two cases behave similarly as long as the non-scalar formulation enjoys \emph{general kinematics}: no proper subsum of external vertices should have momentum sum zero, and no cancellation of terms should occur in the sum $\scrU\cdot\Delta^E_\massm +\scrF^W_0$.

\begin{prop} \label{prop-Patterson-W}
Suppose $\matM_G$ is connected and $|E|\geq 2$.
    \begin{asparaenum}
    \item If $p$ is scalar and not identically zero, there is a quotient $\matM_{G, p}$ of the graphic matroid $\matM_{G}$, with $\rank(\matM_{G,p})=\rank(\matM_G)-1$, such that the second Symanzik polynomial $\scrF_{0}^{W}$ is a configuration polynomial of the dual matroid $\matM_{G, p}^{\perp}$, \cite[Prop.~3.14]{Patterson}.
    \item If a vector-valued $p$ satisfies general kinematics then there is a matroid quotient $\matM_{G,p}$ of $\matM_G$ with  $\rank(\matM_{G,p})=\rank(\matM_G)-1$ such that $\scrF_0^W$ is a matroid support polynomial of $\matM_{G,p}$, \cite[Prop.~3.5]{W-Feynman}\qed
    \end{asparaenum}
    \end{prop}

\begin{define} \label{def - feynman diagram poly}
    For a Feynman diagram $G = (V,E,\massm,p)$, the associated \emph{Feynman diagram polynomial} $\mathscr{G}$ is
    \[
    \mathscr{G} := \mathscr{U} + \mathscr{U}\cdot \Delta_{\massm}^{E} + \scrF_{0}^{W} = \mathscr{U}\cdot(1 + \Delta_{\massm}^{E}) + \scrF_{0}^{W}.
    \]
\end{define}

In Quantum Field Theory, the following Mellin type integral is of crucial interest, where $D\in \RR_{>0}$, $\boldb\in\NN^E$:
\begin{eqnarray}
\label{eqn-FeynmanIntegral}
I_G(D,\boldb)&:=&\int\limits_{(\RR_{>0})^E}\frac{\prod_{e\in E}x_e^{b_e-1}}{\scrG^{D/2}}\,\mathrm{d}x_E.
\end{eqnarray}

Often, the Feynman diagram polynomial is a special instance of a Feynman integrand, see Definition \ref{def - feynman integrands}.
\begin{proposition} \label{prop-FDiagPoly=FInt}
    Let $\matM_G$ be connected, with $|E|\geq 2$. Suppose $p$ is scalar and not identically zero, or vector-valued satisfying general kinematics. Then the Feynman diagram polynomial $\mathscr{G}$ is a Feynman integrand. Specifically, 
    \[
    \mathscr{G} = \Feyn(\Psi_{\matM_{G}^{\perp}}, \Delta_{\massm}^{E}, \scrF_{0}^{W}).
    \]
\end{proposition}

\begin{proof}
    We know  that $\mathscr{U} = \Psi_{\matM_{G}^{\perp}}$ (straight from the definition) and $\scrF_{0}^{W}$ (by Proposition \ref{prop-Patterson-W}) are matroid support polynomials 
on (the connected matroid) $\matM_G^\perp$ and a matroid
    $\matM_{G, p}^\perp$ respectively, where $\matM_{G,p}$ is is a quotient of $\matM_{G}$ with $\rank(\matM_{G}) = \rank(\matM_{G, p}) + 1$. By Remark \ref{rmk-mat-quots}.\ref{item-dual-DC} and Definition \ref{dfn-dual-matroid}, $\matM_{G}^{\perp}$ is a quotient of $\matM_{G, p}^{\perp}$ with $\rank(\matM_{G, p}^{\perp}) =  \rank(\matM_{G}^{\perp}) + 1.$
\end{proof}

Our detailed study of Feynman integrands and their jet schemes applies to Feynman integrands: 

\begin{theorem} \label{thm-feyn-poly-ratsing}
    Let $G=G(V,E,\massm,p)$ be a Feynman diagram such that $\matM_G$ is connected.   Suppose
    \begin{itemize}
    \item $\massm,p$ are identically zero,  or
    \item $p$ is scalar and not identically zero, or
    \item $p$ is vector-valued and has general kinematics.
    \end{itemize}
    In all of the three cases above, the Feynman diagram polynomial $\mathscr{G}$ has rational singularities or is smooth (or, if $G=K_2$, is constant).
\end{theorem}

\begin{proof}
    If $G$ has fewer that two vertices, there is nothing to discuss. If $G$ is a single (coloop) edge, the unique 2-forest is empty, and $\scrG=1(1+m_1^2x_1)+\scrF^W_0$ is linear or constant.  

    If $\massm=0=p$ then $\scrG=\scrU$ is the matroid basis polynomial of $\matM_{G}^{\perp}$ and we are done by Corollary \ref{cor-mtrdl-rat-sing}.

    So we may assume both that $|E| \geq 2$ and that $\scrG$ is the Feynman integrand $\Feyn(\Psi_{\matM_{G}^{\perp}}, \Delta_{\massm}^{E}, \scrF_{0}^{W})$, thanks to Proposition \ref{prop-FDiagPoly=FInt}. Since $\matM_G$ is connected, $\rank(\matM_G) < |E|$ forcing $\rank(\matM_{G}^{\perp}) \geq 1$. By Corollary \ref{cor-Feyn-ratsing}, $\scrG$ has rational singularities or is smooth.

\end{proof}

\section{Future Directions}

Here we document some natural questions  on the hypersurfaces considered throughout the paper. 

\subsection{Moduli Space of Matroid Support Polynomials} \label{subsect-ModuliSpace}

Fix a connected matroid $\matM$ on $E$. A matroid support polynomial 
\[
\chi_\matM = \sum_{B\in\calB_\matM}c_B\bsx^B \in \mathbb{K}[E] \qquad c_B\neq 0\,\,\forall B\in\calB_\matM
\]
on $\matM$ is characterized by the monomial support condition. So matroid support polynomials on $\matM$ are in bijection with 
%the vectors $(c_B)_{B \in \calB_\matM} \in \mathbb{K}^{E} \cap T_E$, where $\mathbb{K}^{E} \cap T_E$ is denotes the standard open torus. 
points in the standard open torus $(\KK^\times)^{|\calB_{\matM}|}$. 

Now let $\mathscr{P}$ be a geometric property of a variety in $\mathbb{K}^E$. %For example, $\mathscr{P}$ could encode codimension or Cohen--Macaulayness of the singular locus. 
The moduli space of matroid support polynomials with property $\mathscr{P}$ is
\begin{align*}
\mathscr{C}(\matM, \mathscr{P}) 
    &= \{ (c_B)_{B \in \calB_\matM} \in (\KK^{\times)^{|\calB_{\matM}|}} \mid  \Var( \sum_{B \in \calB_\matM} c_B \bsx^B) \subseteq \KK^E \text{ has property } \mathscr{P} \}.
\end{align*}

\begin{problem}
    For a connected matroid $\matM$, discuss geometric properties of  $\mathscr{C}(\matM, \mathscr{P})$.
\end{problem}
The conditions we have in mind include being open or closed, empty or filling all of $(\KK^\times)^{|\calB_{\matM}|}$, or specific properties of the singular locus such as codimension or projective dimension.

If $\matM$ has positive rank and $\mathscr{P}$ encodes ``having rational singularities (or being smooth)'', Corollary \ref{cor-mtrdl-rat-sing} shows $\mathscr{C}(\matM, \mathscr{P}) = (\KK^\times)^{|\calB_{\matM}|}$. If $\matM$ is representable and connected of rank $2$ or more,  and if  $\mathscr{P}$ encodes ``the singular locus has codimension $3$'', then $\mathscr{C}(\matM, \mathscr{P})$ contains the configuration polynomials on $\matM$, but may not equal all of $(\KK^\times)^{|\calB_{\matM}|}$, \cite[Main Theorem, Ex.~5.2]{DSW}.  The situation is somewhat similar when $\scrP$ specifies Cohen--Macaulayness of the singular locus, \cite[Ex.~5.2, 5.3]{DSW}. 

\subsection{Matroidal Polynomials}

Matroid support polynomials are characterized either by the monomial support condition or by being matroidal polynomials with a particular choice of singleton data, \emph{cf.}\ Proposition \ref{prop-MSP-are-MtrdlM}. Matroidal polynomials are defined by Deletion-Contraction axioms. 

\begin{question}
    Let $\matM$ be a matroid on $E$. Consider the polynomials in $\KK[E]$
    \begin{equation*}
        \bigcup_{\sigma} \, \{\zeta_{\matN} \in \Matroidal(E, \sigma) \mid \matN = \matM \} 
    \end{equation*}
    where $\sigma$ ranges over all possible choices of singleton data. Does this class admit an explicit characterization, for example in terms of monomial support?
\end{question}

\subsection{Bernstein--Sato Polynomials}

The Bernstein--Sato polynomial $b_{f}(s) \in \CC[s]$ of $f \in \mathbb{K}[E]$ is the minimal, monic, $\CC[s]$-polynomial satisfying the functional equation 
\[
b_{f}(s) f^{s} = P(s) f^{s+1}
\]
where $P(s) \in A_{\KK^{E}}[s]$, $A_{\KK^{E}}[s]$ is the polynomial ring extension of the Weyl algebra $A_{\KK^{E}}$ over $\KK^E$, and differential operators $P(s)$ act on $f^{s+1}$ by formal application of the chain rule. (See the \cite{SaitoOnBFunction,Kollar,W-Dmodsurvey} for details). Since $b_f(-1)=0$ in every case, one may define the \emph{minimal exponent} of $f$ as 
\[
\tilde\alpha_f:=- \max \{ r\in\RR \mid r \text{ is a root of }b_{f}(s)/(s+1) \}.
\]
 Then $f$ having rational singularities is equivalent to $\tilde\alpha_f>1$, (\emph{cf.}\ Remark \ref{rmk-rat-sing-properties}).

\begin{question} \label{question-Bfunction}
    Fix a connected matroid $\matM$ of positive rank. Determine the minimal exponent and the Bernstein--Sato polynomial of the matroid basis polynomial $\psi_\matM$.
    \end{question}
One can also ask for a description for  arbitrary matroid support polynomials $\chi_\matM$, but even for configuration polynomials, the answer will have to involve some data from the configuration. The Bernstein--Sato polynomials of the two realizations of $\matU_{3,6}$ in \cite[Ex.~5.3]{DSW} have roots $-\{1,2,3/2\}$ (special) and $-\{1,3/2,5/3,7/3\}$ (generic). The stratification of $(\KK^\times)^E$ by Bernstein--Sato polynomial is known to be algebraic by \cite{Lyubeznik-bfu}.

%By Corollary \ref{cor-flag-mtrdl-ratsing}, we know any flag matroidal polynomial $\zeta_{\mathscr{M}}$ on a terminally connected, terminally strict flag matroid $\mathscr{M}$ of terminal rank at least two has rational singularities. Hence the minimal exponent of $\zeta_{\mathscr{M}}$ is greater than $1$. Such flag matroidal polynomials enumerate a large chunk of squarefree polynomials. So we ask:
%
%\begin{question}
%    Given a (singular) squarefree polynomial $f \in \mathbb{K}[E]$, what is the probability its minimal exponent is $\leq 1$? That is, what is the probability $f$ does not have rational singularities? \dan{actually I don't know any singular squarefree polynomials with minimal exponent $\leq 1$.}
%\end{question}

\subsection{Jet Scheme Data}

For simplicity, let here $\chi_\matM$ denote a matroid support polynomial on a connected matroid $\matM$ of positive rank. The architecture of this paper involved estimating, but never actually answering the following:

\begin{question}\label{question-DimJetsOverSing}
    What is the dimension of 
    \begin{equation*}
        \mathscr{L}_m(\KK^E, X_{\chi_M}, X_{\chi_M, \Sing})
    \end{equation*}
    for a fixed $m \in \mathbb{Z}_{\geq 0}$?
\end{question}

The $m$-jets of a hypersurface $X \subseteq \KK^E$ lying over the singular locus $X_{\Sing}$ are the crucial ingredient needed for understanding the \emph{$m$-contact locus} $\mathscr{X}_m(\KK^E, X, X_{\Sing})$. These contact loci (see \cite{ContactLociArcSpaces,BudurCohomologyContactLoci,BudurEmbeddedNash} for an introduction) are the $m$-jets of $\KK^E$ vanishing on $X_{\Sing}$ that have order of vanishing along $X$ exactly $m$:

\begin{equation} \label{eqn-mContactLocus}
    \mathscr{X}_m(\KK^E, X, X_{\Sing}) = \bigg[ \pi_{m, m-1}^{-1}(\mathscr{L}_{m-1}(\mathbb{K}^{E}, X, X_{\Sing}) \bigg]  \minus \bigg[ \mathscr{L}_m(\mathbb{K}^{E}, X, X_{\Sing}) \bigg].
\end{equation}
 A major open problem in birational geometry is the (non)embedded \emph{Nash problem}. Roughly, this demands a description of the valuations of $X \subseteq \KK^E$ (data wedded to logarithmic resolutions of singularities) that arise from $m$-contact loci (which are jet/arc space data).   Satisfactory accounts of the $m$-contact loci are sparse and typically require \emph{ad hoc}  methods; compare for example \cite[Thm.~1.20, Sec.~5]{BudurEmbeddedNash}. A significant advance would be to count the number of irreducible components of $\mathscr{X}_{m}(\KK^E, X_{\zeta_\matM}, X_{\zeta_{\matM}, \Sing})$.

%%%%
\iffalse
\begin{question} \label{question-mContactLocus}
    Let $\matM$ be a connected matroid of positive rank; let $\chi_\matM \in \KK[E]$ be a matroid support polynomial. For each $m \in \mathbb{Z}_{\geq 0}$,  what are the irreducible components of the $m$-contact locus $\mathscr{X}_{m}(\KK^E, X, X_{\Sing})$?
\end{question}
\fi%%%%%
%\uli{it feels to me that this question should better be in the previous paragraph. It seems a little isolated in the flow of things. Suggestion: "\ldots compare for example \cite[Thm.~1.20, Sec.~5]{BudurEmbeddedNash}. A significant advance would be to count the number of irreducible components of 
%$\mathscr{X}_{m}(\KK^E, X, X_{\Sing})$." and then delete the question. (Note the rephrasing. I am unclear on how knowing the components can be a first step to discussing their dimension.)}\dan{changed. Though I think the weirdness may come from the fact I cited the the wrong question initially in the subsequent paragraph. Knowing dim of m-contact loci equates to knowing m-log threshold. See change in sentence following comment}

Knowing the dimension of $\mathscr{X}_{m}(\KK^E, X, X_{\Sing})$ is desirable as this is equivalent to knowing the \emph{$m$-log canonical threshold} (\cite[Def.~1.17, Prop.~1.18]{BudurEmbeddedNash}. This quantity is defined in terms of certain types of log-discrepancies on $m$-separating log resolutions, but very few explicit computations are known; \emph{cf.}\ \cite[Sec.~5]{BudurEmbeddedNash}. While $m$-thresholds closer to zero indicate "worse" singularities, the $m$-log-canonical threshold of a  rational hypersurface singularity need not equal 1, in contrast to the classical log-canonical threshold. Thus, both the $m$-log canonical threshold and the minimal exponent are tools for introducing a hierarchy of rational-inclusive singularities.

\subsection{Feynman Integrals}

By Theorem \ref{thm-feyn-poly-ratsing}, we know under mild assumptions on the Feynman diagram that the denominator of the Feynman integral \eqref{eqn-FeynmanIntegral} is a real (or even rational) power of a polynomial with rational singularities. In general, consider the multivariate Mellin transformation of $g \in \RR[E]$ by
\[
\{M(g)\}(\textbf{s}) = \int\limits_{(\RR_>0)^E} \big( \prod_{e \in E}x_{e}^{s_{e}-1} \big) g\, \de x
\]
where $\textbf{s}$ denotes the tuple $(s_e)_{e \in E}$ and $\de x$ denotes $\prod_{e \in E} \de x_e$. 

\begin{question}
    For $g \in \RR[E]$, what are the implications of $g$ having rational singularities over the complex numbers on the multivariate Mellin transformation $\{M(g)\}(\textbf{s})$?
\end{question}

%%%%%%%%%%%%%%%%%%%%%%%%%%%%%%%%%%%%%%%%%%%%%%%%%%%%%%%%%%%%%%%%%%%%%%
\appendix
\section*{Acknowledgements}
We would like to thank Linquan Ma for helpful conversations about $F$-singularities. We also thank Graham Denham and Mathias Schulze for helpful comments on configuration polynomials and other matroid matters. The first author is grateful to Nero Budur for introducing him to jet schemes and explaining their importance.

\bibliographystyle{alpha}
\bibliography{refs}

\end{document}